\documentclass[11pt]{article}
\usepackage{amsmath}
\usepackage{amsfonts}
\usepackage{amssymb}
\usepackage{mathrsfs}
\usepackage{amsthm}
\usepackage{palatino}
\usepackage[margin=2.5cm, vmargin={1.5cm}]{geometry}

\usepackage{dcolumn}

\usepackage{times}
\usepackage{graphicx}
\usepackage{xcolor}

\usepackage{pstricks}
\usepackage{pst-plot}
\usepackage{pst-grad}

\usepackage{wrapfig}

\usepackage{bm}

\renewcommand\footnotemark{}

\begin{document}

\title{Integer continued fractions for complex numbers}

\author{
  Cormac ~O'Sullivan\footnote{{\it Date:} August 20, 2025.
\newline \indent \ \ \
  {\it 2020 Mathematics Subject Classification:}  11Y65, 11F06, 11H55
\newline \indent \ \ \
{\em Key words and phrases.}  Continued fractions, the modular group, binary quadratic forms, cutting sequences.
  \newline \indent \ \ \
Support for this project was provided by a PSC-CUNY Award, jointly funded by The Professional Staff Congress and The City
\newline \indent \ \ \
University of New York.}
  }

\date{}

\maketitle


\def\s#1#2{\langle \,#1 , #2 \,\rangle}

\def\H{{\mathbb{H}}}
\def\F{{\mathfrak F}}
\def\Fd{{\mathcal F}}
\def\Kd{{\mathcal K}}
\def\Gd{{\mathcal G}}
\def\C{{\mathbb C}}
\def\R{{\mathbb R}}
\def\Z{{\mathbb Z}}
\def\Q{{\mathbb Q}}
\def\N{{\mathbb N}}
\def\G{{\Gamma}}
\def\GH{{\G \backslash \H}}
\def\g{{\gamma}}
\def\L{{\Lambda}}
\def\ee{{\varepsilon}}
\def\K{{\mathcal K}}
\def\Re{\mathrm{Re}}
\def\Im{\mathrm{Im}}
\def\PSL{\mathrm{PSL}}
\def\SL{\mathrm{SL}}
\def\GL{\mathrm{GL}}
\def\Vol{\operatorname{Vol}}
\def\lqs{\leqslant}
\def\gqs{\geqslant}
\def\sgn{\operatorname{sgn}}
\def\res{\operatornamewithlimits{Res}}
\def\li{\operatorname{Li_2}}
\def\lip{\operatorname{Li}'_2}
\def\pl{\operatorname{Li}}

\def\ei{\mathrm{Ei}}

\def\clp{\operatorname{Cl}'_2}
\def\clpp{\operatorname{Cl}''_2}
\def\farey{\mathscr F}

\def\dm{{\mathcal A}}
\def\ov{{\overline{p}}}
\def\ja{{K}}

\def\nb{{\mathcal B}}
\def\cc{{\mathcal C}}
\def\nd{{\mathcal D}}

\def\u{{\text{\rm u}}}
\def\v{{\text{\rm v}}}
\def\ib{{\text{\,\rm i}}}
\def\jb{{\text{\,\rm j}}}
\def\qq{{f}}

\def\mg{{\mathcal M_{G}}}
\def\mz{{\mathcal M_{Z}}}
\def\stab{{\text{\rm Stab}}}
\def\aut{{\text{\rm Aut}}}
\def\ze{{z}}

\newcommand{\stira}[2]{{\genfrac{[}{]}{0pt}{}{#1}{#2}}}
\newcommand{\stirb}[2]{{\genfrac{\{}{\}}{0pt}{}{#1}{#2}}}
\newcommand{\eu}[2]{{\left\langle\!\! \genfrac{\langle}{\rangle}{0pt}{}{#1}{#2}\!\!\right\rangle}}
\newcommand{\eud}[2]{{\big\langle\! \genfrac{\langle}{\rangle}{0pt}{}{#1}{#2}\!\big\rangle}}
\newcommand{\norm}[1]{\left\lVert #1 \right\rVert}
\newcommand{\dx}[1]{\overset{*}{#1}}

\newcommand{\e}{\eqref}
\newcommand{\la}{\label}
\newcommand{\bo}[1]{O\left( #1 \right)}
\newcommand{\ol}[1]{\,\overline{\!{#1}}} 


\newtheorem{theorem}{Theorem}[section]
\newtheorem{lemma}[theorem]{Lemma}
\newtheorem{prop}[theorem]{Proposition}
\newtheorem{conj}[theorem]{Conjecture}
\newtheorem{cor}[theorem]{Corollary}
\newtheorem{assume}[theorem]{Assumptions}
\newtheorem{adef}[theorem]{Definition}

\newtheorem*{algo1}{Continued fraction algorithm for $\bm{z \in  \C}$}
\newtheorem*{algo2}{Continued fraction algorithm for $\bm{z \in  \C}$, Version 2}
\newtheorem*{algo3}{Continued fraction algorithm for $\bm{z \in  \C}$, Version 3}

\numberwithin{figure}{section}
\numberwithin{table}{section}


\newcounter{counrem}
\newtheorem{remark}[counrem]{Remark}

\renewcommand{\labelenumi}{(\roman{enumi})}
\newcommand{\spr}[2]{\sideset{}{_{#2}^{-1}}{\textstyle \prod}({#1})}
\newcommand{\spn}[2]{\sideset{}{_{#2}}{\textstyle \prod}({#1})}

\numberwithin{equation}{section}

\let\originalleft\left
\let\originalright\right
\renewcommand{\left}{\mathopen{}\mathclose\bgroup\originalleft}
\renewcommand{\right}{\aftergroup\egroup\originalright}

\bibliographystyle{alpha}

\begin{abstract}
We study a natural extension to complex numbers of the standard continued fractions. The basic algorithm is due to Lagrange and Gauss, though it seems to have gone mostly unnoticed as a way to create continued fractions. The new representations are shown to be  unique, and to have useful properties. They also admit a geometric cutting sequence interpretation.
\end{abstract}

\section{Introduction}
A standard continued fraction takes the  form
\begin{equation*}
  [ a_0,a_1, \dots, a_r ] := a_0+\frac{1}{a_1 + \displaystyle \frac{1}{\ddots + \displaystyle \frac 1{a_r}}}
\end{equation*}
which may also be infinite with  $r \to \infty$. This form is called simple, meaning that all the numerators are $1$. Every  real number can be represented in this way, with integers $a_j$ which are positive for $j\gqs 1$. In this paper we examine an extension to complex numbers that was found in  the work \cite{OStop} on binary quadratic forms.

\SpecialCoor
\psset{griddots=5,subgriddiv=0,gridlabels=0pt}
\psset{xunit=1.5cm, yunit=1.5cm, runit=1.5cm}
\psset{linewidth=1pt}
\psset{dotsize=3pt 0,dotstyle=*}
\begin{figure}[ht]
\centering
\begin{pspicture}(0.5,0)(9,4.2) 

\psset{arrowscale=1.4,arrowinset=0.3,arrowlength=1.1}
\newrgbcolor{light}{0.8 0.8 1.0}
\newrgbcolor{pale}{1 0.7 1}
\newrgbcolor{pale}{1 0.7 0.4}
\newrgbcolor{pale}{0.9179 0.7539 0.5781}
\newrgbcolor{pale}{0.996094, 0.71875, 0.617188}

\rput(6.7,2){%
        \begin{pspicture}(-1.2,-2)(2.2,2)

\pspolygon[linecolor=pale,fillstyle=solid,fillcolor=pale](-0.5,2)(-0.5,-2)(0.5,-2)(0.5,2)(-0.5,2)

\pswedge[linecolor=white,fillstyle=solid,fillcolor=white](-1,0){1}{-90}{90}
\pswedge[linecolor=white,fillstyle=solid,fillcolor=white](1,0){1}{90}{270}

\psarc[linecolor=pale](-1,0){1}{0}{60}
\psarc[linecolor=pale](1,0){1}{180}{240}

\psarc[linestyle=dashed](-1,0){1}{0}{60}
\psarc(1,0){1}{120}{180}
\psarc(-1,0){1}{-60}{0}
\psarc[linestyle=dashed](1,0){1}{180}{240}

\psline(-0.5,2)(-0.5,0.866025)
\psline[linestyle=dashed](0.5,2)(0.5,0.866025)
\psline[linestyle=dashed](-0.5,-2)(-0.5,-0.866025)
\psline(0.5,-2)(0.5,-0.866025)

\psarc[linecolor=light](0,0){1}{-60}{60}
\psarc[linecolor=light](0,0){1}{120}{240}

\psarc[linecolor=light](2,0){1}{120}{240}

\psarc[linecolor=light](1,0){1}{-120}{120}

\psline[linecolor=light](1.5,2)(1.5,0.866025)
\psline[linecolor=light](1.5,-2)(1.5,-0.866025)

\psline[linecolor=gray](-1.2,0)(2.2,0)
\psline[linecolor=gray](-0.5,0.04)(-0.5,-0.04)
\psline[linecolor=gray](0.5,0.04)(0.5,-0.04)
\psline[linecolor=gray](1.5,0.04)(1.5,-0.04)


\psdots(0.5, 0.866025)(-0.5, 0.866025)(0.5, -0.866025)(-0.5, -0.866025)
\rput(0.77,0.9){$_\rho$}
\rput(-0.77,0.9){$_{\rho^2}$}
\rput(0.77,-0.9){$_{\rho^5}$}
\rput(-0.77,-0.9){$_{\rho^4}$}

\rput(-1.25,1.55){$\mathcal F'$}
\rput(0.12,-0.12){$_0$}
\rput(1.12,-0.12){$_1$}
\end{pspicture}}

\rput(2,2){%
        \begin{pspicture}(-1.5,-0.5)(1.5,2.5)
\pspolygon[linecolor=pale,fillstyle=solid,fillcolor=pale](-0.5,2.5)(-0.5,0)(0.5,0)(0.5,2.5)(-0.5,2.5)

\pswedge[linecolor=white,fillstyle=solid,fillcolor=white](0,0){1}{0}{180}

\psarc[linecolor=pale](0,0){1}{60}{90}
\psarc[linestyle=dashed](0,0){1}{60}{90}
\psarc(0,0){1}{90}{120}

\psline(-0.5,2.5)(-0.5,0.866025)
\psline[linestyle=dashed](0.5,2.5)(0.5,0.866025)

\psarc[linecolor=light](0,0){1}{0}{60}
\psarc[linecolor=light](0,0){1}{120}{180}

\psline[linecolor=gray](-1.2,0)(1.2,0)
\psline[linecolor=gray](-0.5,0.08)(-0.5,-0.08)
\psline[linecolor=gray](0.5,0.08)(0.5,-0.08)

\rput(-0.5,-0.22){$_{-1/2}$}
\rput(0.5,-0.22){$_{1/2}$}

\rput(-1.2,1.7){$\mathcal F$}

\psdot(0,1)
\rput(0,0.8){$_i$}

\end{pspicture}}

\end{pspicture}
\caption{The regions $\mathcal F$ and $\mathcal F'$ with $\rho=e^{\pi i/3}$}
\label{funds}
\end{figure}
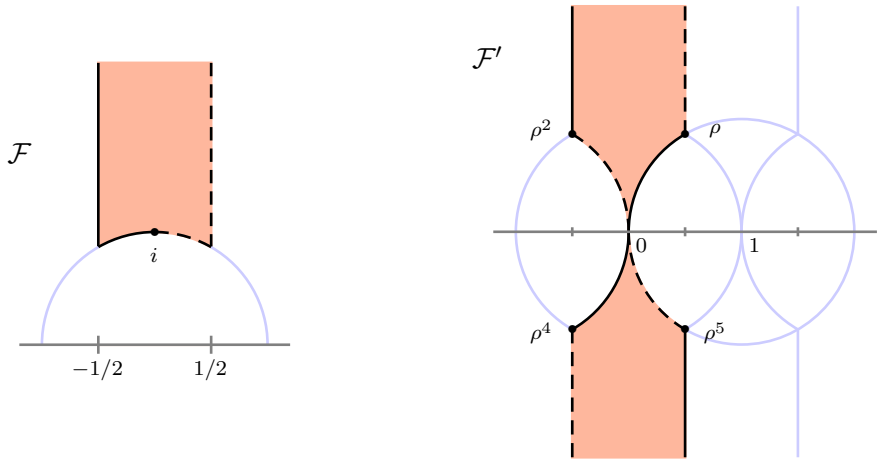

Our key definitions and main theorem are as follows. The group $\SL(2,\Z)$ of $2 \times 2$ integer matrices with determinant $1$  acts on the upper half plane $\H :=\{z\in \C \, : \, \Im(z)>0\}$ by linear fractional transformations.
A standard fundamental domain for this action is
\begin{equation}\label{fdb}
\Fd:= \big\{ z  \in \H \, : \,  -\tfrac 12\lqs \Re(z) < \tfrac 12, \ |z|\gqs 1 \text{ and } |z|=1 \implies \Re(z) \lqs 0\big\},
\end{equation}
 as discussed in Section \ref{hyf}. Let $\Fd'$ be the infinite hourglass shape seen in Figure \ref{funds} that consists of $0$ and the images of $\Fd$ under the maps $z \mapsto \pm z$ and  $z \mapsto \pm 1/z$.  
 
\begin{adef} \la{srf}
{\rm A {\em continued fraction representation} may take an infinite form $[a_0, a_1, a_2, \dots ]$, or a finite form $[a_0, a_1, \dots, a_{m-1}, a_m+z_0]$,  with integers $a_j$ that are positive after the first and with $z_0 \in \Fd'$. A {\em canonical continued fraction representation}
has the extra conditions in the finite case of  $z_0 \neq e^{\pm 2\pi i/3}$ and $a_m\gqs 2$ if $z_0=0$ and $m\gqs 1$.}
\end{adef}

\begin{theorem} \label{main}
Every $z\in \C$ has a canonical continued fraction representation as in Definition \ref{srf}. This representation is unique.
\end{theorem}

For a first example,
\begin{equation} \la{423}
  -\frac{101732}{28505}+\frac{i}{171030} = \left[-4, 2, 3, 7, 1, 5+\tfrac{1-3i}4 \right].
\end{equation}
We claim that these new continued fractions are  natural and useful,  based on their following properties:
\begin{itemize}
  \item For real numbers they are the usual continued fractions.
  \item The   algorithm to produce them is similar to the normal one,  enlarging the terminating set $\{0\}$.
   \item Each $z\in \C$ has a  unique  representation, with a sequence of positive integers after the first.
  \item These continued fractions directly produce a minimal basis for any two-dimensional lattice.
  \item For an integral binary quadratic form of any discriminant,  they may be used to find  equivalent reduced forms. The turning sequence leading to the reduced forms in the positive discriminant case is given by the usual continued fraction coefficients of a real zero of the form.  In the negative discriminant case, the turning sequence is given by the new continued fraction coefficients of a complex zero.
  \item These continued fractions can be used to give the unique representation of each element of $\PSL(2,\Z)$  in terms of generators of order $2$ and $3$.
  \item They allow a simple extension to $\C$ of Serret's theorem on equivalent numbers.
  \item They have an illustrative geometric interpretation that extends the cutting sequences of C. Series.
\end{itemize}
These features will all be covered in the following sections.
The algorithm that produces  these representations is described next. It is a version of the simple procedures used by Lagrange to reduce positive definite binary quadratic forms and by Gauss to reduce two-dimensional lattices. We will see that it is essentially the `Standard Gaussian Algorithm' of \cite{DFV}.

\begin{algo1}
Start with index $i=0$.
\begin{itemize}
  \item Let $m=\lfloor \Re(z)\rfloor$. If $z-m-\delta \in \Fd'$ for $\delta=0$ or $1$ then let $a_{i}=m+\delta$, $z_0=z-m-\delta$ and finish. (Choose $\delta=0$ if $\delta=0$, $1$ are both possible.) Otherwise  put $a_{i}=m$ and replace $z$ by $1/(z-m)$.
\end{itemize}
Increment $i$ and repeat. The output is $(a_0,a_1, \dots )$ along with $z_0$ if it terminates.
\end{algo1}

When $z\in \R$ then this reduces to the usual algorithm, as in \cite[Sect. 10.6]{hawr}.
The ambiguous case when $\delta=0$ and $1$ are both valid in the final step of this procedure can only happen  when $z-m = e^{\pm \pi i/3}$.  (It could also be avoided by removing the corners $e^{\pm 2\pi i/3}$ from $\Fd'$.) In this way, the output always corresponds to a canonical continued fraction representation.
We examine this algorithm in detail in the next sections. 

Continued fractions for real numbers have a long and rich history, as recounted in \cite{bre80} for example, with related research and applications being  currently pursued in many directions. Simple continued fractions for complex numbers  using coefficients $a_j$  in the ring $\Z+i \Z$ of Gaussian integers have been examined by many authors. This allows closer and closer approximations of complex numbers by Gaussian integer fractions, and the  approaches of Hurwitz and Schmidt are described in \cite[Chap. 5]{hen}. See also the work \cite{mar} in general imaginary quadratic fields.  Our continued fractions in Definition \ref{srf} cannot be directly used  to closely approximate nonreal complex numbers, though perhaps they may aid in the study of approximating reals and have wider uses.

\section{Hyperbolic fan geometry} \la{hyf}
The required operations of addition by an integer and taking a reciprocal correspond to the actions of 
\begin{equation} \la{tj}
 T :=\begin{pmatrix} 1 & 1 \\ 0 & 1 \end{pmatrix}, \quad  J := \begin{pmatrix} 0 & 1\\ 1 & 0 \end{pmatrix},
\end{equation}
so that $T^m x =x+m$ and $J x=1/x$.
Also
put
\begin{equation} \la{su}
  S := \begin{pmatrix} 0 & -1\\ 1 & 0 \end{pmatrix},  \quad  U := TS= \begin{pmatrix} 1&-1\\1&0 \end{pmatrix},
\end{equation}
and $\SL(2,\Z)$ is generated by any two of $S$, $T$ and  $U$. We have $S^2=-I=U^3$; see Section \ref{gen}.

The region $\Fd$ in \e{fdb} is
a fundamental domain,
meaning that for any $z\in \H$ the orbit $\SL(2,\Z) z$  intersects $\Fd$ exactly once.
The hourglass region
\begin{equation}\label{fdp}
\Fd':= \Fd \cup S\Fd \cup \{0\}\cup -(\Fd \cup S\Fd)
\end{equation}
from the introduction is mapped to itself under both $z \mapsto -z$ and $z \mapsto 1/z$ (for $z\neq 0$). It may be seen that
\begin{equation}\label{fdp2}
\text{closure}(\Fd') = \big\{ z  \in \C \, : \, -1/2\lqs \Re(z) \lqs 1/2, |z \pm 1 |\gqs 1 \big\}.
\end{equation}

The work in this section is based on \cite[Sect. 2]{OStop}, elaborating on its details.
The proof of Theorem \ref{genc} below uses the hyperbolic fans from \cite[pp. 28, 29]{con97}. They are not strictly necessary to show that every complex number has the desired expansion, but they provide more information as we will see in Corollary \ref{fgc}.
For relatively prime integers $p$ and $q$ with $q\gqs 0$, define the {\em fan} $F_{p/q}$ as follows:
\begin{equation}\label{fans}
F_{1/0}=F_{-1/0}:=\bigcup_{n \in \Z} T^n \Fd, \qquad F_{p/q}:= \begin{pmatrix} p & * \\ q & * \end{pmatrix} F_{1/0},
\end{equation}
where $(\begin{smallmatrix}p & * \\ q & *\end{smallmatrix})$ is completed to an element of $\SL(2,\Z)$ by B\'ezout's identity and $F_{p/q}$ is well-defined.  Each fan is inscribed by a Ford circle; this is the image of the horizontal line with imaginary part $1$ under the matrix in \e{fans}. A typical fan is shown in Figure \ref{fanf2}.
\SpecialCoor
\psset{griddots=5,subgriddiv=0,gridlabels=0pt}
\psset{xunit=4cm, yunit=4cm, runit=4cm}
\psset{linewidth=1pt}
\psset{dotsize=5pt 0,dotstyle=*}
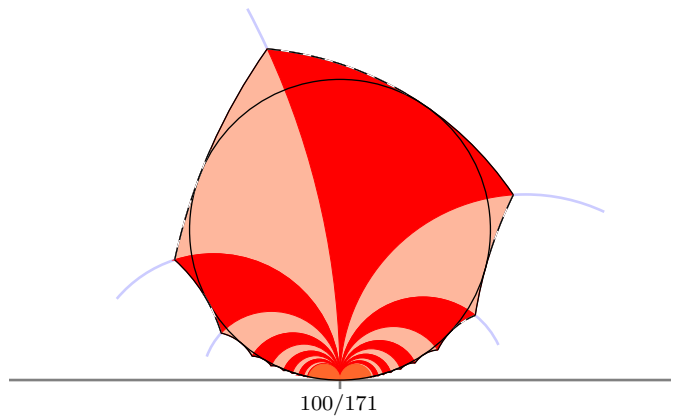
\begin{figure}[ht]
\centering
\begin{pspicture}(-1.1,-0.1)(1.1,1.3) 

\psset{arrowscale=1.1,arrowinset=0.1}
\newrgbcolor{light}{0.8 0.7 1.0}
\newrgbcolor{bmid}{0.6 0.5 1}
\newrgbcolor{bmid}{0.99609375 0.40234375 0.16796875}

\newrgbcolor{bb}{0.4 0.3 1}
\newrgbcolor{bb}{1 0 0}
\newrgbcolor{cx}{0.8 0.8 1.0}

\newrgbcolor{white2}{0.4 0.3 1}
\newrgbcolor{white2}{1 0 0}
\newrgbcolor{light}{0.99609375 0.71875 0.6171875}


\psarc[linecolor=cx](-2.63077,0){2.63077}{24.7559}{28}
\psarc[linecolor=cx](0.617329,0){0.617329}{65}{93.8332}

\psarc[linecolor=cx](0.276252,0){0.276252}{25}{51.1406}

\psarc[linecolor=cx](-0.420147,0){0.420147}{107.912}{140}
\psarc[linecolor=cx](-0.228304,0){0.228304}{136.849}{160}


\psline[linecolor=gray](-1.1,0)(1.1,0)
\psline[linecolor=gray](0.0,0.0)(0.0,-0.03)

\rput(0.0,-0.08){$_{100/171}$}

\psset{linewidth=0.1pt}

\pscustom[linecolor=white2,fillstyle=solid,fillcolor=bb]{
\psarc(-0.34288,0){1.10627}{33.8332}{84.7559}
\psarcn(-2.63077,0){2.63077}{24.7559}{0}
\psarcn(0.617329,0){0.617329}{180}{93.8332}
}

\pscustom[linecolor=white2,fillstyle=solid,fillcolor=light]{
\psarc(1.82966,0){1.39675}{153.833}{171.141}
\psarc(0.276252,0){0.276252}{51.1406}{180.}
\psarcn(0.617329,0){0.617329}{180.}{93.8332}
}

\pscustom[linecolor=white2,fillstyle=solid,fillcolor=bb]{
\psarc(0.532756,0){0.230636}{111.141}{154.259}

\psarc(0.17794,0){0.17794}{34.2586}{180.}

\psarcn(0.276252,0){0.276252}{180.}{51.1406}
}

\pscustom[linecolor=white2,fillstyle=solid,fillcolor=light]{
\psarc(0.332467,0){0.100445}{94.2586}{145.612}

\psarc(0.131236,0){0.131236}{25.6122}{180.}

\psarcn(0.17794,0){0.17794}{180.}{34.2586}
}

\pscustom[linecolor=white2,fillstyle=solid,fillcolor=bb]{
\psarc(0.245223,0){0.0568971}{85.6122}{140.413}

\psarc(0.103951,0){0.103951}{20.4134}{180.}

\psarcn(0.131236,0){0.131236}{180.}{25.6122}
}

\pscustom[linecolor=white2,fillstyle=solid,fillcolor=light]{
\psarc(0.195251,0){0.0367708}{80.4134}{136.956}

\psarc(0.0860594,0){0.0860594}{16.9561}{180.}

\psarcn(0.103951,0){0.103951}{180.}{20.4134}
}

\pscustom[linecolor=white2,fillstyle=solid,fillcolor=bb]{
\psarc(0.162563,0){0.025763}{76.9561}{134.495}

\psarc(0.0734221,0){0.0734221}{14.4949}{180.}

\psarcn(0.0860594,0){0.0860594}{180.}{16.9561}
}

\pscustom[linecolor=white2,fillstyle=solid,fillcolor=light]{
\psarc(0.139409,0){0.0190711}{74.4949}{132.655}

\psarc(0.064021,0){0.064021}{12.6551}{180.}

\psarcn(0.0734221,0){0.0734221}{180.}{14.4949}
}

\pscustom[linecolor=white2,fillstyle=solid,fillcolor=bb]{
\psarc(0.122106,0){0.014694}{72.6551}{131.228}

\psarc(0.0567541,0){0.0567541}{11.2284}{180.}

\psarcn(0.064021,0){0.064021}{180.}{12.6551}
}

\pscustom[linecolor=white2,fillstyle=solid,fillcolor=light]{
\psarc(0.108666,0){0.011672}{71.2284}{130.09}

\psarc(0.0509687,0){0.0509687}{10.09}{180.}

\psarcn(0.0567541,0){0.0567541}{180.}{11.2284}
}

\pscustom[linecolor=white2,fillstyle=solid,fillcolor=bb]{
\psarc(0.0979149,0){0.00949714}{70.09}{129.161}

\psarc(0.0462537,0){0.0462537}{9.1608}{180.}

\psarcn(0.0509687,0){0.0509687}{180.}{10.09}
}

\pscustom[linecolor=white2,fillstyle=solid,fillcolor=light]{
\psarc(1.31736,0){1.90906}{144.756}{167.912}

\psarcn(-0.420147,0){0.420147}{107.912}{0.}

\psarc(-2.63077,0){2.63077}{0.}{24.7559}
}

\pscustom[linecolor=white2,fillstyle=solid,fillcolor=bb]{
\psarc(-0.910446,0){0.538707}{16.8489}{47.912}

\psarcn(-0.420147,0){0.420147}{107.912}{0.}

\psarc(-0.228304,0){0.228304}{0.}{136.849}
}

\pscustom[linecolor=white2,fillstyle=solid,fillcolor=light]{
\psarc(-0.431347,0){0.160349}{29.6232}{76.8489}

\psarcn(-0.228304,0){0.228304}{136.849}{0.}

\psarc(-0.156737,0){0.156737}{0.}{149.623}
}

\pscustom[linecolor=white2,fillstyle=solid,fillcolor=bb]{
\psarc(-0.292478,0){0.0792611}{36.6444}{89.6232}

\psarcn(-0.156737,0){0.156737}{149.623}{0.}

\psarc(-0.11933,0){0.11933}{0.}{156.644}
}

\pscustom[linecolor=white2,fillstyle=solid,fillcolor=light]{
\psarc(-0.223372,0){0.0476266}{41.0535}{96.6444}

\psarcn(-0.11933,0){0.11933}{156.644}{0.}

\psarc(-0.096338,0){0.096338}{0.}{161.054}
}

\pscustom[linecolor=white2,fillstyle=solid,fillcolor=bb]{
\psarc(-0.181346,0){0.0318707}{44.0714}{101.054}

\psarcn(-0.096338,0){0.096338}{161.054}{0.}

\psarc(-0.0807747,0){0.0807747}{0.}{164.071}
}

\pscustom[linecolor=white2,fillstyle=solid,fillcolor=light]{
\psarc(-0.152892,0){0.0228535}{46.2639}{104.071}

\psarcn(-0.0807747,0){0.0807747}{164.071}{0.}

\psarc(-0.0695405,0){0.0695405}{0.}{166.264}
}

\pscustom[linecolor=white2,fillstyle=solid,fillcolor=bb]{
\psarc(-0.132275,0){0.0172008}{47.9278}{106.264}

\psarcn(-0.0695405,0){0.0695405}{166.264}{0.}

\psarc(-0.0610496,0){0.0610496}{0.}{167.928}
}

\pscustom[linecolor=white2,fillstyle=solid,fillcolor=light]{
\psarc(-0.116618,0){0.0134197}{49.2333}{107.928}

\psarcn(-0.0610496,0){0.0610496}{167.928}{0.}

\psarc(-0.0544066,0){0.0544066}{0.}{169.233}
}

\pscustom[linecolor=white2,fillstyle=solid,fillcolor=bb]{
\psarc(-0.104309,0){0.0107646}{50.2845}{109.233}

\psarcn(-0.0544066,0){0.0544066}{169.233}{0.}

\psarc(-0.0490674,0){0.0490674}{0.}{170.285}
}


\pscustom[linecolor=bmid,fillstyle=solid,fillcolor=bmid]{
\psarcn(-0.0526316,0){0.0526316}{169.583}{0}
\psarcn(0,0.5){0.5}{270}{257}
}

\pscustom[linecolor=bmid,fillstyle=solid,fillcolor=bmid]{
\psarc(0.0526316,0){0.0526316}{10.4174}{180}
\psarc(0,0.5){0.5}{270}{283}
}

\psset{linewidth=0.5pt}

\pscircle(0,0.5){0.5}

\psarc(-0.34288,0){1.10627}{33.8332}{84.7559}  

\psarc[linecolor=white](-0.34288,0){1.10627}{59.2945}{84.7559}
\psarc[linestyle=dashed](-0.34288,0){1.10627}{59.2945}{84.7559}

\psarc(1.82966,0){1.39675}{153.833}{171.141} 

\psarc[linecolor=white](1.82966,0){1.39675}{153.833}{162.487}
\psarc[linestyle=dashed](1.82966,0){1.39675}{153.833}{162.487}

\psarc(0.532756,0){0.230636}{111.141}{154.259} 

\psarc[linecolor=white](0.532756,0){0.230636}{111.141}{132.7}
\psarc[linestyle=dashed](0.532756,0){0.230636}{111.141}{132.7}

\psarc(1.31736,0){1.90906}{144.756}{167.912} 

\psarc[linecolor=white](1.31736,0){1.90906}{156.334}{167.912}
\psarc[linestyle=dashed](1.31736,0){1.90906}{156.334}{167.912}

\psarc(-0.910446,0){0.538707}{16.8489}{47.912} 

\psarc[linecolor=white](-0.910446,0){0.538707}{16.8489}{32.3804}
\psarc[linestyle=dashed](-0.910446,0){0.538707}{16.8489}{32.3804}

\psarc(-0.431347,0){0.160349}{29.6232}{76.8489} 
\psarc(-0.292478,0){0.0792611}{36.6444}{89.6232}
\psarc(-0.223372,0){0.0476266}{41.0535}{96.6444}
\psarc(-0.181346,0){0.0318707}{44.0714}{101.054}
\psarc(-0.152892,0){0.0228535}{46.2639}{104.071}
\psarc(-0.132275,0){0.0172008}{47.9278}{106.264}
\psarc(-0.116618,0){0.0134197}{49.2333}{107.928}

\psarc(0.332467,0){0.100445}{94.2586}{145.612} 
\psarc(0.245223,0){0.0568971}{85.6122}{140.413}
\psarc(0.195251,0){0.0367708}{80.4134}{136.956}
\psarc(0.162563,0){0.025763}{76.9561}{134.495}
\psarc(0.139409,0){0.0190711}{74.4949}{132.655}
\psarc(0.122106,0){0.014694}{72.6551}{131.228}
\psarc(0.108666,0){0.011672}{71.2284}{130.09}


\end{pspicture}
\caption{The fan $F_{100/171}$ with its inscribed Ford circle of diameter  $1/171^2$} 
\label{fanf2}
\end{figure}
These fans tile $\H$, as seen in Figure \ref{fant}, with all points in exactly one fan, except that points in the orbit of $i$ (middle of edge boundaries) are in two and those in the orbit of $e^{\pi i/3}$ (corners) are in three.

\SpecialCoor
\psset{griddots=5,subgriddiv=0,gridlabels=0pt}
\psset{xunit=4cm, yunit=4cm, runit=4cm}
\psset{linewidth=1pt}
\psset{dotsize=5pt 0,dotstyle=*}
\begin{figure}[ht]
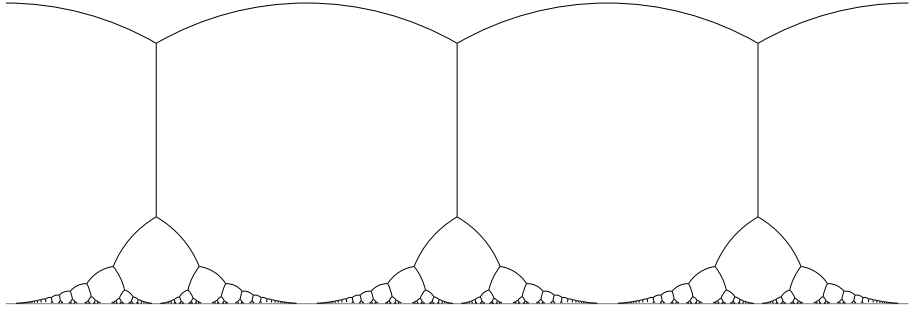

\centering

\caption{Tiling $\H$ with fans}
\label{fant}
\end{figure}

It is convenient to define the alternative fundamental domain
\begin{equation}\label{gfun}
\Gd:= \big\{ z  \in \H \, : \, -\tfrac 12 < \Re(z) \lqs \tfrac 12, |z|\gqs 1 \text{ and } |z|=1 \implies \Re(z) \gqs 0\big\},
\end{equation}
with boundary on the right. Then $-\Fd = \overline \Gd$, where  the bar indicates conjugation.
Hence $\Fd'$ has another expression as $\Fd \cup S\Fd \cup \{0\}\cup \overline \Gd \cup S\overline \Gd$.
The similar fans
$$
G_{1/0}=G_{-1/0}:=\bigcup_{n \in \Z} T^n \Gd, \qquad G_{p/q}:= \begin{pmatrix} p & * \\ q & * \end{pmatrix} G_{1/0},
$$
may also be defined
where $F_{p/q}$ and $G_{p/q}$ only differ by having complementary boundaries.

\begin{lemma} \la{flip}
We have
\begin{equation*}
  z \in F_{p/q} \iff 1/z  \in \overline G_{q/p}.
\end{equation*}
\end{lemma}
\begin{proof}
Note that
\begin{equation} \la{jm}
  J \begin{pmatrix} p & * \\ q & * \end{pmatrix} = \begin{pmatrix} q & * \\ p & * \end{pmatrix}(-I)J S
\end{equation}
with $(\begin{smallmatrix}p & * \\ q & *\end{smallmatrix})$ and $(\begin{smallmatrix}q & * \\ p & *\end{smallmatrix})$ in $\SL(2,\Z)$. Also
\begin{equation} \la{jst}
  J S T^n \Fd = - T^n \Fd = T^{-n} \overline\Gd.
\end{equation}
Therefore, with the matrices from \e{jm},
\begin{equation*}
   J \begin{pmatrix} p & * \\ q & * \end{pmatrix} T^n \Fd = \begin{pmatrix} q & * \\ p & * \end{pmatrix} J S  T^n \Fd
   = \begin{pmatrix} q & * \\ p & * \end{pmatrix} T^{-n} \overline\Gd
\end{equation*}
and the lemma follows.
\end{proof}

\SpecialCoor
\psset{griddots=5,subgriddiv=0,gridlabels=0pt}
\psset{xunit=4cm, yunit=4cm, runit=4cm}
\psset{linewidth=1pt}
\psset{dotsize=5pt 0,dotstyle=*}
\begin{figure}[ht]
\centering
\begin{pspicture}(-1.1,-0.1)(1.1,1) 

\psset{arrowscale=1.1,arrowinset=0.1}
\newrgbcolor{light}{0.8 0.7 1.0}
\newrgbcolor{bmid}{0.6 0.5 1}
\newrgbcolor{bmid}{0.99609375 0.40234375 0.16796875}

\newrgbcolor{bb}{0.4 0.3 1}
\newrgbcolor{bb}{1 0 0}
\newrgbcolor{cx}{0.8 0.8 1.0}

\newrgbcolor{white2}{0.4 0.3 1}
\newrgbcolor{white2}{1 0 0}
\newrgbcolor{light}{0.99609375 0.71875 0.6171875}


\psarc[linecolor=cx](-1,0){1}{60}{90}
\psarc[linecolor=cx](1,0){1}{90}{120}

\psarc[linecolor=cx](-0.3333,0){0.333}{120}{158.213}
\psarc[linecolor=cx](0.3333,0){0.333}{21.7868}{60}

\psarc[linecolor=cx](0.625,0){0.125}{32.2042}{81.7868}  
\psarc[linecolor=cx](-0.625,0){0.125}{98.2132}{147.796}

\psarcn[linecolor=cx](0.733333,0){0.0666667}{92.2042}{38.2132}  
\psarc[linecolor=cx](-0.733333,0){0.0666667}{87.7958}{141.787}


\psarc[linecolor=cx](0.2,0){0.2}{13}{38.2132}
\psarc[linecolor=cx](-0.2,0){0.2}{141.787}{167}
\psarc[linecolor=cx](-0.8,0){0.2}{13}{38.2132}
\psarc[linecolor=cx](0.8,0){0.2}{141.787}{167}


\psline[linecolor=gray](-1.1,0)(1.1,0)
\psline[linecolor=gray](-0.5,0.03)(-0.5,-0.03)
\psline[linecolor=gray](0.5,0.03)(0.5,-0.03)

\rput(-0.5,-0.08){$_{k-1/2}$}
\rput(0.5,-0.08){$_{k+1/2}$}

\psset{linewidth=0.1pt}

\pscustom[linecolor=white2,fillstyle=solid,fillcolor=bb]{
\psarc(0,0){1}{60}{120}
\psarcn(-1,0){1}{60}{0}
\psarcn(1,0){1}{180}{120}
}

\pscustom[linecolor=white2,fillstyle=solid,fillcolor=light]{
\psline(-0.5,0.2887)(-0.5,0.866)
\psarcn(-1,0){1}{60}{0}
\psarc(-0.333,0){0.333}{0}{120}
}

\pscustom[linecolor=white2,fillstyle=solid,fillcolor=light]{
\psline(0.5,0.2887)(0.5,0.866)
\psarc(1,0){1}{120}{180}
\psarcn(0.333,0){0.333}{180}{60}
}

\pscustom[linecolor=white2,fillstyle=solid,fillcolor=bb]{
\psarc(-0.333,0){0.333}{0}{120}
\psarcn(-0.6667,0){0.333}{60}{21.7868}
\psarcn(-0.2,0){0.2}{141.787}{0}
}

\pscustom[linecolor=white2,fillstyle=solid,fillcolor=bb]{
\psarcn(0.333,0){0.333}{180}{60}
\psarc(0.6667,0){0.333}{120}{158.213}
\psarc(0.2,0){0.2}{38.2132}{180}
}

\pscustom[linecolor=white2,fillstyle=solid,fillcolor=light]{
\psarc(-0.2,0){0.2}{0}{141.787}
\psarcn(-0.375,0){0.125}{81.7868}{32.2042}
\psarcn(-0.142857,0){0.142857}{152.204}{0}
}

\pscustom[linecolor=white2,fillstyle=solid,fillcolor=light]{
\psarcn(0.2,0){0.2}{180}{38.2132}
\psarc(0.375,0){0.125}{98.2132}{147.796}
\psarc(0.142857,0){0.142857}{27.7958}{180}
}

\pscustom[linecolor=white2,fillstyle=solid,fillcolor=bb]{
\psarc(-0.142857,0){0.142857}{0}{152.204}
\psarcn(-0.266667,0){0.0666667}{92.2042}{38.2132}
\psarcn(-0.111111,0){0.111111}{158.213}{0}
}

\pscustom[linecolor=white2,fillstyle=solid,fillcolor=bb]{
\psarcn(0.142857,0){0.142857}{180}{27.7958}
\psarc(0.266667,0){0.0666667}{87.7958}{141.787}
\psarc(0.111111,0){0.111111}{21.7868}{180}
}

\pscustom[linecolor=white2,fillstyle=solid,fillcolor=light]{
\psarc(-0.111111,0){0.111111}{0}{158.213}
\psarcn(-0.208333,0){0.0416667}{98.2132}{42.1034}
\psarcn(-0.0909091,0){0.0909091}{162.103}{0}
}

\pscustom[linecolor=white2,fillstyle=solid,fillcolor=light]{
\psarcn(0.111111,0){0.111111}{180}{21.7868}
\psarc(0.208333,0){0.0416667}{81.7868}{137.897}
\psarc(0.0909091,0){0.0909091}{17.8966}{180}
}

\pscustom[linecolor=white2,fillstyle=solid,fillcolor=bb]{
\psarc(-0.0909091,0){0.0909091}{0}{162.103}
\psarcn(-0.171429,0){0.0285714}{102.103}{44.8218}
\psarcn(-0.0769231,0){0.0769231}{164.822}{0}
}

\pscustom[linecolor=white2,fillstyle=solid,fillcolor=bb]{
\psarcn(0.0909091,0){0.0909091}{180}{17.8966}
\psarc(0.171429,0){0.0285714}{77.8966}{135.178}
\psarc(0.0769231,0){0.0769231}{15.1782}{180}
}

\pscustom[linecolor=white2,fillstyle=solid,fillcolor=light]{
\psarc(-0.0769231,0){0.0769231}{0}{164.822}
\psarcn(-0.145833,0){0.0208333}{104.822}{46.8264}
\psarcn(-0.0666667,0){0.0666667}{166.826}{0}
}

\pscustom[linecolor=white2,fillstyle=solid,fillcolor=light]{
\psarcn(0.0769231,0){0.0769231}{180}{15.1782}
\psarc(0.145833,0){0.0208333}{75.1782}{133.174}
\psarc(0.0666667,0){0.0666667}{13.1736}{180}
}

\pscustom[linecolor=white2,fillstyle=solid,fillcolor=bb]{
\psarc(-0.0666667,0){0.0666667}{0}{166.826}
\psarcn(-0.126984,0){0.015873}{106.826}{48.3649}
\psarcn(-0.0588235,0){0.0588235}{168.365}{0}
}

\pscustom[linecolor=white2,fillstyle=solid,fillcolor=bb]{
\psarcn(0.0666667,0){0.0666667}{180}{13.1736}
\psarc(0.126984,0){0.015873}{73.1736}{131.635}
\psarc(0.0588235,0){0.0588235}{11.6351}{180}
}

\pscustom[linecolor=white2,fillstyle=solid,fillcolor=light]{
\psarc(-0.0588235,0){0.0588235}{0}{168.365}
\psarcn(-0.1125,0){0.0125}{108.365}{49.5826}
\psarcn(-0.0526316,0){0.0526316}{169.583}{0}
}

\pscustom[linecolor=white2,fillstyle=solid,fillcolor=light]{
\psarcn(0.0588235,0){0.0588235}{180}{11.6351}
\psarc(0.1125,0){0.0125}{71.6351}{130.417}
\psarc(0.0526316,0){0.0526316}{10.4174}{180}
}

\pscustom[linecolor=bmid,fillstyle=solid,fillcolor=bmid]{
\psarcn(-0.0526316,0){0.0526316}{169.583}{0}
\psarcn(0,0.5){0.5}{270}{257}
}

\pscustom[linecolor=bmid,fillstyle=solid,fillcolor=bmid]{
\psarc(0.0526316,0){0.0526316}{10.4174}{180}
\psarc(0,0.5){0.5}{270}{283}
}

\psset{linewidth=0.5pt}

\pscircle(0,0.5){0.5}

\psarc(0,0){1}{60}{90} 
\psarc[linecolor=white](0,0){1}{90}{120}
\psarc[linestyle=dashed,](0,0){1}{90}{120} 

\psline[linecolor=white](-0.5,0.2887)(-0.5,0.5)
\psline[linestyle=dashed](-0.5,0.2887)(-0.5,0.5) 
\psline(-0.5,0.57735)(-0.5,0.866)
\psline(0.5,0.2887)(0.5,0.57735)
\psline[linecolor=white](0.5,0.5)(0.5,0.866)
\psline[linestyle=dashed](0.5,0.5)(0.5,0.866)

\psarc[linecolor=white](-0.6667,0){0.333}{21.7868}{40.8934}
\psarc[linestyle=dashed,dash=2pt 2pt](-0.6667,0){0.333}{21.7868}{40.8934}  
\psarc(-0.6667,0){0.333}{40.8934}{60}
\psarc[linecolor=white](0.6667,0){0.333}{120}{139.106}
\psarc[linestyle=dashed,dash=2pt 2pt](0.6667,0){0.333}{120}{139.106}
\psarc(0.6667,0){0.333}{139.106}{158.213}

\psarc(-0.375,0){0.125}{32.2042}{81.7868}  
\psarc(0.375,0){0.125}{98.2132}{147.796}

\psarcn(-0.266667,0){0.0666667}{92.2042}{38.2132}  
\psarc(0.266667,0){0.0666667}{87.7958}{141.787}

\psarcn(-0.208333,0){0.0416667}{98.2132}{42.1034}  
\psarc(0.208333,0){0.0416667}{81.7868}{137.897}

\psarcn(-0.171429,0){0.0285714}{102.103}{44.8218}  
\psarc(0.171429,0){0.0285714}{77.8966}{135.178}

\psarcn(-0.145833,0){0.0208333}{104.822}{46.8264}  
\psarc(0.145833,0){0.0208333}{75.1782}{133.174}

\psarcn(-0.126984,0){0.015873}{106.826}{48.3649}   
\psarc(0.126984,0){0.015873}{73.1736}{131.635}

\psarcn(-0.1125,0){0.0125}{108.365}{49.5826}   
\psarc(0.1125,0){0.0125}{71.6351}{130.417}

\psarc(0,0.5){0.5}{257}{283}

\rput(0,0.75){$n=0$}
\rput(-0.3,0.48){$_{n=1}$}
\rput(0.3,0.48){$_{n=-1}$}
\rput(-0.3,0.25){$_{2}$}
\rput(0.3,0.25){$_{-2}$}

\end{pspicture}
\caption{The fan $F_{k/1}$ with its inscribed Ford circle of diameter $1$}
\label{fanf}
\end{figure}
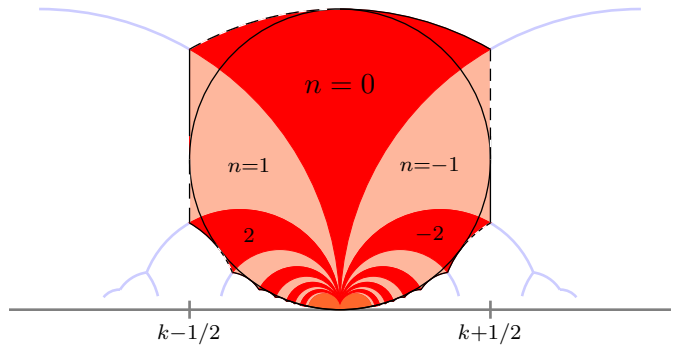

\begin{theorem} \label{genc}
The  continued fraction algorithm for $z\in \C$ produces the usual infinite continued fraction expansion if $z$ is an irrational real. Otherwise, the algorithm terminates, producing the integers $(a_0, a_1, \dots, a_r)$ with  $a_j$ positive for $j \gqs 1$, and $z_0 \in \Fd'$. Then 
\begin{equation*}
  z =  [ a_0,a_1, \dots, a_{r-1}, a_r+ z_0 ],
\end{equation*}
and this is the standard continued fraction when $z$ is rational.
\end{theorem}
\begin{proof} 
As already noted, the algorithm reduces to the usual one  when $z\in \R$, giving the familiar finite and infinite representations; see \cite[Thms. 161, 170]{hawr}. For $z$ outside of $\R$ we next consider the different parts of the fan tiling in Figure \ref{fant}. 

If $z\in k+\Fd'$ for $k\in \Z$ then the algorithm almost always produces $z=[k+z_0]$. The exceptions have $z=k+e^{\pm \pi i/3}$ so that, with our convention in the ambiguous cases, the algorithm gives $z=[k+e^{\pm \pi i/3}]$.
Next suppose that $z \in F_{k/1}$.  We call $T^k S T^n \Fd$ the $n$th part of this fan, as shown in Figure \ref{fanf}, and denote the corners by
\begin{equation*}
  \alpha_{n,k} := T^k S T^n e^{\pi i/3} = k- 1/(n+e^{\pi i/3}). 
\end{equation*}
For $z$ in the $n$th part of $F_{k/1}$ and not  a corner, we claim that the algorithm gives
\begin{subequations} \la{kx}
\begin{alignat}{2}
  z & = [ k-1,1,n-1+z_0 ] \quad &&\text{for \quad $n\gqs 2$}, \la{kx1}\\
  z & = [ k-1,1+z_0 ] \quad &&\text{for  \quad $n=1$},  \la{kx2}\\
  z & = [ k+z_0 ] \quad &&\text{for  \quad $n=0$},  \la{kx3}\\
  z & = [ k,|n|+z_0 ] \quad &&\text{for  \quad $n\lqs -1$},  \la{kx4}
\end{alignat}
\end{subequations}
with some  $z_0 \in \Fd'-\{0\}$ in each case. This is clear for $n=0$, as already seen. If $n\lqs -1$ then $m=k$ and $z-m \in S T^n \Fd$ gets sent to $1/(z-m)\in J S T^n \Fd =  T^{-n} \overline\Gd$ by \e{jst}. The relation \e{kx4} now follows.
 If $n\gqs 1$ then $m=k-1$ and $z-m \in T S T^n \Fd$ gets sent to $1/(z-m) \in J T S T^n \Fd$. We have
\begin{equation*}
  (J T) S T^n \Fd = (T S T J) S T^n \Fd = T S T (J S T^n \Fd) = T S T^{1-n} \overline\Gd
\end{equation*}
by \e{jst}. If $n=1$ we then obtain \e{kx2} as $S \overline\Gd \subseteq \Fd'$. Otherwise, for $n \gqs 2$ the algorithm applies again to    $z' \in T S T^{1-n} \overline\Gd$ which is  the $(1-n)$th part of  $\overline G_{1/1}$. Now $m=1$ and $1/(z'-m)$ is in
\begin{equation*}
  J S T^{1-n} \overline\Gd = - T^{1-n} \overline\Gd = T^{n-1} \Fd
\end{equation*}
giving \e{kx1}. The claim is proven.

The corner $\alpha_{n-1,k}$ is in the $n$th part of $F_{k/1}$ but only agrees with \e{kx} above for $n\lqs 0$, because of our ambiguous case convention. Explicitly, the algorithm produces
\begin{equation} \la{kc}
\begin{aligned}
  \alpha_{n,k} & = [ k-1,1,n-1+e^{\pi i/3} ] \quad &&\text{for \quad $n\gqs 2$},  \\
   \alpha_{n,k} & = [ k-1,1 +e^{-\pi i/3} ] \quad &&\text{for \quad $n= 1$},  \\
   \alpha_{n,k} & = [ k-1+e^{\pi i/3} ] \quad &&\text{for  \quad $n=0$},  \\
   \alpha_{n,k} & = [ k+e^{\pi i/3} ] \quad &&\text{for  \quad $n=-1$},  \\
   \alpha_{n,k} & = [ k,|n|-1+e^{-\pi i/3} ] \quad &&\text{for  \quad $n\lqs -2$}.
\end{aligned}
\end{equation}

Next suppose that $z$ is below $\R$,  in the $n$th part  $T^k S T^n \overline\Gd$ of $\overline G_{k/1}$ and not a corner.
The above arguments, with $\Fd$ and $\overline\Gd$ swapped, demonstrate that 
\e{kx} holds for these $z$ values too. The corners of $\overline G_{k/1}$  are  the conjugates of $\alpha_{n,k}$ and applying the algorithm to these produces \e{kc} with both sides conjugated. Our work  has verified the theorem for  $z$ in all the regions covered so far -- the expansion coefficients are integers that are positive  after the first.

Suppose the algorithm is applied to an initial $z$ in a fan $F_{p/q}$ with $q\gqs 2$, such as the one in Figure \ref{fanf2}. Then $m=\lfloor \Re(z)\rfloor$ also equals $\lfloor p/q\rfloor$ so that $z-m \in F_{p'/q}$ for $0<p'<q$. We have $z'=1/(z-m) \in \overline G_{q/p'}$ by Lemma \ref{flip} for $q/p'>1$. Similarly, with initial $z$ in $\overline G_{p/q}$ we obtain $z' \in F_{q/p'}$. Thus $z=m+1/z'$ and while the points remain in
$F_{a/b}$, $\overline G_{a/b}$ for $b\gqs 2$, the algorithm cannot terminate since $\Fd'-\{0\} \subseteq F_{1/0}\cup F_{0/1} \cup \overline G_{0/1} \cup  \overline G_{1/0}$. Following the usual continued fraction of $p/q$, we eventually arrive at
\begin{equation*}
  z= [ a_0,a_1, \dots, a_{v}, z' ],
\end{equation*}
for some $z'$ in $F_{k/1}$ or $\overline G_{k/1}$ with $k\gqs 2$. The final expansion coefficients are now given by  \e{kx} or by \e{kc} and their conjugates if $z$ is a corner. These last coefficients are positive since $k\gqs 2$.
\end{proof}

From the details of the proof we also obtain:
\begin{cor} \la{fgc}
Suppose $z \in F_{p/q}$ or $\overline G_{p/q}$ for $q\gqs 1$. Let $p/q=[a_0,a_1, \dots, a_m]$ 
be canonical. Then the canonical continued fraction for $z$ must equal of one of these:
\begin{equation} \la{pqx}
\begin{aligned}
  & [a_0, \dots, a_{m-1}, a_m-1,  1, a_{m+2}+z_0],\\
  & [a_0, \dots, a_{m-1}, a_m-1,  1+z_0],\\
  & [a_0, \dots, a_{m-1}, a_m-1+z_0],  \\
  & [a_0, \dots, a_{m-1}, a_m+z_0],  \\
  & [a_0, \dots, a_{m-1}, a_m, \, a_{m+1}+z_0].
\end{aligned}
\end{equation}
\end{cor}

So the first $m$ terms of the continued fractions of $p/q$ and $z$ agree. Ignoring $z_0$, the next terms differ by at most $1$, and the  continued fraction of  $z$ has up to two further  terms.

Note that if $z$ is in the middle of a boundary edge then it is in two different fans. If $z$ is a corner then it is in three. Any pair of fans sharing a border must have fraction indices $p_1/q_1$, $p_2/q_2$ that are Farey neighbors: $|p_1 q_2- p_2 q_1|=1$. So  Corollary \ref{fgc} implies that  the continued fractions of Farey neighbors can only differ in their last terms. 
This will be seen in Proposition \ref{fnb} with a geometric proof.

For example,  $p_1/q_1=3483/1258$ and $p_2/q_2=778/281$ are Farey neighbors, and let $p_3/q_3$ be their mediant. Then the fans $F_{p_1/q_1}$,  $F_{p_2/q_2}$ and  $F_{p_3/q_3}$ all contain a common corner $z$. We have
\begin{align*}
  p_1/q_1 & = [2, 1, 3, 3, 10, 2, 4], \\
  p_2/q_2 & = [2, 1, 3, 3, 10, 2], \\
   p_3/q_3 & = [2, 1, 3, 3, 10, 2, 5], \\
    z = \frac{p_1(e^{\pi i/3}-1)-p_2}{q_1(e^{\pi i/3}-1)-q_2} & = [2, 1, 3, 3, 10, 2, 4+e^{\pi i/3}],
\end{align*}
with $z$ satisfying one of \e{pqx} for each of $p_1/q_1$, $p_2/q_2$ and $p_3/q_3$.
Proposition \ref{conv} will provide a converse to Corollary \ref{fgc}.

\section{Uniqueness} \la{uni}

A real number can have more than one continued fraction expansion, with for example $[4,3]=[4,2,1]$. The next well-known result, see \cite[Thms.~160, 169]{hawr}, shows that this is the only possible kind of non-uniqueness.

\begin{theorem} \la{cfhw}
Let $a_j$, $b_j$ be integers, positive for $j\gqs 1$. Suppose 
\begin{equation*} 
[a_0,a_1, \dots ,a_{m-1},a_m]=[b_0,b_1, \dots ,b_{n-1} ,b_n]
\end{equation*}
for  $a_m, b_n \gqs 2$. Then $m=n$ and $a_j=b_j$ for all $j$. 
For two infinite expansions 
$
[a_0,a_1, a_2, \dots ]=[b_0,b_1,  b_2, \dots]
$,
we may also conclude that $a_j=b_j$ for all $j$. 
\end{theorem}

Let $\rho:=e^{\pi i/3}$. Then $\{\rho, \rho^2, \rho^4, \rho^5\} \subseteq \Fd'$ are the four corners of the region, as seen in Figure \ref{funds}. The left corners $\rho^2, \rho^4$ are the primitive $3$rd roots of unity; the right corners $\rho, \rho^5$ are the primitive $6$th roots.

\begin{theorem} \la{corn2}
Let $a_j$, $b_j$ be integers, positive for $j\gqs 1$. Suppose 
\begin{equation} \la{ab}
[a_0,a_1, \dots ,a_{m-1},a_m+z_0]=[b_0,b_1, \dots ,b_{n-1} ,b_n+w_0]
\end{equation}
for  $z_0$, $w_0 \in \Fd'-\{0, \rho^2, \rho^4\}$. Then $m=n$, $a_j=b_j$ for all $j$, and $z_0 = w_0$.
\end{theorem}
\begin{proof}
Let $z$ be the common value of \e{ab}. By adding a large integer to both sides, we may suppose that $a_0$, $b_0$ are positive. Assume that $z_0$, $w_0 \in \Fd'-\{0\}$ and we will examine how uniqueness may fail when one is equal to $\rho^2$ or $\rho^4$. 

Define the matrices in $\SL(2,\Z)$,
\begin{equation} \la{lru}
  L := T=\begin{pmatrix} 1 & 1 \\ 0 & 1 \end{pmatrix}, \quad R := TST = -ST^{-1}S =\begin{pmatrix} 1 & 0 \\ 1 & 1 \end{pmatrix}.
\end{equation}
Then
\begin{equation*}
  L^{a_j} w= a_j +w, \qquad  R^{a_j} w= \frac 1{a_j +1/w},
\end{equation*}
and hence
\begin{equation*}
  z=L^{a_0}R^{a_1} \cdots L^{a_m} z_0 \qquad \text{and} \qquad  z=L^{a_0}R^{a_1} \cdots R^{a_m} J z_0
\end{equation*}
for $m$ even and odd, respectively. 
Put
\begin{equation} \la{z1}
  z_1 = \begin{cases}
  z_0 & \quad \text{ if $m$ even};\\
   J z_0 & \quad \text{ if $m$ odd},\\
  \end{cases}
\qquad \qquad
  w_1 = \begin{cases}
  w_0 & \quad \text{ if $n$ even};\\
   J w_0 & \quad \text{ if $n$ odd}.\\
  \end{cases}
\end{equation}
Therefore
\begin{equation} \la{alt}
  L^{a_0}R^{a_1} \cdots [L\text{ or }R]^{a_m} z_1 = L^{b_0}R^{b_1} \cdots [L\text{ or }R]^{b_n} w_1
\end{equation}
for $z_1$, $w_1 \in \Fd'-\{0\}$. The notation in \e{alt} will be used from here on to always indicate alternating powers of $L$ and $R$. If  $z\in \H$, as we will assume until the final part of the proof, then so are $z_1$, $w_1$ and hence
\begin{equation} \la{z2}
  z_2=S^\alpha z_1 \in \Fd, \qquad  w_2=S^\beta w_1 \in \Fd \qquad \text{for some} \quad \alpha, \beta \in \{0,1\}.
\end{equation}
Then
 $ M z_2 = N w_2$ for 
\begin{equation*}
 M:=L^{a_0}R^{a_1} \cdots [L\text{ or }R]^{a_m}S^\alpha, \qquad 
 N:= L^{b_0}R^{b_1} \cdots [L\text{ or }R]^{b_n} S^\beta.
\end{equation*}
Hence $z_2=w_2$ and $N^{-1}M \in \SL(2,\Z)$ fixes this element of $\Fd$. In $\SL(2,\Z)$ it is well known that $i$ is fixed by $\{ \pm I, \pm S\}$, $\rho^2=e^{2\pi i/3}$ is fixed by $\{ \pm I, \pm ST, \pm (ST)^2\}$ and every other element  of $\Fd$ is just fixed by $\{ \pm I\}$.

Assume first that $N^{-1}M \in \{ \pm I, \pm S\}$. Then for some $\delta \in \{0,1\}$ we have 
\begin{equation*}
 L^{a_0}R^{a_1} \cdots [L\text{ or }R]^{a_m} = \pm L^{b_0}R^{b_1} \cdots [L\text{ or }R]^{b_n} S^\delta.
\end{equation*}
Note that $L^{a_0}R^{a_1} \cdots [L\text{ or }R]^{a_m} \in \SL(2,\Z)$ has $4$ strictly positive entries  if $m\gqs 1$ and $3$ strictly positive entries and bottom left  entry zero if $m=0$. Similarly for $L^{b_0}R^{b_1} \cdots [L\text{ or }R]^{b_n}$. Hence, by this positivity, $\delta$ must be $0$ and
\begin{equation} \la{lar}
 L^{a_0}R^{a_1} \cdots [L\text{ or }R]^{a_m} =  L^{b_0}R^{b_1} \cdots [L\text{ or }R]^{b_n}.
\end{equation}
Letting each side of \e{lar} act on $1$ next implies that
\begin{equation*}
  [a_0, a_1, \dots, a_{m-1}, a_m+1]=[b_0, b_1, \dots, b_{n-1}, b_n+1],
\end{equation*}
and the theorem now follows in this case  from Theorem \ref{cfhw}.

That leaves the fixed point $z_2 =w_2 = \rho^2$ and $N^{-1}M = \pm ST$ or $\pm (ST)^{-1}$. In the former,
\begin{equation} \la{ara}
  L^{a_0}R^{a_1} \cdots [L\text{ or }R]^{a_m} = \pm L^{b_0}R^{b_1} \cdots [L\text{ or }R]^{b_n} S^\beta ST S^\alpha.
\end{equation}
We examine the four possibilities for $\alpha$ and $\beta$ in \e{ara} next, describing the exceptions to uniqueness they give rise to.

\vskip 3mm
{\bf Case  $\bm{\alpha=0}$, $\bm{\beta=0}$.} 
If $n$ is odd then the right side of \e{ara} has
\begin{equation*}
  \pm L^{b_0}R^{b_1} \cdots R^{b_n} ST = \pm L^{b_0}R^{b_1} \cdots R^{b_n-1} TSTST
\end{equation*}
from which it follows that
\begin{equation*}
  L^{a_0}R^{a_1} \cdots [L\text{ or }R]^{a_m}S = \pm L^{b_0}R^{b_1} \cdots R^{b_n-1}.
\end{equation*}
However, this is not possible by positivity. When $n$ is even, use $L^{b_n-1} L ST =  L^{b_n-1}R$ and act on $1$ to find
\begin{equation*}
  [a_0, a_1, \dots, a_{m-1}, a_m+1]= \begin{cases}
  [b_0, b_1, \dots, b_{n-1}+2] & \quad \text{ if $b_n=1$};\\
  [b_0, b_1, \dots, b_n-1, 2] & \quad \text{ if $b_n \gqs 2$}.\\
  \end{cases}
\end{equation*}
Therefore, when $b_n=1$ then $m=n-1$, $a_m=b_{n-1}+1$, $z_0=\rho^4$ and $w_0=\rho^2$. This means that, in this instance,  \e{ab} corresponds to the identity
\begin{equation} \la{cc1}
  [c_0, c_1, \dots, c_{m-1}, (c_m+1)+\rho^4] =  [c_0, c_1, \dots, c_{m}, 1 +\rho^2].
\end{equation}
When $b_n\gqs 2$ then $m=n+1$, $a_{n+1}=1$, $a_n=b_{n}-1$, $z_0=\rho^4$ and $w_0=\rho^2$. Then  \e{ab} corresponds to 
\begin{equation} \la{cc2}
  [c_0, c_1, \dots, c_{n}, 1+\rho^4] =  [c_0, c_1, \dots, (c_{n}+1) +\rho^2].
\end{equation}

{\bf Case  $\bm{\alpha=0}$, $\bm{\beta=1}$.}
Here
\begin{equation*} 
  L^{a_0}R^{a_1} \cdots [L\text{ or }R]^{a_m} = L^{b_0}R^{b_1} \cdots [L\text{ or }R]^{b_n} L.
\end{equation*}
Acting on $1$ finds
\begin{equation*}
  [a_0, a_1, \dots, a_{m-1}, a_m+1]= \begin{cases}
  [b_0, b_1, \dots, b_{n-1}, b_n+2] & \quad \text{ if $n$ even};\\
  [b_0, b_1, \dots, b_{n-1}, b_n, 2] & \quad \text{ if $n$ odd}.\\
  \end{cases}
\end{equation*}
So for $n$ even, $m=n$, $a_m=b_n+1$, $z_0=\rho^2$ and $w_0=\rho$. This means that \e{ab} corresponds to the identity
\begin{equation} \la{cc3}
  [c_0, c_1, \dots, c_{m-1}, (c_m+1)+\rho^2] =  [c_0, c_1, \dots, c_{m-1}, c_m +\rho].
\end{equation}
For $n$ odd, $m=n+1$, $a_m=1$, $z_0=\rho^2$ and $w_0=\rho^5$. Thus \e{ab} corresponds to 
\begin{equation} \la{cc4}
  [c_0, c_1, \dots, c_{n-1}, c_n, 1+\rho^2] =  [c_0, c_1, \dots, c_{n-1}, c_n +\rho^5].
\end{equation}

{\bf Case  $\bm{\alpha=1}$, $\bm{\beta=0}$.}
With $STS \cdot R=-I$, \e{ara} becomes
\begin{equation*}
   L^{a_0}R^{a_1} \cdots [L\text{ or }R]^{a_m}R = L^{b_0}R^{b_1} \cdots [L\text{ or }R]^{b_n}.
\end{equation*}
Consequently
\begin{equation*}
  [b_0, b_1, \dots, b_{n-1}, b_n+1]= \begin{cases}
  [a_0, a_1, \dots,  a_m, 2] & \quad \text{ if $m$ even};\\
  [a_0, a_1, \dots, a_{m-1}, a_m+2] & \quad \text{ if $m$ odd}.\\
  \end{cases}
\end{equation*}
If $m$ is even then $n=m+1$, $b_n=1$, $z_0=\rho$ and $w_0 = \rho^4$. Therefore \e{ab} corresponds to the identity
\begin{equation} \la{cc5}
  [c_0, c_1, \dots, c_{m-1}, c_m+\rho] =  [c_0, c_1, \dots,  c_m, 1+\rho^4].
\end{equation}
If $m$ is odd then $n=m$, $b_n=a_m+1$, $z_0=\rho^5$ and $w_0 = \rho^4$. Here \e{ab} corresponds to 
\begin{equation} \la{cc6}
  [c_0, c_1, \dots, c_{m-1}, c_m+\rho^5] =  [c_0, c_1, \dots, c_{m-1}, (c_m+ 1)+\rho^4].
\end{equation}

{\bf Case  $\bm{\alpha=1}$, $\bm{\beta=1}$.}
This is easily ruled out by positivity.

\vskip 3mm
We have completed the analysis for $N^{-1}M = \pm ST$. 
When $N^{-1}M = \pm (ST)^{-1}$ then this is just the same, but with $M$ and $N$ interchanged, and so produces no further exceptions to uniqueness. Since $\rho^2$ or $\rho^4$ appear in each of the exceptions \e{cc1} -- \e{cc6}, we obtain the theorem provided our assumption holds  that the common value $z$ of \e{ab} is in $\H$. If $z\in \overline{\H}$ then apply this result to $[1,z]=1+1/z \in \H$ to obtain the theorem in this case also.
\end{proof}

Note that $\rho^2$ and $\rho^4$ are the only two corners with at least one  appearing in each of \e{cc1} -- \e{cc6}. So we cannot replace them in the statement of Theorem \ref{corn2} by a different pair or a single corner. This also means that our algorithm convention of choosing $\delta=0$ when $z-m = e^{\pm \pi i/3}$ is the natural one.

\begin{proof} [Proof of Theorem \ref{main}] Theorem \ref{genc} shows that the algorithm produces a continued fraction representation for every $z\in \C$, necessarily canonical. Theorems \ref{cfhw}, \ref{corn2} prove its uniqueness.
\end{proof}

We also see that for every integer $a_0$, sequence $a_1, \dots, a_n$ of positive integers and $z_0 \in \Fd' - \{0, \rho^2, \rho^4\}$, there is a unique $[a_0, a_1, \dots, a_n +z_0]$ in $\C - \R$, with all elements of $\C - \R$ taking this form.

\section{Some more properties}

The simple reasoning at the beginning of the proof of Theorem \ref{corn2} will be needed again, so we state it as a lemma.

\begin{lemma} \la{fff}
Let $z \in \H$ have continued fraction $[a_0, a_1, \dots, a_n+z_0]$. Then $z = M z_2$ for $z_2 \in \Fd$ and 
\begin{equation*}
  M = L^{a_0}R^{a_1} \cdots [L\text{ or }R]^{a_n}S^\alpha \in \SL(2,\Z),
\end{equation*}
for some $\alpha \in \{0,1\}$. Here  $z_2 = S^\alpha z_0$ if $n$ is even and $z_2 = S^\alpha (1/z_0)$ if $n$ is odd.
\end{lemma}

Note that throughout this paper, a continued fraction $[a_0, a_1, \dots, a_n+z_0]$ is always assumed to satisfy the first part of Definition \ref{srf}, but may not  be canonical.

Starting with a usual continued fraction $[a_0, a_1, \dots, a_n]$,  its {\em convergents} are $[a_0, a_1, \dots, a_m] = h_m/k_m$ in lowest terms, for $m\lqs n$.  These numerators and denominators satisfy the recursions
\begin{equation} \la{hk}
  h_m = a_m h_{m-1}+h_{m-2}, \qquad  k_m = a_m k_{m-1}+k_{m-2}, \qquad \text{with} \qquad \frac{h_{-1}}{k_{-1}}=\frac 10, \quad \frac{h_{-2}}{k_{-2}}=\frac 01.
\end{equation}
Since $a_j \gqs 1$ for $j\gqs 1$, it follows from \e{hk} that $k_m\gqs f_{m+1}$ for the exponentially growing Fibonacci numbers $f_0=0$, $f_1=1,$ etc. 
For any $z\in \C$ we have, \cite[Sect. 10.2]{hawr},
\begin{equation}\label{hk2}
  [a_0, a_1, \dots, a_n+z]  = \frac{h_{n-1} z +h_{n}}{k_{n-1} z +k_{n}} = \begin{pmatrix} h_{n-1} & h_n \\ k_{n-1} & k_n \end{pmatrix} z,
\end{equation}
where the matrix in \e{hk2} has determinant $(-1)^n$.
The sign of a real number is $-1, 0$ or $1$ and 
we have easily:
\begin{equation}\label{sgn}
  z=[a_0,a_1, \dots, a_n+z_0] \quad \implies \quad \sgn(\Im(z)) = (-1)^n \sgn(\Im(z_0)).
\end{equation}

\begin{prop} \la{conv}
Suppose $z \in \H$ has continued fraction $[a_0,a_1, \dots, a_n+z_0]$. Then, with \e{hk2}, $z$ is in the fan $F_{h_{n-1}/k_{n-1}}$ if $z_0 \in \pm\Fd$ and in $F_{h_{n}/k_{n}}$ otherwise.
\end{prop}
\begin{proof}
If $n$ is even, then  \e{sgn} implies that $z_0 \in \Fd$ or $S\Fd$. In these cases, using  \e{hk2}, \e{fans},
\begin{equation*}
  z \in \begin{pmatrix} h_{n-1} & h_n \\ k_{n-1} & k_n \end{pmatrix} \Fd \subseteq F_{h_{n-1}/k_{n-1}} \qquad \text{or} \qquad
  z \in \begin{pmatrix} h_{n-1} & h_n \\ k_{n-1} & k_n \end{pmatrix} S\Fd 
  = \begin{pmatrix} h_{n} & -h_{n-1} \\ k_{n} & -k_{n-1} \end{pmatrix} \Fd\subseteq F_{h_{n}/k_{n}}.
\end{equation*}
If $n$ is odd, then by \e{sgn}, $z_0 \in -\Fd$ or $-S\Fd$. Therefore,
\begin{align*}
  z & \in \begin{pmatrix} h_{n-1} & h_n \\ k_{n-1} & k_n \end{pmatrix} \begin{pmatrix} 1 & 0 \\ 0 & -1 \end{pmatrix} \Fd 
  =\begin{pmatrix} h_{n-1} & -h_n \\ k_{n-1} & -k_n \end{pmatrix}  \Fd 
  \subseteq F_{h_{n-1}/k_{n-1}} \qquad \text{or} \\
  z & \in \begin{pmatrix} h_{n-1} & h_n \\ k_{n-1} & k_n \end{pmatrix} \begin{pmatrix} -1 & 0 \\ 0 & 1 \end{pmatrix} S\Fd 
  = \begin{pmatrix} h_{n} & h_{n-1} \\ k_{n} & k_{n-1} \end{pmatrix} \Fd\subseteq F_{h_{n}/k_{n}}.
\end{align*}
Note that  \e{fans} requires a matrix with determinant $1$, as we have provided here.
\end{proof}

Let $C_{p/q}$ denote the Ford circle 
 that is tangent to $\R$ at $p/q$. It has diameter $1/q^2$ and fits exactly inside the fan $F_{p/q}$.

\begin{cor} \la{conv2}
Suppose $z \in \H$ has continued fraction $[a_0,a_1, \dots, a_n+z_0]$. Then $z$ is in the Ford circle $C_{h_{n-1}/k_{n-1}}$ if $|\Im(z_0)|\gqs 1$. The Ford circle $C_{h_{n}/k_{n}}$ contains $z$ if $|z_0 \pm i/2| \lqs 1/2$. Otherwise, $z$ is not in any Ford circle.
\end{cor}

\SpecialCoor
\psset{griddots=5,subgriddiv=0,gridlabels=0pt}
\psset{xunit=0.8cm, yunit=0.8cm, runit=0.8cm}
\psset{linewidth=1pt}
\psset{dotsize=5pt 0,dotstyle=*}
\begin{figure}[ht]
\centering
\begin{pspicture}(-3.5,-2.8)(8,2.8) 

\psset{arrowscale=1.1,arrowinset=0.1}
\newrgbcolor{light}{0.8 0.7 1.0}
\newrgbcolor{bmid}{0.6 0.5 1}
\newrgbcolor{bmid}{0.99609375 0.40234375 0.16796875}

\newrgbcolor{bb}{0.4 0.3 1}
\newrgbcolor{bb}{1 0 0}
\newrgbcolor{cx}{0.8 0.8 1.0}

\newrgbcolor{white2}{0.4 0.3 1}
\newrgbcolor{white2}{1 0 0}
\newrgbcolor{light}{0.99609375 0.71875 0.6171875}

\newrgbcolor{pale}{0.996094, 0.71875, 0.617188}


\psarc[linecolor=pale,fillstyle=solid,fillcolor=pale](0.23498,0){2.76304}{0}{360}

\psarc[linecolor=white,fillstyle=solid,fillcolor=white](-0.198145,0){2.32992}{0}{360}
\psarc[linecolor=white,fillstyle=solid,fillcolor=white](3.83142,0){1.30336}{0}{360}
\psarc[linecolor=white,fillstyle=solid,fillcolor=white](1.88637,0){0.641696}{0}{360}

\psline[linecolor=gray](-3.5,0)(4,0)

\psarc[linestyle=dotted,linecolor=bb](0.,0){2.52806}{0.}{360.}

\psarc[linecolor=pale](3.83142,0){1.30336}{180}{221.709}
\psarc[linestyle=dashed](3.83142,0){1.30336}{180}{221.709}
\psarc(3.83142,0){1.30336}{138.291}{180}

\psarc(1.88637,0){0.641696}{-75.4619}{0}
\psarc[linecolor=pale](1.88637,0){0.641696}{0}{75.4619}
\psarc[linestyle=dashed](1.88637,0){0.641696}{0}{75.4619}

\psarc[linecolor=pale](0.23498,0){2.76304}{18.2914}{180}
\psarc[linestyle=dashed](0.23498,0){2.76304}{18.2914}{180}

\psarc(-0.198145,0){2.32992}{15.4619}{180}
\psarc[linecolor=pale](-0.198145,0){2.32992}{180}{344.538}
\psarc[linestyle=dashed](-0.198145,0){2.32992}{180}{344.538}

\psarc(0.23498,0){2.76304}{180}{341.709}

\psdot(2.25029,0.584693)

\rput(1.66,.6){$z$}
\rput(8,1.2){$
               z  = [-4, 2, 3, 7, 1, 5+\tfrac{1-3i}4 ]$}
 \rput(8,.6){$ \approx -3.57+0.00000585 i
             $}

\end{pspicture}
\caption{A typical image of $\Fd'$. This one contains $z$ from \e{423}}
\label{egf}
\end{figure}

A more precise location for $z$ from its continued fraction coefficients $a_0, \dots, a_n$ is given by the particular image of $\Fd'$ under $\SL(2,\Z)$ that contains it. A typical image is displayed  in Figure \ref{egf}. Since $\SL(2,\Z)$ maps the imaginary axis to a vertical line or a circle centered on $\R$, we see these images are translates of $\Fd'$ or thickened circles that narrow down to cusps at rational numbers.

\begin{cor} \la{conv4}
Suppose $z \in \C$ has continued fraction $[a_0,a_1, \dots, a_n+z_0]$. Then $z$ is in the image of $\Fd'$ under $\SL(2,\Z)$ that has cusps at $h_{n-1}/k_{n-1}$ and $h_{n}/k_{n}$.
\end{cor}

In this way, the coefficients $a_0, \dots, a_n$ determine where $z$ lies in the tiling of Figure \ref{tile}, giving an address for each tile. These tiles accumulate at the real line and the discs there in the figure indicate the limit of its accuracy. 

\begin{prop} \la{cplx}
Let $n\gqs 1$, let  $a_j$ be integers that are positive after the first, and suppose $z_0 \in \Fd'$. Then we have, in the  notation of \e{hk2} and with $\phi:=(1+\sqrt{5})/2$,
\begin{equation} \la{top}
  \Bigl| [a_0,a_1, \dots, a_n+z_0] -  [a_0,a_1, \dots, a_n] \Bigr| \lqs \frac{1}{k_n k_{n-1}} < \frac 4{\phi^{2n}}.
\end{equation}
\end{prop}
\begin{proof}
Verify the identity
\begin{equation} \la{id}
  \frac{a z+b}{c z+d} = \frac ac -\frac{(-1)^n}{c(c z+d)}  \qquad \text{when} \qquad \det \begin{pmatrix}a & b \\ c & d \end{pmatrix} = (-1)^n.
\end{equation}
We can assume $z_0 \neq 0$ and let $w=1/z_0 \in \Fd'$. Then by \e{hk2} and \e{id}, the left side of \e{top} is bounded by
\begin{equation}\label{id2}
  \frac{1}{|k_n(k_n w+k_{n-1})|} = \frac 1{k_n^2 |w+k_{n-1}/k_n|}  \qquad \text{for} \qquad 0< k_{n-1}/k_n \lqs 1.
\end{equation}
Also $|w-(-1)|\gqs 1$ for $w\in \Fd'$ by \e{fdp2}. Hence, for $0 \lqs r\lqs 1$,
\begin{equation*}
  |w+r|=|w+1+r-1|\gqs |w+1|-|1-r|\gqs 1-(1-r)=r,
\end{equation*}
by the triangle inequality, so that \e{id2} is at most $1/(k_n k_{n-1})$. Then $k_n k_{n-1} \gqs f_n f_{n+1}>\phi^{2n}/4$.
\end{proof}

We obtain the following corollary, giving a bound on the length of continued fractions. See \cite[Sect. 3]{DFV} for the same treatment. 
\begin{cor} 
If $z \in \C$ has $|\Im (z) | > 4/\phi^{2n}$ for $n\gqs 1$ then $z$ has continued fraction $[a_0,a_1, \dots, a_k+z_0]$ for $k \lqs n-1$.
\end{cor}

\SpecialCoor
\psset{griddots=5,subgriddiv=0,gridlabels=0pt}
\psset{xunit=12cm, yunit=12cm, runit=12cm}
\psset{linewidth=1pt}
\psset{dotsize=5pt 0,dotstyle=*}
\begin{figure}[ht]
\centering

\caption{Tiling $\C$ with images of $\Fd'$}
\label{tile}
\end{figure}

Two complex numbers $w$ and $z$ are said to be {\em equivalent} if $w =M z$ for some $M \in \GL(2,\Z)$. The classic result of Serret from 1850, see \cite[Thm. 175]{hawr}, says that two irrational reals are equivalent if and only if their continued fractions agree after some point. Equivalence for other numbers now has a simple characterization:

\begin{prop} \la{se}
Suppose  $z, w \in \C$ 
are not irrational reals, so they have 
continued fractions 
\begin{equation} \la{zw}
  z =  [a_0,  \dots, a_m+z_0], \qquad w =  [b_0,  \dots, b_n+w_0].
\end{equation}
Then $z$ and $w$ are equivalent if and only if $w_0=\pm z_0$ or $\pm 1/z_0$.
\end{prop}
\begin{proof}
By \e{hk2}, the representations in  \e{zw} imply $z=A z_0$ and $w=B w_0$ for $A$, $B \in \GL(2,\Z)$. If $z$ and $w$ are equivalent then $w =M z$ and hence $w_0 = C z_0$ for  $C =B^{-1}M A$. If $z_0=0$ then $w_0=C 0 \in \Q$ implies that $w_0=0$. Similarly $w_0=0$ implies $z_0=0$. Otherwise, apply $J$ and $S$ as necessary to move $w_0$ and $z_0$ into $\Fd$: $w_0'=(J)(S)w_0 \in \Fd$, $z_0'=(J)(S)z_0 \in \Fd$. Hence $w_0'= C' z_0'$ with $C'$ necessarily in $\SL(2,\Z)$. Therefore $w_0'=z_0'$ and so $w_0=\pm z_0$ or $w_0=\pm 1/z_0$, as we wanted to show.
Reversing this reasoning gives the converse.
\end{proof}

As a special case, all rationals are equivalent; this is \cite[Thm. 174]{hawr}. For the three sets of the rationals, the irrational reals, and the rest of $\C$, it is easy to see that any element  from one set must be inequivalent to all elements in the other two.

\section{Reduction} \la{red}
\subsection{Lattices} \la{latt}
A two-dimensional lattice may be expressed as $\Z w_1+\Z w_2$ where $w_1$, $w_2 \in \C$ are $\R$-linearly independent. The given basis $(w_1,w_2)$ might be larger than necessary and a smaller one is found as follows. First select a non-zero $u$ in the lattice with $|u|$ minimal. Then choose $v$ with $|v|$ minimal over all  lattice elements that are $\R$-linearly independent of $u$. Then  $(u,v)$ is  a basis of the original lattice and it is called {\em minimal}.

Elementary 
arguments, as in \cite[Sect. 17.1]{gal}, show that a basis $(u,v)$  is  minimal if and only if
\begin{equation} \la{uv}
  |u| \lqs |v| \lqs |u\pm v|.
\end{equation}
With $z=v/u$ we find that \e{uv} is equivalent to
\begin{equation} \la{uvz}
  1 \lqs |z|, \qquad -1/2 \lqs \Re(z) \lqs 1/2.
\end{equation}
If  the order of the basis elements is not important, then swapping $u$ and $v$ gives $1/z$ and the desired region is seen to be the closure of $\Fd'$ from \e{fdp}, (excluding $0$). Hence we may define a lattice basis $(u,v)$ to be {\em unordered minimal} if $u/v$ is in the closure of $\Fd'$.

\begin{theorem} \la{bm}
Let $\Z w_1+\Z w_2$ be any two-dimensional lattice in $\C$. If the continued fraction representation of $w_1/w_2$ is
\begin{equation*}
  \frac{w_1}{w_2} =  [ a_0,a_1, \dots, a_{r-1}, a_r+ z_0 ] = \begin{pmatrix} h_{r-1} & h_r \\ k_{r-1} & k_r \end{pmatrix} z_0,
\end{equation*}
with \e{hk2},  put
\begin{equation} \la{invm}
 \begin{pmatrix} u \\ v \end{pmatrix} := \begin{pmatrix} h_{r-1} & h_r \\ k_{r-1} & k_r \end{pmatrix}^{-1}  \begin{pmatrix} w_1 \\ w_2 \end{pmatrix}
  = (-1)^r \begin{pmatrix} k_r & -h_r \\ -k_{r-1} & h_{r-1} \end{pmatrix}  \begin{pmatrix} w_1 \\ w_2 \end{pmatrix}.
\end{equation}
Then $(u,v)$ is an unordered minimal basis of the lattice.
\end{theorem}
\begin{proof}
Firstly, $(u,v)$ is a  basis since the integer matrices in \e{invm} have determinant $\pm 1$. It is unordered minimal since
\begin{equation} \la{qed}
 \frac uv = \begin{pmatrix} h_{r-1} & h_r \\ k_{r-1} & k_r \end{pmatrix}^{-1}  \frac{w_1}{w_2} 
  = z_0 \in \Fd'. 
\end{equation}
\end{proof}

We can describe how the continued fraction algorithm changes the basis at each stage of this reduction, ignoring the $a_j$ coefficients. Starting with $(u,v)=(w_1,w_2)$, the step
\begin{equation}\label{swa}
(u,v) \mapsto (v, u-\lfloor \Re(u/v) \rfloor v)
\end{equation}
is repeated until $|u|\lqs |v|$ or $|v-u|$. Then $(u,v)$ or $(u,v-u)$ is the desired unordered minimal basis. In terms of $z=v/u$, this means repeating $z \mapsto 1/z-\lfloor \Re(1/z) \rfloor$ until $1\lqs |z|$ or $|z-1|$, with then $z$ or $z-1$ in the closure of $\Fd'$.
In \cite[Sect. 2]{DFV} they call this the Standard Gaussian Algorithm, with the only difference of using the condition $1/2<|z-1/2|$ to exit from the loop. This can finish a step earlier, with a further transformation  needed to get $z$ and the basis into the wanted form. See Section \ref{fur} where equivalent presentations of our complex continued fraction algorithm are given. 

The usual version of Gauss lattice reduction, as described in \cite[p. 42]{lll} for example, works in the following way. Let $\lfloor x \rceil$ denote the closest integer to $x\in \R$, so equaling $\lfloor x+1/2 \rfloor$ or similar. Starting with $(u,v)=(w_1,w_2)$, the step
\begin{equation}\label{swa2}
(u,v) \mapsto (v, u-\lfloor \Re(u/v) \rceil v)
\end{equation}
is repeated until $|u|\lqs |v|$. Then $(u,v)$  is a minimal basis. In terms of $z=v/u$, this means repeating $z \mapsto 1/z-\lfloor \Re(1/z) \rceil$ until $1\lqs |z|$. An alternative continued fraction could be built from this, though the resulting coefficients $a_j$ would be both positive and negative.

\subsection{Quadratic forms}

An integral binary quadratic form is a polynomial $q(x,y)=a x^2 + b xy + c y^2$ with integer coefficients $a, b, c$. It may be written $q=[a, b, c]$, and its discriminant is $D=b^2-4ac$.  The group $\SL(2,\Z)$ acts on forms on the right:
\begin{equation}\label{qm}
q|M := q(\alpha x +\beta y, \g x+\delta y) \qquad \text{for} \qquad M = \begin{pmatrix} \alpha &\beta \\ \g &\delta \end{pmatrix},
\end{equation}
preserving the discriminant. Two  forms $q_1$, $q_2$ are equivalent if $q_2=q_1|M$ for some $M \in \SL(2,\Z)$. Lagrange 
and Gauss studied form equivalence classes. To understand when two forms are equivalent, and to count the number of classes for each discriminant, they used a reduction method. 
Gauss refined Lagrange's work slightly and his reduction criterion for a form $[a,b,c]$ of discriminant $D<0$ with $a, c>0$, (i.e. positive definite),  may be presented as:
\begin{equation}\label{lag}
   |b| \lqs a \lqs c, \quad \text{and if  $|b|= a$ or $a = c$ then $b\gqs 0$}.
\end{equation}
They provided simple procedures to convert any form into a reduced one, laying the foundation of this theory, and
there is exactly one form satisfying \e{lag} in each class.  An equivalent way to state this condition uses the zeros of $q(x,1)$, where we define $z_q$ to be the zero in $\H$. This equals $(-b+\sqrt{|D|} i)/(2a)$. Gauss's  criterion  then becomes simply $z_q \in \Fd$.

We will need
\begin{equation}\label{roo}
  p=q|M \quad \implies \quad z_{p}=M^{-1} z_{q}.
\end{equation}

\begin{theorem} \la{bm2}
Let $q$ be a quadratic form $[a,b,c]$ of discriminant $D<0$ with $a, c>0$.  If the continued fraction representation of its root $z_q$ is
$[ a_0,a_1, \dots, a_{m-1}, a_m+ z_0 ]$
then, for an appropriate choice of $(S)$,
\begin{equation} \la{reda}
  q|M  \quad \text{ is Gauss reduced for } \quad M:=  L^{a_0}R^{a_1} \cdots [L\text{ or }R]^{a_m} (S).
\end{equation}
\end{theorem}
\begin{proof}
 By Lemma \ref{fff} we have $z_q =M z_2$ for $M$ in the form in \e{reda} and $z_2 \in \Fd$. Let $p := q|M$. Then by \e{roo} 
\begin{equation} \la{prc}
  z_{p}= M^{-1}z_q =z_2 \in \Fd.
\end{equation}
\end{proof}

The situation for forms $q$ of positive, non-square discriminant $D$ is similar, but with the key difference that the reduced forms in an equivalence class naturally form finite cycles.  The root $z_q$ of $q$ is now a real quadratic irrational and its usual continued fraction $[a_0, a_1, a_2, \dots ]$ is infinite and becomes periodic. Let
\begin{equation} \la{redb}
  q_m := q| L^{a_0}R^{a_1} \cdots [L\text{ or }R]^{a_m},
\end{equation}
just using the first $m+1$ coefficients. Then this sequence of forms also becomes periodic as $m$ increases, giving a cycle of reduced forms. 
Two forms with the same positive, non-square discriminant are equivalent exactly when their cycles of reduced forms are the same.
We are describing the {\em simple reduction} of \cite[Sect. 6.3]{OStop} where these results may be found. The reduction method used by Gauss in this case was a little different, see  \cite[Sect. 7.1]{OStop}.

With \e{reda} and \e{redb} it is clear that our continued fraction for the complex root plays the same role as the usual continued fraction for the real root, giving the sequence of $L$s and $R$s to reach the reduced forms. In fact, it was looking for a continued fraction interpretation of \e{reda} that led to the complex continued fractions in \cite[Sect. 2]{OStop}. On Conway's topograph each quadratic form corresponds to a directed edge on a tree in the plane where all vertices have degree $3$. There the $L$s and $R$s correspond to turning sequences of Lefts and Rights on the tree. In Conway's terminology, the reduced form is at the topograph {\em well} when $D<0$, and the reduced cycle is on the periodic {\em river} when $D>0$. These techniques also work for square $D$, with the reduced form at a {\em lake}. See  \cite[Sects. 6, 7]{OStop} for details of all this.

\subsection{Connections to the literature}
Is there a link between reducing lattices and reducing quadratic forms with negative discriminant? Gauss noted that there was in an 1831 review of a paper on ternary forms, see \cite[p. 131]{mat}. To give the link simply, associate a form $q$ with the lattice basis $(z_q,1)$, recalling its root $z_q \in \H$. Writing bases vertically, Theorem \ref{bm} finds $N \in \text{GL}(2,\Z)$ so that $N (\begin{smallmatrix} z_q  \\ 1 \end{smallmatrix})$ is unordered minimal. More precisely, $N z_q \in \Fd'$ by \e{qed}. Including $S$ and $J$ if necessary, there is then a matrix $(S)(J)N$ so that  $(S)(J)N z_q \in \Fd$. Let $M\in \SL(2,\Z)$ be the inverse of this matrix and put $p=q|M$. Then $p$ is the reduced form we are looking for since, with \e{roo},
$
  z_p = M^{-1}  z_q \in \Fd$. In general, a positive definite quadratic form corresponds to the squared lengths of vectors in a lattice.
See also Stange's notes \cite[Sects. 2.7 -- 2.9]{sta}. They overlap with our work here in various ways, as we will see, pointing to further avenues of research.

It was Dirichlet, in the work \cite{dir} from 1854, who first employed continued fractions of real roots to simplify Gauss's quadratic form reduction method for  positive discriminants. Earlier, in the 1850 paper \cite{dir50}, he examined form reduction in the negative discriminant case and used a lattice reduction method in \cite[Sects.~1 -- 4]{dir50}, following Gauss's 1831 outline. This is equivalent to our description in the previous paragraph; see the discussion in \cite[pp. 124 -- 131]{mat}.

As we have seen, the continued fraction algorithm for complex numbers we are studying is essentially the same as the Standard Gaussian Algorithm investigated by Daud\'e,  Flajolet and  Vall\'ee in \cite{DFV}, and they  provide a detailed analysis of its behavior and complexity.
For example they show that, starting with  $z$ in the disc with diameter $[0,1]$, the probability that the algorithm needs 4 or more iterations is $\approx 0.01027$, while the average number of iterations is $\approx 1.35113$. The authors also  explicitly note the continued fractions for complex numbers the algorithm produces, though they do not seem to have followed up this direction.


\section{Generators and relations} \la{gen}

Recall the matrices $T, S, U, L, R \in \SL(2,\Z)$ from \e{tj}, \e{su}, \e{lru}.
The group $\SL(2,\Z)$ has the well-known abstract presentation
\begin{equation*}
  \left\langle s, u \, \left | \, s^4=1, s^2=u^3 \right.\right\rangle,
\end{equation*}
which is realized  in matrices with $s \mapsto S$ and $u  \mapsto  U$.
The next result, see for example 
 \cite[Thm. 5.9]{aig}, reflects this.

\begin{theorem} \la{us}
Every $M \in \SL(2,\Z)$ has a unique expression as
\begin{equation}\label{ueb}
   \pm(S) \quad \text{or} \quad \pm (S) U^{e_1} S U^{e_2} S \cdots S U^{e_n}(S),
\end{equation}
with $n\gqs 1$ and each $e_j$ equaling $1$ or $2$. As usual $(S)$ here means $S$ or $I$.
\end{theorem}

Given an $M$, we may use our continued fractions for $\C$ to easily find the exponents  $e_j$  as follows. Let $z:=(S)M(2i) \in \H$ with the factor $(S)$ chosen to ensure $\Re(z)\gqs 0$. (Any point in $\Fd$ could be used in place of $2i$, except $i$ and $e^{2\pi i/3}$ which have non-trivial stability groups.) As a continued fraction,
\begin{equation*}
  z=[ a_0,a_1, \dots, a_{m-1}, a_m+ z_0 ] =   L^{a_0} R^{a_1} \cdots  [L\text{ or }R]^{a_m} (S) z_2
\end{equation*}
with $a_0 \gqs 0$ and $z_2 \in \Fd$ by Lemma \ref{fff}. Then
\begin{equation*}
  ((S)M)^{-1} L^{a_0} R^{a_1} \cdots  [L\text{ or }R]^{a_m} (S) z_2 =2i
\end{equation*}
for $z_2, 2i \in \Fd$ implying that $z_2=2i$ and
\begin{equation} \la{m}
   M= \pm (S) L^{a_0} R^{a_1} \cdots  [L\text{ or }R]^{a_m} (S).
\end{equation}
Since $L=-U S$ and $R=-U^2 S$ we obtain
\begin{equation*}
  M=\pm (S) \underbrace{US \cdot US  \cdots US}_{a_0} \cdot \underbrace{U^2 S \cdot U^2 S  \cdots U^2 S}_{a_1}\cdots (S),
\end{equation*}
from which  the $e_j$s in \e{ueb} can be quickly found.

We also obtain a nice expression for $M$ in terms of $S$ and positive powers of $T$ from \e{m}, using $L=T$ and $R=TST$:
\begin{align}
  M & =\pm (S) T^{a_0} \cdot \underbrace{TST \cdot TST   \cdots TST}_{a_1} \cdot T^{a_2} \cdots (S), \notag\\
  & =\pm (S) T^{b_1} S  T^{b_2}S \cdots  S  T^{b_r}(S), \la{ww}
\end{align}
on relabelling, with $b_j\gqs 1$ and $b_j \gqs 2$ for $1<j<r$. 
The expansions \e{m} and \e{ww} are also unique:

\begin{theorem} \la{us2}
Every $M \in \SL(2,\Z)$ has a unique expression as 
\begin{equation} \la{lrex}
  \pm(S) \quad \text{or} \quad \pm (S) L^{a_0} R^{a_1} \cdots  [L\text{ or }R]^{a_m} (S),
\end{equation}
for integers $m\gqs 0$,  $a_j\gqs 1$,  (with $a_0=0$ possible if $m\gqs 1$). Every $M$ also has a unique expression as
\begin{equation}  \la{tsex}
    \pm(S) \quad \text{or} \quad \pm (S) T^{b_1} S  T^{b_2}S \cdots  S  T^{b_r}(S),
\end{equation}
for integers $r\gqs 1$, $b_j\gqs 1$ and $b_j \gqs 2$ for $1<j<r$. 
\end{theorem}
\begin{proof}
For $n\gqs 1$, let $W_n$ be the set of words $L^{a_0} R^{a_1} \cdots  [L\text{ or }R]^{a_m}$ with the stated restrictions on $a_j$ and with $a_0+ \cdots + a_m=n$. To see that there are $2^n$ such words, recall that the number of integer compositions of $n$ (ways to write $n$ as a sum of positive integers where the order of the summands matters) is $2^{n-1}$. We get twice $2^{n-1}$ since $a_0$ can be $0$ or not.

Replacing $L$ by $US$ and $R$ by $U^2 S$ in these words 
gives  $U^{e_1} S U^{e_2} S \cdots S U^{e_n}S$ with $e_j=1$ or $2$. There are clearly $2^n$ words in $U$ and $S$ of this form. Alternatively, replacing $L$ by $T$ and $R$ by $TST$ in each word in $W_n$ gives $T^{b_1} S  T^{b_2}S \cdots  S  T^{b_r}$ with the stated restrictions on $b_j$ and $b_1+ \cdots +b_r=n+r-1$. An exercise with compositions shows there are also $2^n$ words in $T$ and $S$ of this form.

It follows that there are bijections between the unique representations in Theorem \ref{us} and the representations here, completing the proof.
\end{proof}

Hence the continued fraction of $(S)M(2i)$ gives the unique expansions of $M$ into each of the forms \e{ueb}, \e{lrex} and \e{tsex}.
For example, let
\begin{equation*}
  M= \begin{pmatrix} 152 & 33 \\ -129 & -28 \end{pmatrix}, \qquad \text{so that} \qquad z=SM(2i)= \frac{79356 + 2 i}{93505}.
\end{equation*}
Then
\begin{equation*}
  z=[0, 1, 5, 1, 1, 1, 1, 4 -i/2]  \qquad \text{and therefore} \qquad M=-SR^1 L^5 R^1 L^1 R^1 L^1 R^4.
\end{equation*}
Converting to the other representations gives
\begin{align*}
  M & = -SU^2 S US US US US US U^2 S US U^2 S US U^2S  U^2S  U^2S  U^2S, \\
    & = -S T S T^7 S T^3 S T^3 S T^2 S  T^2 S  T^2 S  T.
\end{align*}

Closely related to $\SL(2,\Z)$ is the modular group $\PSL(2,\Z) := \SL(2,\Z)/\{\pm I\}$ with   presentation
$\langle s, u \,  | \, s^2=u^3=1 \rangle$. Theorems \ref{us} and \ref{us2} are easily also true for $\PSL(2,\Z)$,
with matrices now representing elements $\pm M$.

\section{Cutting sequences} \la{scu}
Let $\SL(2,\Z)$ act on $\mathbb I:=\{y i\, : \,y>0\}$ to produce the hyperbolic geodesics in $\H$  shown in Figure \ref{far}, called Farey edges since their endpoints  are Farey neighbors. 
This gives the  tiling of $\H$ into Farey triangles. 
\SpecialCoor
\psset{griddots=5,subgriddiv=0,gridlabels=0pt}
\psset{xunit=5cm, yunit=5cm, runit=5cm}
\psset{linewidth=0.3pt}
\psset{dotsize=7pt 0,dotstyle=*}
\begin{figure}[ht]
\centering
\begin{pspicture}(-0.7,-0.12)(1.9,0.86) 

\psset{arrowscale=1.4,arrowinset=0.3,arrowlength=1.1}
\newrgbcolor{light}{0.8 0.8 1.0}
\newrgbcolor{pale}{1 0.7 1}
\newrgbcolor{pale}{1 0.7 0.4}
\newrgbcolor{pale}{0.9179 0.7539 0.5781}
\newrgbcolor{pale}{0.996094, 0.71875, 0.617188}

\newrgbcolor{red}{0 0 1.0}
\newrgbcolor{red}{0.5 0.5 1.0}

\psline(-0.7,0)(1.9,0)
\psline(0,0)(0,0.76)
\psline(1,0)(1,0.76)

\rput(-0.5,0.35){%
        \begin{pspicture}(0,0)(1,0.7)
\psarc(0.5,0){0.5}{0}{180}

\psarc(0.25,0){0.25}{0}{180}

\psarc(0.75,0){0.25}{0}{180}

\psarc(0.166667,0){0.166667}{0}{180}

\psarc(0.416667,0){0.0833333}{0}{180}

\psarc(0.583333,0){0.0833333}{0}{180}

\psarc(0.833333,0){0.166667}{0}{180}

\psarc(0.125,0){0.125}{0}{180}

\psarc(0.291667,0){0.0416667}{0}{180}

\psarc(0.366667,0){0.0333333}{0}{180}

\psarc(0.45,0){0.05}{0}{180}

\psarc(0.55,0){0.05}{0}{180}

\psarc(0.633333,0){0.0333333}{0}{180}

\psarc(0.708333,0){0.0416667}{0}{180}

\psarc(0.875,0){0.125}{0}{180}

\psarc(0.1,0){0.1}{0}{180}

\psarc(0.225,0){0.025}{0}{180}

\psarc(0.267857,0){0.0178571}{0}{180}

\psarc(0.309524,0){0.0238095}{0}{180}

\psarc(0.354167,0){0.0208333}{0}{180}

\psarc(0.3875,0){0.0125}{0}{180}

\psarc(0.414286,0){0.0142857}{0}{180}

\psarc(0.464286,0){0.0357143}{0}{180}

\psarc(0.535714,0){0.0357143}{0}{180}

\psarc(0.585714,0){0.0142857}{0}{180}

\psarc(0.6125,0){0.0125}{0}{180}

\psarc(0.645833,0){0.0208333}{0}{180}

\psarc(0.690476,0){0.0238095}{0}{180}

\psarc(0.732143,0){0.0178571}{0}{180}

\psarc(0.775,0){0.025}{0}{180}

\psarc(0.9,0){0.1}{0}{180}

\psarc(0.0833333,0){0.0833333}{0}{180}

\psarc(0.183333,0){0.0166667}{0}{180}

\psarc(0.211111,0){0.0111111}{0}{180}

\psarc(0.236111,0){0.0138889}{0}{180}

\psarc(0.261364,0){0.0113636}{0}{180}

\psarc(0.279221,0){0.00649351}{0}{180}

\psarc(0.292857,0){0.00714286}{0}{180}

\psarc(0.316667,0){0.0166667}{0}{180}

\psarc(0.348485,0){0.0151515}{0}{180}

\psarc(0.369318,0){0.00568182}{0}{180}

\psarc(0.408333,0){0.00833333}{0}{180}

\psarc(0.422619,0){0.00595238}{0}{180}

\psarc(0.436508,0){0.00793651}{0}{180}

\psarc(0.472222,0){0.0277778}{0}{180}

\psarc(0.527778,0){0.0277778}{0}{180}

\psarc(0.563492,0){0.00793651}{0}{180}

\psarc(0.577381,0){0.00595238}{0}{180}

\psarc(0.591667,0){0.00833333}{0}{180}

\psarc(0.630682,0){0.00568182}{0}{180}

\psarc(0.651515,0){0.0151515}{0}{180}

\psarc(0.683333,0){0.0166667}{0}{180}

\psarc(0.707143,0){0.00714286}{0}{180}

\psarc(0.720779,0){0.00649351}{0}{180}

\psarc(0.738636,0){0.0113636}{0}{180}

\psarc(0.763889,0){0.0138889}{0}{180}

\psarc(0.788889,0){0.0111111}{0}{180}

\psarc(0.816667,0){0.0166667}{0}{180}

\psarc(0.916667,0){0.0833333}{0}{180}

\psarc(0.0714286,0){0.0714286}{0}{180}

\psarc(0.154762,0){0.0119048}{0}{180}

\psarc(0.174242,0){0.00757576}{0}{180}

\psarc(0.190909,0){0.00909091}{0}{180}

\psarc(0.226496,0){0.0042735}{0}{180}

\psarc(0.240385,0){0.00961538}{0}{180}

\psarc(0.303846,0){0.00384615}{0}{180}

\psarc(0.320513,0){0.0128205}{0}{180}

\psarc(0.345238,0){0.0119048}{0}{180}

\psarc(0.36039,0){0.00324675}{0}{180}

\psarc(0.449495,0){0.00505051}{0}{180}

\psarc(0.477273,0){0.0227273}{0}{180}

\psarc(0.522727,0){0.0227273}{0}{180}

\psarc(0.550505,0){0.00505051}{0}{180}

\psarc(0.63961,0){0.00324675}{0}{180}

\psarc(0.654762,0){0.0119048}{0}{180}

\psarc(0.679487,0){0.0128205}{0}{180}

\psarc(0.696154,0){0.00384615}{0}{180}

\psarc(0.759615,0){0.00961538}{0}{180}

\psarc(0.773504,0){0.0042735}{0}{180}

\psarc(0.809091,0){0.00909091}{0}{180}

\psarc(0.825758,0){0.00757576}{0}{180}

\psarc(0.845238,0){0.0119048}{0}{180}

\psarc(0.928571,0){0.0714286}{0}{180}

\psarc(0.0625,0){0.0625}{0}{180}

\psarc(0.133929,0){0.00892857}{0}{180}

\psarc(0.310096,0){0.00240385}{0}{180}

\psarc(0.322917,0){0.0104167}{0}{180}

\psarc(0.458042,0){0.0034965}{0}{180}

\psarc(0.480769,0){0.0192308}{0}{180}

\psarc(0.519231,0){0.0192308}{0}{180}

\psarc(0.541958,0){0.0034965}{0}{180}

\psarc(0.677083,0){0.0104167}{0}{180}

\psarc(0.689904,0){0.00240385}{0}{180}

\psarc(0.866071,0){0.00892857}{0}{180}

\psarc(0.9375,0){0.0625}{0}{180}

\psarc(0.0555556,0){0.0555556}{0}{180}

\psarc(0.118056,0){0.00694444}{0}{180}

\psarc(0.464103,0){0.0025641}{0}{180}

\psarc(0.483333,0){0.0166667}{0}{180}

\psarc(0.516667,0){0.0166667}{0}{180}

\psarc(0.535897,0){0.0025641}{0}{180}

\psarc(0.881944,0){0.00694444}{0}{180}

\psarc(0.944444,0){0.0555556}{0}{180}

\psarc(0.05,0){0.05}{0}{180}

\psarc(0.105556,0){0.00555556}{0}{180}

\psarc(0.468627,0){0.00196078}{0}{180}

\psarc(0.485294,0){0.0147059}{0}{180}

\psarc(0.514706,0){0.0147059}{0}{180}

\psarc(0.531373,0){0.00196078}{0}{180}

\psarc(0.894444,0){0.00555556}{0}{180}

\psarc(0.95,0){0.05}{0}{180}

\psarc(0.0454545,0){0.0454545}{0}{180}

\psarc(0.0954545,0){0.00454545}{0}{180}

\psarc(0.472136,0){0.00154799}{0}{180}

\psarc(0.486842,0){0.0131579}{0}{180}

\psarc(0.513158,0){0.0131579}{0}{180}

\psarc(0.527864,0){0.00154799}{0}{180}

\psarc(0.904545,0){0.00454545}{0}{180}

\psarc(0.954545,0){0.0454545}{0}{180}

\psarc(0.0416667,0){0.0416667}{0}{180}

\psarc(0.0871212,0){0.00378788}{0}{180}

\psarc(0.474937,0){0.00125313}{0}{180}

\psarc(0.488095,0){0.0119048}{0}{180}

\psarc(0.511905,0){0.0119048}{0}{180}

\psarc(0.525063,0){0.00125313}{0}{180}

\psarc(0.912879,0){0.00378788}{0}{180}

\psarc(0.958333,0){0.0416667}{0}{180}

\psarc(0.0384615,0){0.0384615}{0}{180}

\psarc(0.0801282,0){0.00320513}{0}{180}

\psarc(0.919872,0){0.00320513}{0}{180}

\psarc(0.961538,0){0.0384615}{0}{180}

\psarc(0.0357143,0){0.0357143}{0}{180}

\psarc(0.0741758,0){0.00274725}{0}{180}

\psarc(0.925824,0){0.00274725}{0}{180}

\psarc(0.964286,0){0.0357143}{0}{180}

\psarc(0.0333333,0){0.0333333}{0}{180}

\psarc(0.0690476,0){0.00238095}{0}{180}

\psarc(0.930952,0){0.00238095}{0}{180}

\psarc(0.966667,0){0.0333333}{0}{180}

\psarc(0.03125,0){0.03125}{0}{180}

\psarc(0.0645833,0){0.00208333}{0}{180}

\psarc(0.935417,0){0.00208333}{0}{180}

\psarc(0.96875,0){0.03125}{0}{180}

\psarc(0.0294118,0){0.0294118}{0}{180}

\psarc(0.0606618,0){0.00183824}{0}{180}

\psarc(0.939338,0){0.00183824}{0}{180}

\psarc(0.970588,0){0.0294118}{0}{180}

\psarc(0.0277778,0){0.0277778}{0}{180}

\psarc(0.0571895,0){0.00163399}{0}{180}

\psarc(0.94281,0){0.00163399}{0}{180}

\psarc(0.972222,0){0.0277778}{0}{180}

\psarc(0.0263158,0){0.0263158}{0}{180}

\psarc(0.0540936,0){0.00146199}{0}{180}

\psarc(0.945906,0){0.00146199}{0}{180}

\psarc(0.973684,0){0.0263158}{0}{180}

\psarc(0.025,0){0.025}{0}{180}

\psarc(0.0513158,0){0.00131579}{0}{180}

\psarc(0.948684,0){0.00131579}{0}{180}

\psarc(0.975,0){0.025}{0}{180}

\psarc(0.0238095,0){0.0238095}{0}{180}

\psarc(0.0488095,0){0.00119048}{0}{180}

\psarc(0.95119,0){0.00119048}{0}{180}

\psarc(0.97619,0){0.0238095}{0}{180}

\psarc(0.0227273,0){0.0227273}{0}{180}

\psarc(0.0465368,0){0.00108225}{0}{180}

\psarc(0.953463,0){0.00108225}{0}{180}

\psarc(0.977273,0){0.0227273}{0}{180}

\psarc(0.0217391,0){0.0217391}{0}{180}

\psarc(0.0444664,0){0.000988142}{0}{180}

\psarc(0.955534,0){0.000988142}{0}{180}

\psarc(0.978261,0){0.0217391}{0}{180}

\psarc(0.0208333,0){0.0208333}{0}{180}

\psarc(0.0425725,0){0.000905797}{0}{180}

\psarc(0.957428,0){0.000905797}{0}{180}

\psarc(0.979167,0){0.0208333}{0}{180}

\psarc(0.02,0){0.02}{0}{180}

\psarc(0.0408333,0){0.000833333}{0}{180}

\psarc(0.959167,0){0.000833333}{0}{180}

\psarc(0.98,0){0.02}{0}{180}

\psarc(0.0192308,0){0.0192308}{0}{180}

\psarc(0.0392308,0){0.000769231}{0}{180}

\psarc(0.960769,0){0.000769231}{0}{180}

\psarc(0.980769,0){0.0192308}{0}{180}

\psarc(0.0185185,0){0.0185185}{0}{180}

\psarc(0.0377493,0){0.000712251}{0}{180}

\psarc(0.962251,0){0.000712251}{0}{180}

\psarc(0.981481,0){0.0185185}{0}{180}

\psarc(0.0178571,0){0.0178571}{0}{180}

\psarc(0.0363757,0){0.000661376}{0}{180}

\psarc(0.963624,0){0.000661376}{0}{180}

\psarc(0.982143,0){0.0178571}{0}{180}

\psarc(0.0172414,0){0.0172414}{0}{180}

\psarc(0.0350985,0){0.000615764}{0}{180}

\psarc(0.964901,0){0.000615764}{0}{180}

\psarc(0.982759,0){0.0172414}{0}{180}

\psarc(0.0166667,0){0.0166667}{0}{180}

\psarc(0.033908,0){0.000574713}{0}{180}

\psarc(0.966092,0){0.000574713}{0}{180}

\psarc(0.983333,0){0.0166667}{0}{180}

\psarc(0.016129,0){0.016129}{0}{180}

\psarc(0.0327957,0){0.000537634}{0}{180}

\psarc(0.967204,0){0.000537634}{0}{180}

\psarc(0.983871,0){0.016129}{0}{180}

\psarc(0.015625,0){0.015625}{0}{180}

\psarc(0.031754,0){0.000504032}{0}{180}

\psarc(0.968246,0){0.000504032}{0}{180}

\psarc(0.984375,0){0.015625}{0}{180}

\psarc(0.0151515,0){0.0151515}{0}{180}

\psarc(0.0307765,0){0.000473485}{0}{180}

\psarc(0.969223,0){0.000473485}{0}{180}

\psarc(0.984848,0){0.0151515}{0}{180}

\psarc(0.0147059,0){0.0147059}{0}{180}

\psarc(0.0298574,0){0.000445633}{0}{180}

\psarc(0.970143,0){0.000445633}{0}{180}

\psarc(0.985294,0){0.0147059}{0}{180}

\psarc(0.0142857,0){0.0142857}{0}{180}

\psarc(0.0289916,0){0.000420168}{0}{180}

\psarc(0.971008,0){0.000420168}{0}{180}

\psarc(0.985714,0){0.0142857}{0}{180}

\psarc(0.0138889,0){0.0138889}{0}{180}

\psarc(0.0281746,0){0.000396825}{0}{180}

\psarc(0.971825,0){0.000396825}{0}{180}

\psarc(0.986111,0){0.0138889}{0}{180}

\psarc(0.0135135,0){0.0135135}{0}{180}

\psarc(0.0274024,0){0.000375375}{0}{180}

\psarc(0.972598,0){0.000375375}{0}{180}

\psarc(0.986486,0){0.0135135}{0}{180}

\psarc(0.0131579,0){0.0131579}{0}{180}

\psarc(0.0266714,0){0.000355619}{0}{180}

\psarc(0.973329,0){0.000355619}{0}{180}

\psarc(0.986842,0){0.0131579}{0}{180}

\psarc(0.0128205,0){0.0128205}{0}{180}

\psarc(0.0259784,0){0.000337382}{0}{180}

\psarc(0.974022,0){0.000337382}{0}{180}

\psarc(0.987179,0){0.0128205}{0}{180}

\psarc(0.0125,0){0.0125}{0}{180}

\psarc(0.0253205,0){0.000320513}{0}{180}

\psarc(0.974679,0){0.000320513}{0}{180}

\psarc(0.9875,0){0.0125}{0}{180}

\psarc[fillstyle=solid,fillcolor=red](0.3875,0){0.0125}{0}{180}

\psarc[fillstyle=solid,fillcolor=red](0.6125,0){0.0125}{0}{180}

\psarc[fillstyle=solid,fillcolor=red](0.211111,0){0.0111111}{0}{180}

\psarc[fillstyle=solid,fillcolor=red](0.261364,0){0.0113636}{0}{180}

\psarc[fillstyle=solid,fillcolor=red](0.279221,0){0.00649351}{0}{180}

\psarc[fillstyle=solid,fillcolor=red](0.292857,0){0.00714286}{0}{180}

\psarc[fillstyle=solid,fillcolor=red](0.369318,0){0.00568182}{0}{180}

\psarc[fillstyle=solid,fillcolor=red](0.408333,0){0.00833333}{0}{180}

\psarc[fillstyle=solid,fillcolor=red](0.422619,0){0.00595238}{0}{180}

\psarc[fillstyle=solid,fillcolor=red](0.436508,0){0.00793651}{0}{180}

\psarc[fillstyle=solid,fillcolor=red](0.563492,0){0.00793651}{0}{180}

\psarc[fillstyle=solid,fillcolor=red](0.577381,0){0.00595238}{0}{180}

\psarc[fillstyle=solid,fillcolor=red](0.591667,0){0.00833333}{0}{180}

\psarc[fillstyle=solid,fillcolor=red](0.630682,0){0.00568182}{0}{180}

\psarc[fillstyle=solid,fillcolor=red](0.707143,0){0.00714286}{0}{180}

\psarc[fillstyle=solid,fillcolor=red](0.720779,0){0.00649351}{0}{180}

\psarc[fillstyle=solid,fillcolor=red](0.738636,0){0.0113636}{0}{180}

\psarc[fillstyle=solid,fillcolor=red](0.788889,0){0.0111111}{0}{180}

\psarc[fillstyle=solid,fillcolor=red](0.154762,0){0.0119048}{0}{180}

\psarc[fillstyle=solid,fillcolor=red](0.174242,0){0.00757576}{0}{180}

\psarc[fillstyle=solid,fillcolor=red](0.190909,0){0.00909091}{0}{180}

\psarc[fillstyle=solid,fillcolor=red](0.226496,0){0.0042735}{0}{180}

\psarc[fillstyle=solid,fillcolor=red](0.240385,0){0.00961538}{0}{180}

\psarc[fillstyle=solid,fillcolor=red](0.303846,0){0.00384615}{0}{180}

\psarc[fillstyle=solid,fillcolor=red](0.345238,0){0.0119048}{0}{180}

\psarc[fillstyle=solid,fillcolor=red](0.36039,0){0.00324675}{0}{180}

\psarc[fillstyle=solid,fillcolor=red](0.449495,0){0.00505051}{0}{180}

\psarc[fillstyle=solid,fillcolor=red](0.550505,0){0.00505051}{0}{180}

\psarc[fillstyle=solid,fillcolor=red](0.63961,0){0.00324675}{0}{180}

\psarc[fillstyle=solid,fillcolor=red](0.654762,0){0.0119048}{0}{180}

\psarc[fillstyle=solid,fillcolor=red](0.696154,0){0.00384615}{0}{180}

\psarc[fillstyle=solid,fillcolor=red](0.759615,0){0.00961538}{0}{180}

\psarc[fillstyle=solid,fillcolor=red](0.773504,0){0.0042735}{0}{180}

\psarc[fillstyle=solid,fillcolor=red](0.809091,0){0.00909091}{0}{180}

\psarc[fillstyle=solid,fillcolor=red](0.825758,0){0.00757576}{0}{180}

\psarc[fillstyle=solid,fillcolor=red](0.845238,0){0.0119048}{0}{180}

\psarc[fillstyle=solid,fillcolor=red](0.133929,0){0.00892857}{0}{180}

\psarc[fillstyle=solid,fillcolor=red](0.310096,0){0.00240385}{0}{180}

\psarc[fillstyle=solid,fillcolor=red](0.322917,0){0.0104167}{0}{180}

\psarc[fillstyle=solid,fillcolor=red](0.458042,0){0.0034965}{0}{180}

\psarc[fillstyle=solid,fillcolor=red](0.541958,0){0.0034965}{0}{180}

\psarc[fillstyle=solid,fillcolor=red](0.677083,0){0.0104167}{0}{180}

\psarc[fillstyle=solid,fillcolor=red](0.689904,0){0.00240385}{0}{180}

\psarc[fillstyle=solid,fillcolor=red](0.866071,0){0.00892857}{0}{180}

\psarc[fillstyle=solid,fillcolor=red](0.118056,0){0.00694444}{0}{180}

\psarc[fillstyle=solid,fillcolor=red](0.464103,0){0.0025641}{0}{180}

\psarc[fillstyle=solid,fillcolor=red](0.535897,0){0.0025641}{0}{180}

\psarc[fillstyle=solid,fillcolor=red](0.881944,0){0.00694444}{0}{180}

\psarc[fillstyle=solid,fillcolor=red](0.105556,0){0.00555556}{0}{180}

\psarc[fillstyle=solid,fillcolor=red](0.468627,0){0.00196078}{0}{180}

\psarc[fillstyle=solid,fillcolor=red](0.531373,0){0.00196078}{0}{180}

\psarc[fillstyle=solid,fillcolor=red](0.894444,0){0.00555556}{0}{180}

\psarc[fillstyle=solid,fillcolor=red](0.0954545,0){0.00454545}{0}{180}

\psarc[fillstyle=solid,fillcolor=red](0.472136,0){0.00154799}{0}{180}

\psarc[fillstyle=solid,fillcolor=red](0.527864,0){0.00154799}{0}{180}

\psarc[fillstyle=solid,fillcolor=red](0.904545,0){0.00454545}{0}{180}

\psarc[fillstyle=solid,fillcolor=red](0.0871212,0){0.00378788}{0}{180}

\psarc[fillstyle=solid,fillcolor=red](0.474937,0){0.00125313}{0}{180}

\psarc[fillstyle=solid,fillcolor=red](0.488095,0){0.0119048}{0}{180}

\psarc[fillstyle=solid,fillcolor=red](0.511905,0){0.0119048}{0}{180}

\psarc[fillstyle=solid,fillcolor=red](0.525063,0){0.00125313}{0}{180}

\psarc[fillstyle=solid,fillcolor=red](0.912879,0){0.00378788}{0}{180}

\psarc[fillstyle=solid,fillcolor=red](0.0801282,0){0.00320513}{0}{180}

\psarc[fillstyle=solid,fillcolor=red](0.919872,0){0.00320513}{0}{180}

\psarc[fillstyle=solid,fillcolor=red](0.0741758,0){0.00274725}{0}{180}

\psarc[fillstyle=solid,fillcolor=red](0.925824,0){0.00274725}{0}{180}

\psarc[fillstyle=solid,fillcolor=red](0.0690476,0){0.00238095}{0}{180}

\psarc[fillstyle=solid,fillcolor=red](0.930952,0){0.00238095}{0}{180}

\psarc[fillstyle=solid,fillcolor=red](0.0645833,0){0.00208333}{0}{180}

\psarc[fillstyle=solid,fillcolor=red](0.935417,0){0.00208333}{0}{180}

\psarc[fillstyle=solid,fillcolor=red](0.0606618,0){0.00183824}{0}{180}

\psarc[fillstyle=solid,fillcolor=red](0.939338,0){0.00183824}{0}{180}

\psarc[fillstyle=solid,fillcolor=red](0.0571895,0){0.00163399}{0}{180}

\psarc[fillstyle=solid,fillcolor=red](0.94281,0){0.00163399}{0}{180}

\psarc[fillstyle=solid,fillcolor=red](0.0540936,0){0.00146199}{0}{180}

\psarc[fillstyle=solid,fillcolor=red](0.945906,0){0.00146199}{0}{180}

\psarc[fillstyle=solid,fillcolor=red](0.0513158,0){0.00131579}{0}{180}

\psarc[fillstyle=solid,fillcolor=red](0.948684,0){0.00131579}{0}{180}

\psarc[fillstyle=solid,fillcolor=red](0.0488095,0){0.00119048}{0}{180}

\psarc[fillstyle=solid,fillcolor=red](0.95119,0){0.00119048}{0}{180}

\psarc[fillstyle=solid,fillcolor=red](0.0465368,0){0.00108225}{0}{180}

\psarc[fillstyle=solid,fillcolor=red](0.953463,0){0.00108225}{0}{180}

\psarc[fillstyle=solid,fillcolor=red](0.0444664,0){0.000988142}{0}{180}

\psarc[fillstyle=solid,fillcolor=red](0.955534,0){0.000988142}{0}{180}

\psarc[fillstyle=solid,fillcolor=red](0.0425725,0){0.000905797}{0}{180}

\psarc[fillstyle=solid,fillcolor=red](0.957428,0){0.000905797}{0}{180}

\psarc[fillstyle=solid,fillcolor=red](0.0408333,0){0.000833333}{0}{180}

\psarc[fillstyle=solid,fillcolor=red](0.959167,0){0.000833333}{0}{180}

\psarc[fillstyle=solid,fillcolor=red](0.0392308,0){0.000769231}{0}{180}

\psarc[fillstyle=solid,fillcolor=red](0.960769,0){0.000769231}{0}{180}

\psarc[fillstyle=solid,fillcolor=red](0.0377493,0){0.000712251}{0}{180}

\psarc[fillstyle=solid,fillcolor=red](0.962251,0){0.000712251}{0}{180}

\psarc[fillstyle=solid,fillcolor=red](0.0363757,0){0.000661376}{0}{180}

\psarc[fillstyle=solid,fillcolor=red](0.963624,0){0.000661376}{0}{180}

\psarc[fillstyle=solid,fillcolor=red](0.0350985,0){0.000615764}{0}{180}

\psarc[fillstyle=solid,fillcolor=red](0.964901,0){0.000615764}{0}{180}

\psarc[fillstyle=solid,fillcolor=red](0.033908,0){0.000574713}{0}{180}

\psarc[fillstyle=solid,fillcolor=red](0.966092,0){0.000574713}{0}{180}

\psarc[fillstyle=solid,fillcolor=red](0.0327957,0){0.000537634}{0}{180}

\psarc[fillstyle=solid,fillcolor=red](0.967204,0){0.000537634}{0}{180}

\psarc[fillstyle=solid,fillcolor=red](0.031754,0){0.000504032}{0}{180}

\psarc[fillstyle=solid,fillcolor=red](0.968246,0){0.000504032}{0}{180}

\psarc[fillstyle=solid,fillcolor=red](0.0307765,0){0.000473485}{0}{180}

\psarc[fillstyle=solid,fillcolor=red](0.969223,0){0.000473485}{0}{180}

\psarc[fillstyle=solid,fillcolor=red](0.0298574,0){0.000445633}{0}{180}

\psarc[fillstyle=solid,fillcolor=red](0.970143,0){0.000445633}{0}{180}

\psarc[fillstyle=solid,fillcolor=red](0.0289916,0){0.000420168}{0}{180}

\psarc[fillstyle=solid,fillcolor=red](0.971008,0){0.000420168}{0}{180}

\psarc[fillstyle=solid,fillcolor=red](0.0281746,0){0.000396825}{0}{180}

\psarc[fillstyle=solid,fillcolor=red](0.971825,0){0.000396825}{0}{180}

\psarc[fillstyle=solid,fillcolor=red](0.0274024,0){0.000375375}{0}{180}

\psarc[fillstyle=solid,fillcolor=red](0.972598,0){0.000375375}{0}{180}

\psarc[fillstyle=solid,fillcolor=red](0.0266714,0){0.000355619}{0}{180}

\psarc[fillstyle=solid,fillcolor=red](0.973329,0){0.000355619}{0}{180}

\psarc[fillstyle=solid,fillcolor=red](0.0259784,0){0.000337382}{0}{180}

\psarc[fillstyle=solid,fillcolor=red](0.974022,0){0.000337382}{0}{180}

\psarc[fillstyle=solid,fillcolor=red](0.0125,0){0.0125}{0}{180}

\psarc[fillstyle=solid,fillcolor=red](0.0253205,0){0.000320513}{0}{180}

\psarc[fillstyle=solid,fillcolor=red](0.974679,0){0.000320513}{0}{180}

\psarc[fillstyle=solid,fillcolor=red](0.9875,0){0.0125}{0}{180}
\end{pspicture}}

\rput(0.5,0.35){%
        \begin{pspicture}(0,0)(1,0.7)
\psarc(0.5,0){0.5}{0}{180}

\psarc(0.25,0){0.25}{0}{180}

\psarc(0.75,0){0.25}{0}{180}

\psarc(0.166667,0){0.166667}{0}{180}

\psarc(0.416667,0){0.0833333}{0}{180}

\psarc(0.583333,0){0.0833333}{0}{180}

\psarc(0.833333,0){0.166667}{0}{180}

\psarc(0.125,0){0.125}{0}{180}

\psarc(0.291667,0){0.0416667}{0}{180}

\psarc(0.366667,0){0.0333333}{0}{180}

\psarc(0.45,0){0.05}{0}{180}

\psarc(0.55,0){0.05}{0}{180}

\psarc(0.633333,0){0.0333333}{0}{180}

\psarc(0.708333,0){0.0416667}{0}{180}

\psarc(0.875,0){0.125}{0}{180}

\psarc(0.1,0){0.1}{0}{180}

\psarc(0.225,0){0.025}{0}{180}

\psarc(0.267857,0){0.0178571}{0}{180}

\psarc(0.309524,0){0.0238095}{0}{180}

\psarc(0.354167,0){0.0208333}{0}{180}

\psarc(0.3875,0){0.0125}{0}{180}

\psarc(0.414286,0){0.0142857}{0}{180}

\psarc(0.464286,0){0.0357143}{0}{180}

\psarc(0.535714,0){0.0357143}{0}{180}

\psarc(0.585714,0){0.0142857}{0}{180}

\psarc(0.6125,0){0.0125}{0}{180}

\psarc(0.645833,0){0.0208333}{0}{180}

\psarc(0.690476,0){0.0238095}{0}{180}

\psarc(0.732143,0){0.0178571}{0}{180}

\psarc(0.775,0){0.025}{0}{180}

\psarc(0.9,0){0.1}{0}{180}

\psarc(0.0833333,0){0.0833333}{0}{180}

\psarc(0.183333,0){0.0166667}{0}{180}

\psarc(0.211111,0){0.0111111}{0}{180}

\psarc(0.236111,0){0.0138889}{0}{180}

\psarc(0.261364,0){0.0113636}{0}{180}

\psarc(0.279221,0){0.00649351}{0}{180}

\psarc(0.292857,0){0.00714286}{0}{180}

\psarc(0.316667,0){0.0166667}{0}{180}

\psarc(0.348485,0){0.0151515}{0}{180}

\psarc(0.369318,0){0.00568182}{0}{180}

\psarc(0.408333,0){0.00833333}{0}{180}

\psarc(0.422619,0){0.00595238}{0}{180}

\psarc(0.436508,0){0.00793651}{0}{180}

\psarc(0.472222,0){0.0277778}{0}{180}

\psarc(0.527778,0){0.0277778}{0}{180}

\psarc(0.563492,0){0.00793651}{0}{180}

\psarc(0.577381,0){0.00595238}{0}{180}

\psarc(0.591667,0){0.00833333}{0}{180}

\psarc(0.630682,0){0.00568182}{0}{180}

\psarc(0.651515,0){0.0151515}{0}{180}

\psarc(0.683333,0){0.0166667}{0}{180}

\psarc(0.707143,0){0.00714286}{0}{180}

\psarc(0.720779,0){0.00649351}{0}{180}

\psarc(0.738636,0){0.0113636}{0}{180}

\psarc(0.763889,0){0.0138889}{0}{180}

\psarc(0.788889,0){0.0111111}{0}{180}

\psarc(0.816667,0){0.0166667}{0}{180}

\psarc(0.916667,0){0.0833333}{0}{180}

\psarc(0.0714286,0){0.0714286}{0}{180}

\psarc(0.154762,0){0.0119048}{0}{180}

\psarc(0.174242,0){0.00757576}{0}{180}

\psarc(0.190909,0){0.00909091}{0}{180}

\psarc(0.226496,0){0.0042735}{0}{180}

\psarc(0.240385,0){0.00961538}{0}{180}

\psarc(0.303846,0){0.00384615}{0}{180}

\psarc(0.320513,0){0.0128205}{0}{180}

\psarc(0.345238,0){0.0119048}{0}{180}

\psarc(0.36039,0){0.00324675}{0}{180}

\psarc(0.449495,0){0.00505051}{0}{180}

\psarc(0.477273,0){0.0227273}{0}{180}

\psarc(0.522727,0){0.0227273}{0}{180}

\psarc(0.550505,0){0.00505051}{0}{180}

\psarc(0.63961,0){0.00324675}{0}{180}

\psarc(0.654762,0){0.0119048}{0}{180}

\psarc(0.679487,0){0.0128205}{0}{180}

\psarc(0.696154,0){0.00384615}{0}{180}

\psarc(0.759615,0){0.00961538}{0}{180}

\psarc(0.773504,0){0.0042735}{0}{180}

\psarc(0.809091,0){0.00909091}{0}{180}

\psarc(0.825758,0){0.00757576}{0}{180}

\psarc(0.845238,0){0.0119048}{0}{180}

\psarc(0.928571,0){0.0714286}{0}{180}

\psarc(0.0625,0){0.0625}{0}{180}

\psarc(0.133929,0){0.00892857}{0}{180}

\psarc(0.310096,0){0.00240385}{0}{180}

\psarc(0.322917,0){0.0104167}{0}{180}

\psarc(0.458042,0){0.0034965}{0}{180}

\psarc(0.480769,0){0.0192308}{0}{180}

\psarc(0.519231,0){0.0192308}{0}{180}

\psarc(0.541958,0){0.0034965}{0}{180}

\psarc(0.677083,0){0.0104167}{0}{180}

\psarc(0.689904,0){0.00240385}{0}{180}

\psarc(0.866071,0){0.00892857}{0}{180}

\psarc(0.9375,0){0.0625}{0}{180}

\psarc(0.0555556,0){0.0555556}{0}{180}

\psarc(0.118056,0){0.00694444}{0}{180}

\psarc(0.464103,0){0.0025641}{0}{180}

\psarc(0.483333,0){0.0166667}{0}{180}

\psarc(0.516667,0){0.0166667}{0}{180}

\psarc(0.535897,0){0.0025641}{0}{180}

\psarc(0.881944,0){0.00694444}{0}{180}

\psarc(0.944444,0){0.0555556}{0}{180}

\psarc(0.05,0){0.05}{0}{180}

\psarc(0.105556,0){0.00555556}{0}{180}

\psarc(0.468627,0){0.00196078}{0}{180}

\psarc(0.485294,0){0.0147059}{0}{180}

\psarc(0.514706,0){0.0147059}{0}{180}

\psarc(0.531373,0){0.00196078}{0}{180}

\psarc(0.894444,0){0.00555556}{0}{180}

\psarc(0.95,0){0.05}{0}{180}

\psarc(0.0454545,0){0.0454545}{0}{180}

\psarc(0.0954545,0){0.00454545}{0}{180}

\psarc(0.472136,0){0.00154799}{0}{180}

\psarc(0.486842,0){0.0131579}{0}{180}

\psarc(0.513158,0){0.0131579}{0}{180}

\psarc(0.527864,0){0.00154799}{0}{180}

\psarc(0.904545,0){0.00454545}{0}{180}

\psarc(0.954545,0){0.0454545}{0}{180}

\psarc(0.0416667,0){0.0416667}{0}{180}

\psarc(0.0871212,0){0.00378788}{0}{180}

\psarc(0.474937,0){0.00125313}{0}{180}

\psarc(0.488095,0){0.0119048}{0}{180}

\psarc(0.511905,0){0.0119048}{0}{180}

\psarc(0.525063,0){0.00125313}{0}{180}

\psarc(0.912879,0){0.00378788}{0}{180}

\psarc(0.958333,0){0.0416667}{0}{180}

\psarc(0.0384615,0){0.0384615}{0}{180}

\psarc(0.0801282,0){0.00320513}{0}{180}

\psarc(0.919872,0){0.00320513}{0}{180}

\psarc(0.961538,0){0.0384615}{0}{180}

\psarc(0.0357143,0){0.0357143}{0}{180}

\psarc(0.0741758,0){0.00274725}{0}{180}

\psarc(0.925824,0){0.00274725}{0}{180}

\psarc(0.964286,0){0.0357143}{0}{180}

\psarc(0.0333333,0){0.0333333}{0}{180}

\psarc(0.0690476,0){0.00238095}{0}{180}

\psarc(0.930952,0){0.00238095}{0}{180}

\psarc(0.966667,0){0.0333333}{0}{180}

\psarc(0.03125,0){0.03125}{0}{180}

\psarc(0.0645833,0){0.00208333}{0}{180}

\psarc(0.935417,0){0.00208333}{0}{180}

\psarc(0.96875,0){0.03125}{0}{180}

\psarc(0.0294118,0){0.0294118}{0}{180}

\psarc(0.0606618,0){0.00183824}{0}{180}

\psarc(0.939338,0){0.00183824}{0}{180}

\psarc(0.970588,0){0.0294118}{0}{180}

\psarc(0.0277778,0){0.0277778}{0}{180}

\psarc(0.0571895,0){0.00163399}{0}{180}

\psarc(0.94281,0){0.00163399}{0}{180}

\psarc(0.972222,0){0.0277778}{0}{180}

\psarc(0.0263158,0){0.0263158}{0}{180}

\psarc(0.0540936,0){0.00146199}{0}{180}

\psarc(0.945906,0){0.00146199}{0}{180}

\psarc(0.973684,0){0.0263158}{0}{180}

\psarc(0.025,0){0.025}{0}{180}

\psarc(0.0513158,0){0.00131579}{0}{180}

\psarc(0.948684,0){0.00131579}{0}{180}

\psarc(0.975,0){0.025}{0}{180}

\psarc(0.0238095,0){0.0238095}{0}{180}

\psarc(0.0488095,0){0.00119048}{0}{180}

\psarc(0.95119,0){0.00119048}{0}{180}

\psarc(0.97619,0){0.0238095}{0}{180}

\psarc(0.0227273,0){0.0227273}{0}{180}

\psarc(0.0465368,0){0.00108225}{0}{180}

\psarc(0.953463,0){0.00108225}{0}{180}

\psarc(0.977273,0){0.0227273}{0}{180}

\psarc(0.0217391,0){0.0217391}{0}{180}

\psarc(0.0444664,0){0.000988142}{0}{180}

\psarc(0.955534,0){0.000988142}{0}{180}

\psarc(0.978261,0){0.0217391}{0}{180}

\psarc(0.0208333,0){0.0208333}{0}{180}

\psarc(0.0425725,0){0.000905797}{0}{180}

\psarc(0.957428,0){0.000905797}{0}{180}

\psarc(0.979167,0){0.0208333}{0}{180}

\psarc(0.02,0){0.02}{0}{180}

\psarc(0.0408333,0){0.000833333}{0}{180}

\psarc(0.959167,0){0.000833333}{0}{180}

\psarc(0.98,0){0.02}{0}{180}

\psarc(0.0192308,0){0.0192308}{0}{180}

\psarc(0.0392308,0){0.000769231}{0}{180}

\psarc(0.960769,0){0.000769231}{0}{180}

\psarc(0.980769,0){0.0192308}{0}{180}

\psarc(0.0185185,0){0.0185185}{0}{180}

\psarc(0.0377493,0){0.000712251}{0}{180}

\psarc(0.962251,0){0.000712251}{0}{180}

\psarc(0.981481,0){0.0185185}{0}{180}

\psarc(0.0178571,0){0.0178571}{0}{180}

\psarc(0.0363757,0){0.000661376}{0}{180}

\psarc(0.963624,0){0.000661376}{0}{180}

\psarc(0.982143,0){0.0178571}{0}{180}

\psarc(0.0172414,0){0.0172414}{0}{180}

\psarc(0.0350985,0){0.000615764}{0}{180}

\psarc(0.964901,0){0.000615764}{0}{180}

\psarc(0.982759,0){0.0172414}{0}{180}

\psarc(0.0166667,0){0.0166667}{0}{180}

\psarc(0.033908,0){0.000574713}{0}{180}

\psarc(0.966092,0){0.000574713}{0}{180}

\psarc(0.983333,0){0.0166667}{0}{180}

\psarc(0.016129,0){0.016129}{0}{180}

\psarc(0.0327957,0){0.000537634}{0}{180}

\psarc(0.967204,0){0.000537634}{0}{180}

\psarc(0.983871,0){0.016129}{0}{180}

\psarc(0.015625,0){0.015625}{0}{180}

\psarc(0.031754,0){0.000504032}{0}{180}

\psarc(0.968246,0){0.000504032}{0}{180}

\psarc(0.984375,0){0.015625}{0}{180}

\psarc(0.0151515,0){0.0151515}{0}{180}

\psarc(0.0307765,0){0.000473485}{0}{180}

\psarc(0.969223,0){0.000473485}{0}{180}

\psarc(0.984848,0){0.0151515}{0}{180}

\psarc(0.0147059,0){0.0147059}{0}{180}

\psarc(0.0298574,0){0.000445633}{0}{180}

\psarc(0.970143,0){0.000445633}{0}{180}

\psarc(0.985294,0){0.0147059}{0}{180}

\psarc(0.0142857,0){0.0142857}{0}{180}

\psarc(0.0289916,0){0.000420168}{0}{180}

\psarc(0.971008,0){0.000420168}{0}{180}

\psarc(0.985714,0){0.0142857}{0}{180}

\psarc(0.0138889,0){0.0138889}{0}{180}

\psarc(0.0281746,0){0.000396825}{0}{180}

\psarc(0.971825,0){0.000396825}{0}{180}

\psarc(0.986111,0){0.0138889}{0}{180}

\psarc(0.0135135,0){0.0135135}{0}{180}

\psarc(0.0274024,0){0.000375375}{0}{180}

\psarc(0.972598,0){0.000375375}{0}{180}

\psarc(0.986486,0){0.0135135}{0}{180}

\psarc(0.0131579,0){0.0131579}{0}{180}

\psarc(0.0266714,0){0.000355619}{0}{180}

\psarc(0.973329,0){0.000355619}{0}{180}

\psarc(0.986842,0){0.0131579}{0}{180}

\psarc(0.0128205,0){0.0128205}{0}{180}

\psarc(0.0259784,0){0.000337382}{0}{180}

\psarc(0.974022,0){0.000337382}{0}{180}

\psarc(0.987179,0){0.0128205}{0}{180}

\psarc(0.0125,0){0.0125}{0}{180}

\psarc(0.0253205,0){0.000320513}{0}{180}

\psarc(0.974679,0){0.000320513}{0}{180}

\psarc(0.9875,0){0.0125}{0}{180}

\psarc[fillstyle=solid,fillcolor=red](0.3875,0){0.0125}{0}{180}

\psarc[fillstyle=solid,fillcolor=red](0.6125,0){0.0125}{0}{180}

\psarc[fillstyle=solid,fillcolor=red](0.211111,0){0.0111111}{0}{180}

\psarc[fillstyle=solid,fillcolor=red](0.261364,0){0.0113636}{0}{180}

\psarc[fillstyle=solid,fillcolor=red](0.279221,0){0.00649351}{0}{180}

\psarc[fillstyle=solid,fillcolor=red](0.292857,0){0.00714286}{0}{180}

\psarc[fillstyle=solid,fillcolor=red](0.369318,0){0.00568182}{0}{180}

\psarc[fillstyle=solid,fillcolor=red](0.408333,0){0.00833333}{0}{180}

\psarc[fillstyle=solid,fillcolor=red](0.422619,0){0.00595238}{0}{180}

\psarc[fillstyle=solid,fillcolor=red](0.436508,0){0.00793651}{0}{180}

\psarc[fillstyle=solid,fillcolor=red](0.563492,0){0.00793651}{0}{180}

\psarc[fillstyle=solid,fillcolor=red](0.577381,0){0.00595238}{0}{180}

\psarc[fillstyle=solid,fillcolor=red](0.591667,0){0.00833333}{0}{180}

\psarc[fillstyle=solid,fillcolor=red](0.630682,0){0.00568182}{0}{180}

\psarc[fillstyle=solid,fillcolor=red](0.707143,0){0.00714286}{0}{180}

\psarc[fillstyle=solid,fillcolor=red](0.720779,0){0.00649351}{0}{180}

\psarc[fillstyle=solid,fillcolor=red](0.738636,0){0.0113636}{0}{180}

\psarc[fillstyle=solid,fillcolor=red](0.788889,0){0.0111111}{0}{180}

\psarc[fillstyle=solid,fillcolor=red](0.154762,0){0.0119048}{0}{180}

\psarc[fillstyle=solid,fillcolor=red](0.174242,0){0.00757576}{0}{180}

\psarc[fillstyle=solid,fillcolor=red](0.190909,0){0.00909091}{0}{180}

\psarc[fillstyle=solid,fillcolor=red](0.226496,0){0.0042735}{0}{180}

\psarc[fillstyle=solid,fillcolor=red](0.240385,0){0.00961538}{0}{180}

\psarc[fillstyle=solid,fillcolor=red](0.303846,0){0.00384615}{0}{180}

\psarc[fillstyle=solid,fillcolor=red](0.345238,0){0.0119048}{0}{180}

\psarc[fillstyle=solid,fillcolor=red](0.36039,0){0.00324675}{0}{180}

\psarc[fillstyle=solid,fillcolor=red](0.449495,0){0.00505051}{0}{180}

\psarc[fillstyle=solid,fillcolor=red](0.550505,0){0.00505051}{0}{180}

\psarc[fillstyle=solid,fillcolor=red](0.63961,0){0.00324675}{0}{180}

\psarc[fillstyle=solid,fillcolor=red](0.654762,0){0.0119048}{0}{180}

\psarc[fillstyle=solid,fillcolor=red](0.696154,0){0.00384615}{0}{180}

\psarc[fillstyle=solid,fillcolor=red](0.759615,0){0.00961538}{0}{180}

\psarc[fillstyle=solid,fillcolor=red](0.773504,0){0.0042735}{0}{180}

\psarc[fillstyle=solid,fillcolor=red](0.809091,0){0.00909091}{0}{180}

\psarc[fillstyle=solid,fillcolor=red](0.825758,0){0.00757576}{0}{180}

\psarc[fillstyle=solid,fillcolor=red](0.845238,0){0.0119048}{0}{180}

\psarc[fillstyle=solid,fillcolor=red](0.133929,0){0.00892857}{0}{180}

\psarc[fillstyle=solid,fillcolor=red](0.310096,0){0.00240385}{0}{180}

\psarc[fillstyle=solid,fillcolor=red](0.322917,0){0.0104167}{0}{180}

\psarc[fillstyle=solid,fillcolor=red](0.458042,0){0.0034965}{0}{180}

\psarc[fillstyle=solid,fillcolor=red](0.541958,0){0.0034965}{0}{180}

\psarc[fillstyle=solid,fillcolor=red](0.677083,0){0.0104167}{0}{180}

\psarc[fillstyle=solid,fillcolor=red](0.689904,0){0.00240385}{0}{180}

\psarc[fillstyle=solid,fillcolor=red](0.866071,0){0.00892857}{0}{180}

\psarc[fillstyle=solid,fillcolor=red](0.118056,0){0.00694444}{0}{180}

\psarc[fillstyle=solid,fillcolor=red](0.464103,0){0.0025641}{0}{180}

\psarc[fillstyle=solid,fillcolor=red](0.535897,0){0.0025641}{0}{180}

\psarc[fillstyle=solid,fillcolor=red](0.881944,0){0.00694444}{0}{180}

\psarc[fillstyle=solid,fillcolor=red](0.105556,0){0.00555556}{0}{180}

\psarc[fillstyle=solid,fillcolor=red](0.468627,0){0.00196078}{0}{180}

\psarc[fillstyle=solid,fillcolor=red](0.531373,0){0.00196078}{0}{180}

\psarc[fillstyle=solid,fillcolor=red](0.894444,0){0.00555556}{0}{180}

\psarc[fillstyle=solid,fillcolor=red](0.0954545,0){0.00454545}{0}{180}

\psarc[fillstyle=solid,fillcolor=red](0.472136,0){0.00154799}{0}{180}

\psarc[fillstyle=solid,fillcolor=red](0.527864,0){0.00154799}{0}{180}

\psarc[fillstyle=solid,fillcolor=red](0.904545,0){0.00454545}{0}{180}

\psarc[fillstyle=solid,fillcolor=red](0.0871212,0){0.00378788}{0}{180}

\psarc[fillstyle=solid,fillcolor=red](0.474937,0){0.00125313}{0}{180}

\psarc[fillstyle=solid,fillcolor=red](0.488095,0){0.0119048}{0}{180}

\psarc[fillstyle=solid,fillcolor=red](0.511905,0){0.0119048}{0}{180}

\psarc[fillstyle=solid,fillcolor=red](0.525063,0){0.00125313}{0}{180}

\psarc[fillstyle=solid,fillcolor=red](0.912879,0){0.00378788}{0}{180}

\psarc[fillstyle=solid,fillcolor=red](0.0801282,0){0.00320513}{0}{180}

\psarc[fillstyle=solid,fillcolor=red](0.919872,0){0.00320513}{0}{180}

\psarc[fillstyle=solid,fillcolor=red](0.0741758,0){0.00274725}{0}{180}

\psarc[fillstyle=solid,fillcolor=red](0.925824,0){0.00274725}{0}{180}

\psarc[fillstyle=solid,fillcolor=red](0.0690476,0){0.00238095}{0}{180}

\psarc[fillstyle=solid,fillcolor=red](0.930952,0){0.00238095}{0}{180}

\psarc[fillstyle=solid,fillcolor=red](0.0645833,0){0.00208333}{0}{180}

\psarc[fillstyle=solid,fillcolor=red](0.935417,0){0.00208333}{0}{180}

\psarc[fillstyle=solid,fillcolor=red](0.0606618,0){0.00183824}{0}{180}

\psarc[fillstyle=solid,fillcolor=red](0.939338,0){0.00183824}{0}{180}

\psarc[fillstyle=solid,fillcolor=red](0.0571895,0){0.00163399}{0}{180}

\psarc[fillstyle=solid,fillcolor=red](0.94281,0){0.00163399}{0}{180}

\psarc[fillstyle=solid,fillcolor=red](0.0540936,0){0.00146199}{0}{180}

\psarc[fillstyle=solid,fillcolor=red](0.945906,0){0.00146199}{0}{180}

\psarc[fillstyle=solid,fillcolor=red](0.0513158,0){0.00131579}{0}{180}

\psarc[fillstyle=solid,fillcolor=red](0.948684,0){0.00131579}{0}{180}

\psarc[fillstyle=solid,fillcolor=red](0.0488095,0){0.00119048}{0}{180}

\psarc[fillstyle=solid,fillcolor=red](0.95119,0){0.00119048}{0}{180}

\psarc[fillstyle=solid,fillcolor=red](0.0465368,0){0.00108225}{0}{180}

\psarc[fillstyle=solid,fillcolor=red](0.953463,0){0.00108225}{0}{180}

\psarc[fillstyle=solid,fillcolor=red](0.0444664,0){0.000988142}{0}{180}

\psarc[fillstyle=solid,fillcolor=red](0.955534,0){0.000988142}{0}{180}

\psarc[fillstyle=solid,fillcolor=red](0.0425725,0){0.000905797}{0}{180}

\psarc[fillstyle=solid,fillcolor=red](0.957428,0){0.000905797}{0}{180}

\psarc[fillstyle=solid,fillcolor=red](0.0408333,0){0.000833333}{0}{180}

\psarc[fillstyle=solid,fillcolor=red](0.959167,0){0.000833333}{0}{180}

\psarc[fillstyle=solid,fillcolor=red](0.0392308,0){0.000769231}{0}{180}

\psarc[fillstyle=solid,fillcolor=red](0.960769,0){0.000769231}{0}{180}

\psarc[fillstyle=solid,fillcolor=red](0.0377493,0){0.000712251}{0}{180}

\psarc[fillstyle=solid,fillcolor=red](0.962251,0){0.000712251}{0}{180}

\psarc[fillstyle=solid,fillcolor=red](0.0363757,0){0.000661376}{0}{180}

\psarc[fillstyle=solid,fillcolor=red](0.963624,0){0.000661376}{0}{180}

\psarc[fillstyle=solid,fillcolor=red](0.0350985,0){0.000615764}{0}{180}

\psarc[fillstyle=solid,fillcolor=red](0.964901,0){0.000615764}{0}{180}

\psarc[fillstyle=solid,fillcolor=red](0.033908,0){0.000574713}{0}{180}

\psarc[fillstyle=solid,fillcolor=red](0.966092,0){0.000574713}{0}{180}

\psarc[fillstyle=solid,fillcolor=red](0.0327957,0){0.000537634}{0}{180}

\psarc[fillstyle=solid,fillcolor=red](0.967204,0){0.000537634}{0}{180}

\psarc[fillstyle=solid,fillcolor=red](0.031754,0){0.000504032}{0}{180}

\psarc[fillstyle=solid,fillcolor=red](0.968246,0){0.000504032}{0}{180}

\psarc[fillstyle=solid,fillcolor=red](0.0307765,0){0.000473485}{0}{180}

\psarc[fillstyle=solid,fillcolor=red](0.969223,0){0.000473485}{0}{180}

\psarc[fillstyle=solid,fillcolor=red](0.0298574,0){0.000445633}{0}{180}

\psarc[fillstyle=solid,fillcolor=red](0.970143,0){0.000445633}{0}{180}

\psarc[fillstyle=solid,fillcolor=red](0.0289916,0){0.000420168}{0}{180}

\psarc[fillstyle=solid,fillcolor=red](0.971008,0){0.000420168}{0}{180}

\psarc[fillstyle=solid,fillcolor=red](0.0281746,0){0.000396825}{0}{180}

\psarc[fillstyle=solid,fillcolor=red](0.971825,0){0.000396825}{0}{180}

\psarc[fillstyle=solid,fillcolor=red](0.0274024,0){0.000375375}{0}{180}

\psarc[fillstyle=solid,fillcolor=red](0.972598,0){0.000375375}{0}{180}

\psarc[fillstyle=solid,fillcolor=red](0.0266714,0){0.000355619}{0}{180}

\psarc[fillstyle=solid,fillcolor=red](0.973329,0){0.000355619}{0}{180}

\psarc[fillstyle=solid,fillcolor=red](0.0259784,0){0.000337382}{0}{180}

\psarc[fillstyle=solid,fillcolor=red](0.974022,0){0.000337382}{0}{180}

\psarc[fillstyle=solid,fillcolor=red](0.0125,0){0.0125}{0}{180}

\psarc[fillstyle=solid,fillcolor=red](0.0253205,0){0.000320513}{0}{180}

\psarc[fillstyle=solid,fillcolor=red](0.974679,0){0.000320513}{0}{180}

\psarc[fillstyle=solid,fillcolor=red](0.9875,0){0.0125}{0}{180}
\end{pspicture}}

\rput(1.5,0.35){%
        \begin{pspicture}(0,0)(1,0.7)
\psarc(0.5,0){0.5}{0}{180}

\psarc(0.25,0){0.25}{0}{180}

\psarc(0.75,0){0.25}{0}{180}

\psarc(0.166667,0){0.166667}{0}{180}

\psarc(0.416667,0){0.0833333}{0}{180}

\psarc(0.583333,0){0.0833333}{0}{180}

\psarc(0.833333,0){0.166667}{0}{180}

\psarc(0.125,0){0.125}{0}{180}

\psarc(0.291667,0){0.0416667}{0}{180}

\psarc(0.366667,0){0.0333333}{0}{180}

\psarc(0.45,0){0.05}{0}{180}

\psarc(0.55,0){0.05}{0}{180}

\psarc(0.633333,0){0.0333333}{0}{180}

\psarc(0.708333,0){0.0416667}{0}{180}

\psarc(0.875,0){0.125}{0}{180}

\psarc(0.1,0){0.1}{0}{180}

\psarc(0.225,0){0.025}{0}{180}

\psarc(0.267857,0){0.0178571}{0}{180}

\psarc(0.309524,0){0.0238095}{0}{180}

\psarc(0.354167,0){0.0208333}{0}{180}

\psarc(0.3875,0){0.0125}{0}{180}

\psarc(0.414286,0){0.0142857}{0}{180}

\psarc(0.464286,0){0.0357143}{0}{180}

\psarc(0.535714,0){0.0357143}{0}{180}

\psarc(0.585714,0){0.0142857}{0}{180}

\psarc(0.6125,0){0.0125}{0}{180}

\psarc(0.645833,0){0.0208333}{0}{180}

\psarc(0.690476,0){0.0238095}{0}{180}

\psarc(0.732143,0){0.0178571}{0}{180}

\psarc(0.775,0){0.025}{0}{180}

\psarc(0.9,0){0.1}{0}{180}

\psarc(0.0833333,0){0.0833333}{0}{180}

\psarc(0.183333,0){0.0166667}{0}{180}

\psarc(0.211111,0){0.0111111}{0}{180}

\psarc(0.236111,0){0.0138889}{0}{180}

\psarc(0.261364,0){0.0113636}{0}{180}

\psarc(0.279221,0){0.00649351}{0}{180}

\psarc(0.292857,0){0.00714286}{0}{180}

\psarc(0.316667,0){0.0166667}{0}{180}

\psarc(0.348485,0){0.0151515}{0}{180}

\psarc(0.369318,0){0.00568182}{0}{180}

\psarc(0.408333,0){0.00833333}{0}{180}

\psarc(0.422619,0){0.00595238}{0}{180}

\psarc(0.436508,0){0.00793651}{0}{180}

\psarc(0.472222,0){0.0277778}{0}{180}

\psarc(0.527778,0){0.0277778}{0}{180}

\psarc(0.563492,0){0.00793651}{0}{180}

\psarc(0.577381,0){0.00595238}{0}{180}

\psarc(0.591667,0){0.00833333}{0}{180}

\psarc(0.630682,0){0.00568182}{0}{180}

\psarc(0.651515,0){0.0151515}{0}{180}

\psarc(0.683333,0){0.0166667}{0}{180}

\psarc(0.707143,0){0.00714286}{0}{180}

\psarc(0.720779,0){0.00649351}{0}{180}

\psarc(0.738636,0){0.0113636}{0}{180}

\psarc(0.763889,0){0.0138889}{0}{180}

\psarc(0.788889,0){0.0111111}{0}{180}

\psarc(0.816667,0){0.0166667}{0}{180}

\psarc(0.916667,0){0.0833333}{0}{180}

\psarc(0.0714286,0){0.0714286}{0}{180}

\psarc(0.154762,0){0.0119048}{0}{180}

\psarc(0.174242,0){0.00757576}{0}{180}

\psarc(0.190909,0){0.00909091}{0}{180}

\psarc(0.226496,0){0.0042735}{0}{180}

\psarc(0.240385,0){0.00961538}{0}{180}

\psarc(0.303846,0){0.00384615}{0}{180}

\psarc(0.320513,0){0.0128205}{0}{180}

\psarc(0.345238,0){0.0119048}{0}{180}

\psarc(0.36039,0){0.00324675}{0}{180}

\psarc(0.449495,0){0.00505051}{0}{180}

\psarc(0.477273,0){0.0227273}{0}{180}

\psarc(0.522727,0){0.0227273}{0}{180}

\psarc(0.550505,0){0.00505051}{0}{180}

\psarc(0.63961,0){0.00324675}{0}{180}

\psarc(0.654762,0){0.0119048}{0}{180}

\psarc(0.679487,0){0.0128205}{0}{180}

\psarc(0.696154,0){0.00384615}{0}{180}

\psarc(0.759615,0){0.00961538}{0}{180}

\psarc(0.773504,0){0.0042735}{0}{180}

\psarc(0.809091,0){0.00909091}{0}{180}

\psarc(0.825758,0){0.00757576}{0}{180}

\psarc(0.845238,0){0.0119048}{0}{180}

\psarc(0.928571,0){0.0714286}{0}{180}

\psarc(0.0625,0){0.0625}{0}{180}

\psarc(0.133929,0){0.00892857}{0}{180}

\psarc(0.310096,0){0.00240385}{0}{180}

\psarc(0.322917,0){0.0104167}{0}{180}

\psarc(0.458042,0){0.0034965}{0}{180}

\psarc(0.480769,0){0.0192308}{0}{180}

\psarc(0.519231,0){0.0192308}{0}{180}

\psarc(0.541958,0){0.0034965}{0}{180}

\psarc(0.677083,0){0.0104167}{0}{180}

\psarc(0.689904,0){0.00240385}{0}{180}

\psarc(0.866071,0){0.00892857}{0}{180}

\psarc(0.9375,0){0.0625}{0}{180}

\psarc(0.0555556,0){0.0555556}{0}{180}

\psarc(0.118056,0){0.00694444}{0}{180}

\psarc(0.464103,0){0.0025641}{0}{180}

\psarc(0.483333,0){0.0166667}{0}{180}

\psarc(0.516667,0){0.0166667}{0}{180}

\psarc(0.535897,0){0.0025641}{0}{180}

\psarc(0.881944,0){0.00694444}{0}{180}

\psarc(0.944444,0){0.0555556}{0}{180}

\psarc(0.05,0){0.05}{0}{180}

\psarc(0.105556,0){0.00555556}{0}{180}

\psarc(0.468627,0){0.00196078}{0}{180}

\psarc(0.485294,0){0.0147059}{0}{180}

\psarc(0.514706,0){0.0147059}{0}{180}

\psarc(0.531373,0){0.00196078}{0}{180}

\psarc(0.894444,0){0.00555556}{0}{180}

\psarc(0.95,0){0.05}{0}{180}

\psarc(0.0454545,0){0.0454545}{0}{180}

\psarc(0.0954545,0){0.00454545}{0}{180}

\psarc(0.472136,0){0.00154799}{0}{180}

\psarc(0.486842,0){0.0131579}{0}{180}

\psarc(0.513158,0){0.0131579}{0}{180}

\psarc(0.527864,0){0.00154799}{0}{180}

\psarc(0.904545,0){0.00454545}{0}{180}

\psarc(0.954545,0){0.0454545}{0}{180}

\psarc(0.0416667,0){0.0416667}{0}{180}

\psarc(0.0871212,0){0.00378788}{0}{180}

\psarc(0.474937,0){0.00125313}{0}{180}

\psarc(0.488095,0){0.0119048}{0}{180}

\psarc(0.511905,0){0.0119048}{0}{180}

\psarc(0.525063,0){0.00125313}{0}{180}

\psarc(0.912879,0){0.00378788}{0}{180}

\psarc(0.958333,0){0.0416667}{0}{180}

\psarc(0.0384615,0){0.0384615}{0}{180}

\psarc(0.0801282,0){0.00320513}{0}{180}

\psarc(0.919872,0){0.00320513}{0}{180}

\psarc(0.961538,0){0.0384615}{0}{180}

\psarc(0.0357143,0){0.0357143}{0}{180}

\psarc(0.0741758,0){0.00274725}{0}{180}

\psarc(0.925824,0){0.00274725}{0}{180}

\psarc(0.964286,0){0.0357143}{0}{180}

\psarc(0.0333333,0){0.0333333}{0}{180}

\psarc(0.0690476,0){0.00238095}{0}{180}

\psarc(0.930952,0){0.00238095}{0}{180}

\psarc(0.966667,0){0.0333333}{0}{180}

\psarc(0.03125,0){0.03125}{0}{180}

\psarc(0.0645833,0){0.00208333}{0}{180}

\psarc(0.935417,0){0.00208333}{0}{180}

\psarc(0.96875,0){0.03125}{0}{180}

\psarc(0.0294118,0){0.0294118}{0}{180}

\psarc(0.0606618,0){0.00183824}{0}{180}

\psarc(0.939338,0){0.00183824}{0}{180}

\psarc(0.970588,0){0.0294118}{0}{180}

\psarc(0.0277778,0){0.0277778}{0}{180}

\psarc(0.0571895,0){0.00163399}{0}{180}

\psarc(0.94281,0){0.00163399}{0}{180}

\psarc(0.972222,0){0.0277778}{0}{180}

\psarc(0.0263158,0){0.0263158}{0}{180}

\psarc(0.0540936,0){0.00146199}{0}{180}

\psarc(0.945906,0){0.00146199}{0}{180}

\psarc(0.973684,0){0.0263158}{0}{180}

\psarc(0.025,0){0.025}{0}{180}

\psarc(0.0513158,0){0.00131579}{0}{180}

\psarc(0.948684,0){0.00131579}{0}{180}

\psarc(0.975,0){0.025}{0}{180}

\psarc(0.0238095,0){0.0238095}{0}{180}

\psarc(0.0488095,0){0.00119048}{0}{180}

\psarc(0.95119,0){0.00119048}{0}{180}

\psarc(0.97619,0){0.0238095}{0}{180}

\psarc(0.0227273,0){0.0227273}{0}{180}

\psarc(0.0465368,0){0.00108225}{0}{180}

\psarc(0.953463,0){0.00108225}{0}{180}

\psarc(0.977273,0){0.0227273}{0}{180}

\psarc(0.0217391,0){0.0217391}{0}{180}

\psarc(0.0444664,0){0.000988142}{0}{180}

\psarc(0.955534,0){0.000988142}{0}{180}

\psarc(0.978261,0){0.0217391}{0}{180}

\psarc(0.0208333,0){0.0208333}{0}{180}

\psarc(0.0425725,0){0.000905797}{0}{180}

\psarc(0.957428,0){0.000905797}{0}{180}

\psarc(0.979167,0){0.0208333}{0}{180}

\psarc(0.02,0){0.02}{0}{180}

\psarc(0.0408333,0){0.000833333}{0}{180}

\psarc(0.959167,0){0.000833333}{0}{180}

\psarc(0.98,0){0.02}{0}{180}

\psarc(0.0192308,0){0.0192308}{0}{180}

\psarc(0.0392308,0){0.000769231}{0}{180}

\psarc(0.960769,0){0.000769231}{0}{180}

\psarc(0.980769,0){0.0192308}{0}{180}

\psarc(0.0185185,0){0.0185185}{0}{180}

\psarc(0.0377493,0){0.000712251}{0}{180}

\psarc(0.962251,0){0.000712251}{0}{180}

\psarc(0.981481,0){0.0185185}{0}{180}

\psarc(0.0178571,0){0.0178571}{0}{180}

\psarc(0.0363757,0){0.000661376}{0}{180}

\psarc(0.963624,0){0.000661376}{0}{180}

\psarc(0.982143,0){0.0178571}{0}{180}

\psarc(0.0172414,0){0.0172414}{0}{180}

\psarc(0.0350985,0){0.000615764}{0}{180}

\psarc(0.964901,0){0.000615764}{0}{180}

\psarc(0.982759,0){0.0172414}{0}{180}

\psarc(0.0166667,0){0.0166667}{0}{180}

\psarc(0.033908,0){0.000574713}{0}{180}

\psarc(0.966092,0){0.000574713}{0}{180}

\psarc(0.983333,0){0.0166667}{0}{180}

\psarc(0.016129,0){0.016129}{0}{180}

\psarc(0.0327957,0){0.000537634}{0}{180}

\psarc(0.967204,0){0.000537634}{0}{180}

\psarc(0.983871,0){0.016129}{0}{180}

\psarc(0.015625,0){0.015625}{0}{180}

\psarc(0.031754,0){0.000504032}{0}{180}

\psarc(0.968246,0){0.000504032}{0}{180}

\psarc(0.984375,0){0.015625}{0}{180}

\psarc(0.0151515,0){0.0151515}{0}{180}

\psarc(0.0307765,0){0.000473485}{0}{180}

\psarc(0.969223,0){0.000473485}{0}{180}

\psarc(0.984848,0){0.0151515}{0}{180}

\psarc(0.0147059,0){0.0147059}{0}{180}

\psarc(0.0298574,0){0.000445633}{0}{180}

\psarc(0.970143,0){0.000445633}{0}{180}

\psarc(0.985294,0){0.0147059}{0}{180}

\psarc(0.0142857,0){0.0142857}{0}{180}

\psarc(0.0289916,0){0.000420168}{0}{180}

\psarc(0.971008,0){0.000420168}{0}{180}

\psarc(0.985714,0){0.0142857}{0}{180}

\psarc(0.0138889,0){0.0138889}{0}{180}

\psarc(0.0281746,0){0.000396825}{0}{180}

\psarc(0.971825,0){0.000396825}{0}{180}

\psarc(0.986111,0){0.0138889}{0}{180}

\psarc(0.0135135,0){0.0135135}{0}{180}

\psarc(0.0274024,0){0.000375375}{0}{180}

\psarc(0.972598,0){0.000375375}{0}{180}

\psarc(0.986486,0){0.0135135}{0}{180}

\psarc(0.0131579,0){0.0131579}{0}{180}

\psarc(0.0266714,0){0.000355619}{0}{180}

\psarc(0.973329,0){0.000355619}{0}{180}

\psarc(0.986842,0){0.0131579}{0}{180}

\psarc(0.0128205,0){0.0128205}{0}{180}

\psarc(0.0259784,0){0.000337382}{0}{180}

\psarc(0.974022,0){0.000337382}{0}{180}

\psarc(0.987179,0){0.0128205}{0}{180}

\psarc(0.0125,0){0.0125}{0}{180}

\psarc(0.0253205,0){0.000320513}{0}{180}

\psarc(0.974679,0){0.000320513}{0}{180}

\psarc(0.9875,0){0.0125}{0}{180}

\psarc[fillstyle=solid,fillcolor=red](0.3875,0){0.0125}{0}{180}

\psarc[fillstyle=solid,fillcolor=red](0.6125,0){0.0125}{0}{180}

\psarc[fillstyle=solid,fillcolor=red](0.211111,0){0.0111111}{0}{180}

\psarc[fillstyle=solid,fillcolor=red](0.261364,0){0.0113636}{0}{180}

\psarc[fillstyle=solid,fillcolor=red](0.279221,0){0.00649351}{0}{180}

\psarc[fillstyle=solid,fillcolor=red](0.292857,0){0.00714286}{0}{180}

\psarc[fillstyle=solid,fillcolor=red](0.369318,0){0.00568182}{0}{180}

\psarc[fillstyle=solid,fillcolor=red](0.408333,0){0.00833333}{0}{180}

\psarc[fillstyle=solid,fillcolor=red](0.422619,0){0.00595238}{0}{180}

\psarc[fillstyle=solid,fillcolor=red](0.436508,0){0.00793651}{0}{180}

\psarc[fillstyle=solid,fillcolor=red](0.563492,0){0.00793651}{0}{180}

\psarc[fillstyle=solid,fillcolor=red](0.577381,0){0.00595238}{0}{180}

\psarc[fillstyle=solid,fillcolor=red](0.591667,0){0.00833333}{0}{180}

\psarc[fillstyle=solid,fillcolor=red](0.630682,0){0.00568182}{0}{180}

\psarc[fillstyle=solid,fillcolor=red](0.707143,0){0.00714286}{0}{180}

\psarc[fillstyle=solid,fillcolor=red](0.720779,0){0.00649351}{0}{180}

\psarc[fillstyle=solid,fillcolor=red](0.738636,0){0.0113636}{0}{180}

\psarc[fillstyle=solid,fillcolor=red](0.788889,0){0.0111111}{0}{180}

\psarc[fillstyle=solid,fillcolor=red](0.154762,0){0.0119048}{0}{180}

\psarc[fillstyle=solid,fillcolor=red](0.174242,0){0.00757576}{0}{180}

\psarc[fillstyle=solid,fillcolor=red](0.190909,0){0.00909091}{0}{180}

\psarc[fillstyle=solid,fillcolor=red](0.226496,0){0.0042735}{0}{180}

\psarc[fillstyle=solid,fillcolor=red](0.240385,0){0.00961538}{0}{180}

\psarc[fillstyle=solid,fillcolor=red](0.303846,0){0.00384615}{0}{180}

\psarc[fillstyle=solid,fillcolor=red](0.345238,0){0.0119048}{0}{180}

\psarc[fillstyle=solid,fillcolor=red](0.36039,0){0.00324675}{0}{180}

\psarc[fillstyle=solid,fillcolor=red](0.449495,0){0.00505051}{0}{180}

\psarc[fillstyle=solid,fillcolor=red](0.550505,0){0.00505051}{0}{180}

\psarc[fillstyle=solid,fillcolor=red](0.63961,0){0.00324675}{0}{180}

\psarc[fillstyle=solid,fillcolor=red](0.654762,0){0.0119048}{0}{180}

\psarc[fillstyle=solid,fillcolor=red](0.696154,0){0.00384615}{0}{180}

\psarc[fillstyle=solid,fillcolor=red](0.759615,0){0.00961538}{0}{180}

\psarc[fillstyle=solid,fillcolor=red](0.773504,0){0.0042735}{0}{180}

\psarc[fillstyle=solid,fillcolor=red](0.809091,0){0.00909091}{0}{180}

\psarc[fillstyle=solid,fillcolor=red](0.825758,0){0.00757576}{0}{180}

\psarc[fillstyle=solid,fillcolor=red](0.845238,0){0.0119048}{0}{180}

\psarc[fillstyle=solid,fillcolor=red](0.133929,0){0.00892857}{0}{180}

\psarc[fillstyle=solid,fillcolor=red](0.310096,0){0.00240385}{0}{180}

\psarc[fillstyle=solid,fillcolor=red](0.322917,0){0.0104167}{0}{180}

\psarc[fillstyle=solid,fillcolor=red](0.458042,0){0.0034965}{0}{180}

\psarc[fillstyle=solid,fillcolor=red](0.541958,0){0.0034965}{0}{180}

\psarc[fillstyle=solid,fillcolor=red](0.677083,0){0.0104167}{0}{180}

\psarc[fillstyle=solid,fillcolor=red](0.689904,0){0.00240385}{0}{180}

\psarc[fillstyle=solid,fillcolor=red](0.866071,0){0.00892857}{0}{180}

\psarc[fillstyle=solid,fillcolor=red](0.118056,0){0.00694444}{0}{180}

\psarc[fillstyle=solid,fillcolor=red](0.464103,0){0.0025641}{0}{180}

\psarc[fillstyle=solid,fillcolor=red](0.535897,0){0.0025641}{0}{180}

\psarc[fillstyle=solid,fillcolor=red](0.881944,0){0.00694444}{0}{180}

\psarc[fillstyle=solid,fillcolor=red](0.105556,0){0.00555556}{0}{180}

\psarc[fillstyle=solid,fillcolor=red](0.468627,0){0.00196078}{0}{180}

\psarc[fillstyle=solid,fillcolor=red](0.531373,0){0.00196078}{0}{180}

\psarc[fillstyle=solid,fillcolor=red](0.894444,0){0.00555556}{0}{180}

\psarc[fillstyle=solid,fillcolor=red](0.0954545,0){0.00454545}{0}{180}

\psarc[fillstyle=solid,fillcolor=red](0.472136,0){0.00154799}{0}{180}

\psarc[fillstyle=solid,fillcolor=red](0.527864,0){0.00154799}{0}{180}

\psarc[fillstyle=solid,fillcolor=red](0.904545,0){0.00454545}{0}{180}

\psarc[fillstyle=solid,fillcolor=red](0.0871212,0){0.00378788}{0}{180}

\psarc[fillstyle=solid,fillcolor=red](0.474937,0){0.00125313}{0}{180}

\psarc[fillstyle=solid,fillcolor=red](0.488095,0){0.0119048}{0}{180}

\psarc[fillstyle=solid,fillcolor=red](0.511905,0){0.0119048}{0}{180}

\psarc[fillstyle=solid,fillcolor=red](0.525063,0){0.00125313}{0}{180}

\psarc[fillstyle=solid,fillcolor=red](0.912879,0){0.00378788}{0}{180}

\psarc[fillstyle=solid,fillcolor=red](0.0801282,0){0.00320513}{0}{180}

\psarc[fillstyle=solid,fillcolor=red](0.919872,0){0.00320513}{0}{180}

\psarc[fillstyle=solid,fillcolor=red](0.0741758,0){0.00274725}{0}{180}

\psarc[fillstyle=solid,fillcolor=red](0.925824,0){0.00274725}{0}{180}

\psarc[fillstyle=solid,fillcolor=red](0.0690476,0){0.00238095}{0}{180}

\psarc[fillstyle=solid,fillcolor=red](0.930952,0){0.00238095}{0}{180}

\psarc[fillstyle=solid,fillcolor=red](0.0645833,0){0.00208333}{0}{180}

\psarc[fillstyle=solid,fillcolor=red](0.935417,0){0.00208333}{0}{180}

\psarc[fillstyle=solid,fillcolor=red](0.0606618,0){0.00183824}{0}{180}

\psarc[fillstyle=solid,fillcolor=red](0.939338,0){0.00183824}{0}{180}

\psarc[fillstyle=solid,fillcolor=red](0.0571895,0){0.00163399}{0}{180}

\psarc[fillstyle=solid,fillcolor=red](0.94281,0){0.00163399}{0}{180}

\psarc[fillstyle=solid,fillcolor=red](0.0540936,0){0.00146199}{0}{180}

\psarc[fillstyle=solid,fillcolor=red](0.945906,0){0.00146199}{0}{180}

\psarc[fillstyle=solid,fillcolor=red](0.0513158,0){0.00131579}{0}{180}

\psarc[fillstyle=solid,fillcolor=red](0.948684,0){0.00131579}{0}{180}

\psarc[fillstyle=solid,fillcolor=red](0.0488095,0){0.00119048}{0}{180}

\psarc[fillstyle=solid,fillcolor=red](0.95119,0){0.00119048}{0}{180}

\psarc[fillstyle=solid,fillcolor=red](0.0465368,0){0.00108225}{0}{180}

\psarc[fillstyle=solid,fillcolor=red](0.953463,0){0.00108225}{0}{180}

\psarc[fillstyle=solid,fillcolor=red](0.0444664,0){0.000988142}{0}{180}

\psarc[fillstyle=solid,fillcolor=red](0.955534,0){0.000988142}{0}{180}

\psarc[fillstyle=solid,fillcolor=red](0.0425725,0){0.000905797}{0}{180}

\psarc[fillstyle=solid,fillcolor=red](0.957428,0){0.000905797}{0}{180}

\psarc[fillstyle=solid,fillcolor=red](0.0408333,0){0.000833333}{0}{180}

\psarc[fillstyle=solid,fillcolor=red](0.959167,0){0.000833333}{0}{180}

\psarc[fillstyle=solid,fillcolor=red](0.0392308,0){0.000769231}{0}{180}

\psarc[fillstyle=solid,fillcolor=red](0.960769,0){0.000769231}{0}{180}

\psarc[fillstyle=solid,fillcolor=red](0.0377493,0){0.000712251}{0}{180}

\psarc[fillstyle=solid,fillcolor=red](0.962251,0){0.000712251}{0}{180}

\psarc[fillstyle=solid,fillcolor=red](0.0363757,0){0.000661376}{0}{180}

\psarc[fillstyle=solid,fillcolor=red](0.963624,0){0.000661376}{0}{180}

\psarc[fillstyle=solid,fillcolor=red](0.0350985,0){0.000615764}{0}{180}

\psarc[fillstyle=solid,fillcolor=red](0.964901,0){0.000615764}{0}{180}

\psarc[fillstyle=solid,fillcolor=red](0.033908,0){0.000574713}{0}{180}

\psarc[fillstyle=solid,fillcolor=red](0.966092,0){0.000574713}{0}{180}

\psarc[fillstyle=solid,fillcolor=red](0.0327957,0){0.000537634}{0}{180}

\psarc[fillstyle=solid,fillcolor=red](0.967204,0){0.000537634}{0}{180}

\psarc[fillstyle=solid,fillcolor=red](0.031754,0){0.000504032}{0}{180}

\psarc[fillstyle=solid,fillcolor=red](0.968246,0){0.000504032}{0}{180}

\psarc[fillstyle=solid,fillcolor=red](0.0307765,0){0.000473485}{0}{180}

\psarc[fillstyle=solid,fillcolor=red](0.969223,0){0.000473485}{0}{180}

\psarc[fillstyle=solid,fillcolor=red](0.0298574,0){0.000445633}{0}{180}

\psarc[fillstyle=solid,fillcolor=red](0.970143,0){0.000445633}{0}{180}

\psarc[fillstyle=solid,fillcolor=red](0.0289916,0){0.000420168}{0}{180}

\psarc[fillstyle=solid,fillcolor=red](0.971008,0){0.000420168}{0}{180}

\psarc[fillstyle=solid,fillcolor=red](0.0281746,0){0.000396825}{0}{180}

\psarc[fillstyle=solid,fillcolor=red](0.971825,0){0.000396825}{0}{180}

\psarc[fillstyle=solid,fillcolor=red](0.0274024,0){0.000375375}{0}{180}

\psarc[fillstyle=solid,fillcolor=red](0.972598,0){0.000375375}{0}{180}

\psarc[fillstyle=solid,fillcolor=red](0.0266714,0){0.000355619}{0}{180}

\psarc[fillstyle=solid,fillcolor=red](0.973329,0){0.000355619}{0}{180}

\psarc[fillstyle=solid,fillcolor=red](0.0259784,0){0.000337382}{0}{180}

\psarc[fillstyle=solid,fillcolor=red](0.974022,0){0.000337382}{0}{180}

\psarc[fillstyle=solid,fillcolor=red](0.0125,0){0.0125}{0}{180}

\psarc[fillstyle=solid,fillcolor=red](0.0253205,0){0.000320513}{0}{180}

\psarc[fillstyle=solid,fillcolor=red](0.974679,0){0.000320513}{0}{180}

\psarc[fillstyle=solid,fillcolor=red](0.9875,0){0.0125}{0}{180}
\end{pspicture}}

\pspolygon[linecolor=white,fillstyle=solid,fillcolor=white](-0.7,-0.1)(-0.7,0.6)(-1.1,0.6)(-1.1,-0.1)(-0.7,-0.1)
\pspolygon[linecolor=white,fillstyle=solid,fillcolor=white](1.9,-0.1)(1.9,0.6)(2.1,0.6)(2.1,-0.1)(1.9,-0.1)

\rput(0,-0.07){$_{0}$}
\rput(1,-0.07){$_{1}$}

\rput(1.41421,-0.07){$_{\sqrt{2}}$}

\rput(-0.06,0.6){$\mathbb I$}

\psset{linewidth=0.9pt}
\psset{dotsize=3pt 0,dotstyle=*}

\psarc[linecolor=orange](0.68,0){0.734214}{0}{157.844}
\psdot[linecolor=orange](0,0.276893)

\newrgbcolor{red}{1 0 0}

\rput(0.6,0.66){$\red{\mathcal L}$}
\rput(1.12,0.51){$\red{\mathcal R}$}
\rput(1.27,0.33){$\red{_{\mathcal R}}$}
\rput(1.36,0.14){$\red{_{\mathcal L}}$}

\rput(0.31,0.71){$\psi$}

\psset{arrowscale=1.9,arrowinset=0.3,arrowlength=1.1}
\psline[linecolor=orange]{->}(0.369708, 0.665424)(0.375526, 0.668106)

\end{pspicture}
\caption{The Farey tiling of $\H$ and a cutting sequence for $\sqrt{2}$}
\label{far}
\end{figure}
These Farey edges  do not contain their endpoints, do not intersect one another,  and have a natural partial ordering by containment of their associated closed semidiscs or  right half planes if vertical.  
We may also say that an edge  contains a set in $\C$ if its associated closed semidisc or  right half plane contains it.

Let $z=x+i y$ have $x>0$, $y\gqs 0$. Consider the oriented geodesic $\psi$ from some $w\in \mathbb I$ to $z$. Then the  Farey edges that  intersect $\psi$ are exactly  the edges contained by  $\mathbb I$  that contain $z$. Hence this set  is independent of $w$. Also $\psi$ must meet them in decreasing order when going from $w$ to $z$. Suppose we have two consecutive Farey edge intersection points on $\psi$, making a segment. This segment of $\psi$ divides a Farey triangle in two, with one corner on one side and two on the other. In the direction we are moving along $\psi$, we label the segment $\mathcal L$ if there is only one  corner to the left, or  $\mathcal R$ if there is only one  corner to the right.

So for each $z=x+i y$ with $x>0$, $y\gqs 0$, by including all such segments from $w$ to $z$ we obtain
 the {\em cutting sequence of $z$}:
\begin{equation} \la{cut}
  \mathcal L^{c_0}\mathcal R^{c_1}\mathcal L^{c_2}\mathcal R^{c_3} \cdots ,
\end{equation}
which may be finite or infinite, with positive integers $c_j$, (and $c_0=0$ is possible). Note that if $z$ is in the interior of a Farey triangle then the segment of $\psi$ inside this triangle does not contribute anything to the cutting sequence. We also define the cutting sequence of $z \in \mathbb I$ to be $ \mathcal L^{0}$. Our definition extends the one of Series in \cite[p. 24]{ser} for real numbers. For geodesics $\psi$ in $\H$ that reach their endpoints on $\R$,  their cutting sequences may be extended to the left. This allows a characterization of geodesics up to $\SL(2,\Z)$ equivalence as described in \cite{serb}. See also the wide-ranging survey \cite{kat}.

Figure \ref{far} shows the cutting sequence for $z=\sqrt{2}$. Peering closely, it begins $\mathcal L^{1}\mathcal R^{2}\mathcal L^{2}\mathcal R^{2} \cdots$.
These  sequences for real numbers go back to earlier authors with  Hurwitz in \cite[Sect. 5]{hur94}, for example, giving the following formulation. Start with the Farey fractions $0/1$ and $1/0$. To find rational approximations to $x>0$, repeatedly take the mediant and choose the subinterval containing $x$. Use $\mathcal L$ and $\mathcal R$ to record if this subinterval was  the one on the left or the right when viewed from $\mathbb I$. For our extension to $z \in \H$, choose  the subinterval whose corresponding Farey edge contains $z$.

\SpecialCoor
\psset{griddots=5,subgriddiv=0,gridlabels=0pt}
\psset{xunit=2.5cm, yunit=2.5cm, runit=2.5cm}
\psset{linewidth=1pt}
\psset{dotsize=7pt 0,dotstyle=*}
\begin{figure}[ht]
\centering
\begin{pspicture}(-0.7,-0.2)(5.25,1.85) 

\psset{arrowscale=1.4,arrowinset=0.3,arrowlength=1.1}
\newrgbcolor{light}{0.8 0.8 1.0}
\newrgbcolor{pale}{1 0.7 1}
\newrgbcolor{pale}{1 0.7 0.4}
\newrgbcolor{pale}{0.9179 0.7539 0.5781}
\newrgbcolor{pale}{0.996094, 0.71875, 0.617188}
\newrgbcolor{greendark}{0.1 0.1 1.0}
\newrgbcolor{greenlight}{0.8 0.8 1.0} 

\newrgbcolor{cx}{0.8 0.8 1.0}


\rput(0.5,0.8){
\begin{pspicture}(-1,-0.1)(2,1.7) 

\psset{linecolor=greendark}

\pspolygon[linecolor=greenlight,fillstyle=solid,fillcolor=greenlight](-0.5,1.7)(-0.5,0)(0.5,0)(0.5,1.7)(-0.5,1.7)

\pswedge[linecolor=white,fillstyle=solid,fillcolor=white](-1,0){1}{0}{90}
\pswedge[linecolor=white,fillstyle=solid,fillcolor=white](1,0){1}{90}{180}

\psarc(-1,0){1}{0}{60}
\psarc(1,0){1}{120}{180}

\psarc(0,0){1}{0}{60}
\psarc(2,0){1}{120}{180}

\psarc(0.3333,0){0.3333}{60}{180}
\psarc(0.6666,0){0.3333}{0}{120}

\psline(-0.5,1.7)(-0.5,0.866025)
\psline(0.5,1.7)(0.5,0.866025)

\psline(1.5,1.7)(1.5,0.866025)

\psset{linecolor=black}

\psline(0,1.7)(0,0)
\psline(1,1.7)(1,0)
\psarc(0.5,0){0.5}{0}{180}

\psline[linecolor=gray](-0.7,0)(1.7,0)
\psline[linecolor=gray](-0.5,0.04)(-0.5,-0.04)
\psline[linecolor=gray](0.5,0.04)(0.5,-0.04)
\psline[linecolor=gray](1.5,0.04)(1.5,-0.04)


\psline{->}(2.,0.5)(2.5,0.5)
\rput(2.25,0.63){$M$}

\rput(-0.65,1){$\mathcal Y$}
\rput(1.7,1){$L\mathcal Y$}
\rput(0.6,0.18){$R\mathcal Y$}

\rput(-0.12,1.4){$\mathbb I$}

\rput(0,-0.12){$_0$}
\rput(1,-0.12){$_1$}

\end{pspicture}}

\rput(4,0.8){
\begin{pspicture}(-1.1,-0.1)(1.31,1.7) 

\psset{linecolor=greendark}

\psarc(1.68831,0){2.5974}{139.653}{180}
\psarc(-0.598291,0){1.7094}{0}{79.6529}
\psarc(0.394265,0){0.716846}{0.}{126.609}
\psarc(-0.282132,0){0.626959}{66.6086}{180.}
\psarc(-5.,0){5.}{0.}{6.60861}
\psarc(0.263158,0){0.263158}{84.705}{180.}
\psarc(0.657596,0){0.453515}{0.}{144.705}
\psarc(-0.552585,0){0.356506}{40.9693}{180.}
\psarc(-0.238095,0){0.238095}{0.}{100.969}

\psset{linecolor=black}

\psarc(-0.454545,0){0.454545}{0}{180}
\psarc(0.55555,0){0.55555}{0}{180}
\psarc(0.10101,0){1.0101}{0}{180}

\psline[linecolor=gray](-1.1,0)(1.31,0)

\rput(-0.8,1.4){$M\mathcal Y$}
\rput(-0.55,0.14){$M R\mathcal Y$}
\rput(0.6,0.19){$M L\mathcal Y$}

\end{pspicture}}

\end{pspicture}
\caption{Images of $\mathbb I$ and of $\mathcal Y := \Fd \cup S\Fd$}
\label{funds2}
\end{figure}

\begin{lemma} \la{ser}
Let $z$ have  real part $\gqs 0$ and lie on a Farey edge. The following are true:
\begin{enumerate}
  \item We have $z \in M \mathbb I$ for some $M=L^{a_0}R^{a_1} \cdots [L \text{ or } R]^{a_m}$ with  positive integers $a_j$, ($a_0=0$ is possible).
  \item  The cutting sequence for $z$ is finite and may  be written as $\mathcal L^{c_0}\mathcal R^{c_1} \cdots [\mathcal L \text{ or } \mathcal R]^{c_n}$. 
  \item In parts (i) and (ii) we have $m=n$ and $a_j=c_j$ for all $j$.
\end{enumerate}
\end{lemma}
\begin{proof}
To show part $(i)$, start with $\mathbb I$ and its images $L \mathbb I$ and $R \mathbb I$ as displayed on the left of Figure \ref{funds2}.
These sides make a Farey triangle and to produce the adjacent triangles to the right and below, apply $L$ and $R$ again on the right. Precisely, let $M$ be a product of $L$s and $R$s with $M \mathbb I$ the edge seen on the right side of Figure \ref{funds2}. Then the other sides of the Farey triangle contained by this edge are $M L \mathbb I$ and $M R \mathbb I$. Note that since $\det M =1$, orientation is preserved so that $M L \mathbb I$ is on the left and $M R \mathbb I$ is on the right when viewed from $\mathbb I$. It can also be seen that the new vertex is the mediant of the endpoints of $M \mathbb I$. Starting with $\mathbb I$ we can produce all  Farey edges to the right in this way and they take the  form  stated in $(i)$.

Parts $(ii)$ and $(iii)$ may  be proved together by  induction on the ordering of the edges. They are true for $z \in \mathbb I$ with $L^0$ matching $\mathcal L^{0}$. Next let $z'$ be on another Farey edge $g'$. Fix a geodesic $\psi$ for its cutting sequence. Then $g'$ is on the boundary of two Farey triangles: one contained by $g'$ and a larger one. Let $g$ be the boundary  of the larger triangle, with $g$ containing $g'$. Then $g = M \mathbb I$ for $M=L^{a_0}R^{a_1} \cdots [L \text{ or } R]^{a_m}$ by part $(i)$.  Our induction hypothesis implies that $z :=g \cap \psi$ has cutting sequence $\mathcal L^{a_0}\mathcal R^{a_1} \cdots [\mathcal L \text{ or } \mathcal R]^{a_m}$. Now $g'$ equals  $M L \mathbb I$ or  $M R \mathbb I$ and these correspond to adding a factor $\mathcal L$ or $\mathcal R$, respectively, to the cutting sequence.
\end{proof}

The next lemma is a direct consequence of the previous one.
\begin{lemma} \la{ser2}
Let $z$ have  real part $\gqs 0$ and lie on a Farey edge.   Then  $\mathcal L^{c_0}\mathcal R^{c_1} \cdots [\mathcal L \text{ or } \mathcal R]^{c_n}$ is the cutting sequence for $z$ if and only if 
$
  z$ has the continued fraction $[c_0, c_1,  \dots, c_n+w_0]
$. 
We may also locate $w_0$ more precisely here: $w_0 \in \mathbb I$ if $n$ is even and   $w_0 \in \overline{\mathbb I}$ if $n$ is odd.
\end{lemma}

The connection between cutting sequences and continued fractions for  real $z$ is shown in \cite[p. 24]{ser},  and see also \cite[Sects. 2.11 -- 2.12]{sta}. For general $z$ we have the following. 

\begin{theorem} \la{maincut}
Suppose $z=x+i y$ has $x, y\gqs 0$ and cutting sequence \e{cut}. If this sequence is infinite then $z=[c_0, c_1, c_2, \dots]$ and $z$ is an irrational real. If the cutting sequence is finite with last part $c_n$ then a continued fraction for $z$ is given by one  of the following:
\begin{equation}\label{adj}
  [c_0, c_1,  \dots, c_n+w_0], \quad [c_0, c_1,   \dots, c_n+1+w_0] \quad  \text{or} \quad [c_0, c_1,   \dots, c_n, 1+w_0].
\end{equation}
In particular, if $z$ is rational then it equals the 2nd and 3rd representations in \e{adj} with $w_0=0$, (the 2nd being canonical). More succinctly,
\begin{equation*}
  z=  [c_0, c_1,   \dots, c_n+1] \quad  \text{if} \quad z \in \Q.
\end{equation*}
\end{theorem}
\begin{proof}
Fix the geodesic $\psi$ from $\mathbb I$ to $z$. Since the endpoints of  Farey edges are in $\Q \cup \{\infty\}$, it follows that the cutting sequence of $z$ is finite if and only if $z\in \Q \cup \H$. In this finite case  let $z'\in \H$ be the point where $\psi$ meets its last Farey edge.

If $z=z'$ then we know by Lemma \ref{ser} that $z=M z_0$ for $M=L^{c_0}R^{c_1} \cdots [L \text{ or } R]^{c_n}$ and $z_0 \in \mathbb I$. Hence $z=[c_0, c_1,  \dots, c_n+w_0]$ for $w_0=z_0$ or $1/z_0$. Otherwise, if $z \notin \Q$ then $z$ is inside the Farey triangle seen on the right of Figure \ref{funds2} with boundaries $M \mathbb I$, $MR \mathbb I$ and $ML \mathbb I$. This triangle is shown partitioned into three pieces. With  $\mathcal Y := \Fd \cup S\Fd$,   these pieces are its intersections with $M \mathcal Y$, $MR \mathcal Y$ and $ML \mathcal Y$. These possible extra $R$ and $L$ factors give the continued fractions in \e{adj}. If $z \in \Q$ then $z$ is the common endpoint of $MR \mathbb I$ and $ML \mathbb I$. This is obtained with the limit as $w_0 \to 0$ in the   2nd and 3rd representations in \e{adj}.

Now suppose that $z$ is an irrational  real. Its cutting sequence \e{cut} is infinite. Let $z_n$ be the point on $\psi$ where it meets the geodesic $L^{c_0}R^{c_1} \cdots [L \text{ or } R]^{c_n} \mathbb I$. By Lemma \ref{ser2} we have
\begin{equation} \la{zn}
  z_n = [c_0, c_1,  \dots, c_n+w_0]
\quad \text{for} \quad w_0 \in \Fd'. 
\end{equation}
Also
\begin{equation}\la{nn}
  \Bigl| [c_0, c_1,  \dots, c_n+w_0] - [c_0, c_1,  \dots, c_n]\Bigr| < \frac 4{\phi^{2n}}
\end{equation}
by Proposition \ref{cplx}.
From \e{zn}, \e{nn} it first follows that $\Im(z_n) \to 0$ as $n \to \infty$ and hence $z_n \to z$. Then, secondly, we obtain $[c_0, c_1,  \dots, c_n] \to z$ as  $n \to \infty$.
\end{proof}

If $z\in \H$ is directly above an irrational number $r$ then the cutting sequence for $z$ will match the initial part of the sequence for $r$. It follows from Theorem \ref{maincut} that their continued fractions will then match. For example,
\begin{gather*}
  e+e^{-30}i  = [2, 1, 2, 1, 1, 4, 1, 1, 6, 1, 1, 8, 1, 1, 10, 1, 1, 12+z_0]\quad \text{for} \quad z_0 \approx 0.04-2.44 i, \\
  \frac{I_0(2)}{I_1(2)} + 10^{-16} i  = [1, 2, 3, 4, 5, 6, 7, 8, 9, 10, 11, 1 +z_0]\quad \text{for} \quad z_0 \approx -0.40 - 2.81 i, 
\end{gather*}
matching the initial parts of the continued fractions of $e$ and the Bessel ratio $I_0(2)/I_1(2) \approx 1.433127427$.

An appealing application of cutting sequences ends this section.

\begin{prop} \la{fnb}
Let $a/b$ and $c/d$ be Farey neighbors so that $ad-bc = \pm 1$. 
Suppose
\begin{equation*}
  a/b=[a_0, a_1, \dots, a_r], \qquad c/d=[c_0, c_1, \dots, c_s] \qquad \text{with} \quad 0\lqs r\lqs s
\end{equation*}
are canonical.
Then $r\lqs s \lqs r+2$ and these continued fractions agree for at least the first $r$ places. Precisely, $c/d$ equals one of the following, for some $m \gqs 2$:
\begin{equation*}
  [a_0,  \dots, a_r-1, 2], \qquad [a_0,  \dots, a_r-1, 1, m], \qquad 
  [a_0,  \dots, a_r \pm 1], \qquad [a_0,  \dots, a_r, m].
\end{equation*}
\end{prop}

\SpecialCoor
\psset{griddots=5,subgriddiv=0,gridlabels=0pt}
\psset{xunit=0.5cm, yunit=0.5cm, runit=0.5cm}
\psset{linewidth=0.7pt}
\psset{dotsize=7pt 0,dotstyle=*}
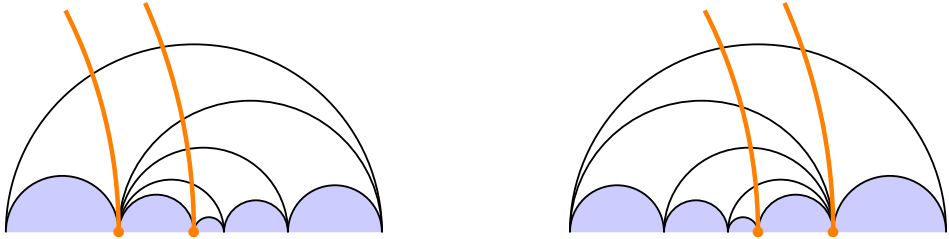
\begin{figure}[ht]
\centering

\begin{pspicture}(0,0)(25,6) 

\rput(5,3){%
\begin{pspicture}(0,0)(10,6) 

\psset{arrowscale=1.4,arrowinset=0.3,arrowlength=1.1}
\newrgbcolor{light}{0.99609375 0.71875 0.6171875}
\newrgbcolor{pale}{1 0.7 1}
\newrgbcolor{pale}{1 0.7 0.4}
\newrgbcolor{pale}{0.9179 0.7539 0.5781}
\newrgbcolor{pale}{0.996094, 0.71875, 0.617188}
\newrgbcolor{pale}{0.957031 0.726563 0.726563}

\newrgbcolor{cx}{0.8 0.8 1.0}

\psarc(5,0){5}{0}{180}
\psarc[fillstyle=solid,fillcolor=cx](1.5,0){1.5}{0}{180}
\psarc[fillstyle=solid,fillcolor=cx](4,0){1}{0}{180}
\psarc[fillstyle=solid,fillcolor=cx](5.4,0){0.4}{0}{180}
\psarc[fillstyle=solid,fillcolor=cx](6.65,0){0.85}{0}{180}
\psarc[fillstyle=solid,fillcolor=cx](8.75,0){1.25}{0}{180}

\psarc(4.4,0){1.4}{0}{180}
\psarc(5.25,0){2.25}{0}{180}
\psarc(6.5,0){3.5}{0}{180}

\psset{linewidth=1.7pt}
\psarc[linecolor=orange](-10,0){13}{0}{27}
\psarc[linecolor=orange](-10,0){15}{0}{24}

\psset{dotsize=4pt 0,dotstyle=*}

\psdot[linecolor=orange](3,0)
\psdot[linecolor=orange](5,0)
\end{pspicture}}

\rput(20,3){%
\begin{pspicture}(-10,0)(0,6) 

\psset{arrowscale=1.4,arrowinset=0.3,arrowlength=1.1}
\newrgbcolor{light}{0.99609375 0.71875 0.6171875}
\newrgbcolor{pale}{1 0.7 1}
\newrgbcolor{pale}{1 0.7 0.4}
\newrgbcolor{pale}{0.9179 0.7539 0.5781}
\newrgbcolor{pale}{0.996094, 0.71875, 0.617188}
\newrgbcolor{pale}{0.957031 0.726563 0.726563}

\newrgbcolor{cx}{0.8 0.8 1.0}

\psarc(-5,0){5}{0}{180}
\psarc[fillstyle=solid,fillcolor=cx](-1.5,0){1.5}{0}{180}
\psarc[fillstyle=solid,fillcolor=cx](-4,0){1}{0}{180}
\psarc[fillstyle=solid,fillcolor=cx](-5.4,0){0.4}{0}{180}
\psarc[fillstyle=solid,fillcolor=cx](-6.65,0){0.85}{0}{180}
\psarc[fillstyle=solid,fillcolor=cx](-8.75,0){1.25}{0}{180}

\psarc(-4.4,0){1.4}{0}{180}
\psarc(-5.25,0){2.25}{0}{180}
\psarc(-6.5,0){3.5}{0}{180}

\psset{linewidth=1.7pt}
\psarc[linecolor=orange](-18,0){13}{0}{27}
\psarc[linecolor=orange](-18,0){15}{0}{24}

\psset{dotsize=4pt 0,dotstyle=*}

\psdot[linecolor=orange](-3,0)
\psdot[linecolor=orange](-5,0)
\end{pspicture}}

\end{pspicture}
\caption{Ends of cutting sequences for Farey neighbors}
\label{cutf}
\end{figure}

\begin{proof}
By adding the same large integer to $a/b$ and $c/d$, we may assume they are both positive. Figure \ref{cutf} shows the typical ends of cutting sequence geodesics for Farey neighbors, with one passing through $m=3$ extra Farey edges. 
For the situation on the left of Figure \ref{cutf}, label the two rational endpoints as $x_1<x_2$. Suppose the cutting sequence of $x_1$ ends with $\dots \mathcal R^{b_{k-1}} \mathcal L^{b_{k}}$. Then by Theorem \ref{maincut}, the continued fraction for $x_1$ ends with $[ \dots, b_{k-1}, b_k +1]$. The cutting sequence for $x_2$  ends with 
$
  \dots \mathcal R^{b_{k-1}} \mathcal L^{b_{k}} \cdot \mathcal L \mathcal R^{m-1},
$
so that its continued fraction ends with
\begin{align*}
  [ \dots, b_{k-1}, b_k +1, m] & \qquad \text{if} \quad m \gqs 2, \\
  [ \dots, b_{k-1}, b_k +2] & \qquad \text{if} \quad m =1,
\end{align*}
by Theorem \ref{maincut} again.
The other cases are similar and the proof follows.
\end{proof}

\section{Further results and variations} \la{fur}

Suppose $z=[a_0, a_1, \dots, a_n+z_0]$. Then what is the continued fraction representation of $\overline z$? By conjugation,
\begin{equation} \la{ov}
  \overline z=[a_0, a_1, \dots, a_n+\overline{z_0}],
\end{equation}
but \e{ov} is a continued fraction for $\overline z$ only if $\overline{z_0} \in \Fd'$. This is the case for $z_0$ in the interior of $\Fd'$, or equaling $0$ or one of the corners $\rho, \rho^2, \rho^4, \rho^5$. Otherwise, $\overline{z_0}$ is in one of the dashed borders of $\Fd'$ shown in Figure \ref{funds}. For $t>\sqrt{3}/2$ these borders are: $B_1$  above $\rho$ given by $1/2+i t$,  $B_2$ between $\rho^2$ and $0$ given by $-1/(1/2+i t)$, $B_3$ between $0$ and $\rho^5$  given by  $1/(1/2+i t)$, and $B_4$  below $\rho^4$  given by  $-1/2-i t$. In these cases we must adjust the end of \e{ov} as follows:
\begin{equation} \la{tex}
\begin{aligned}
 \text{if \ $\overline{z_0} \in B_1$ \quad then}\quad \overline z & =[a_0,a_1, \dots, a_{n-1}, a_n+1+(-z_0)],\\
 \text{if \ $\overline{z_0} \in B_2$ \quad then}\quad  \overline z & =\begin{cases} 
                         [a_0, a_1, \dots, a_{n-1}, a_n-1, 1+z_0] & \quad \text{if} \quad a_n \gqs 2; \\
                         [a_0, a_1, \dots, a_{n-1}+ 1+z_0] &  \quad \text{if} \quad a_n=1,
                       \end{cases}\\
 \text{if \ $\overline{z_0} \in B_3$ \quad then}\quad   \overline z & =[a_0,a_1, \dots, a_{n-1}, a_n, 1+(-1/z_0)], \\
 \text{if \ $\overline{z_0} \in B_4$ \quad then}\quad \overline z & =\begin{cases} 
                         [a_0, a_1, \dots, a_{n-1}, a_n-1+(-z_0)] &  \quad \text{if} \quad a_n \gqs 2; \\
                         [a_0, a_1, \dots, a_{n-1} +(-1/z_0)] &  \quad \text{if} \quad a_n=1.
                       \end{cases}
\end{aligned}
\end{equation}
These were found by computing the continued fraction of $\overline{z_0}$ and are straightforward to verify.

Our continued fraction algorithm uses  the region $\Fd'$ based on the fundamental domain $\Fd$. We see next that little changes if we instead use  the region
\begin{equation*}
  \Gd' := \Gd \cup S\Gd \cup \{0\}\cup -(\Gd \cup S\Gd) 
\end{equation*}
based on the fundamental domain $\Gd$ from \e{gfun} with different boundary. 
Add the region used as a subscript in the representation, so that $[a_0, a_1, \dots, a_r+z_0]_{\Fd'}$ is what he have been studying and $[a_0, a_1, \dots, a_r+z_0]_{\Gd'}$ is the new expansion.
Since
\begin{equation*}
  \Gd'  = \Gd \cup S\Gd \cup \{0\}\cup  \overline \Fd \cup S\overline \Fd = \overline \Fd',
\end{equation*}
it may be seen that
\begin{equation} \la{gf}
  z=[a_0, a_1, \dots, a_r+z_0]_{\Gd'} \qquad \iff \qquad \overline{z} = [a_0, a_1, \dots, a_r+\overline{z_0}]_{\Fd'}.
\end{equation}
The notion of canonical is the same for $[a_0, a_1, \dots, a_r+z_0]_{\Gd'}$, namely $z_0 \neq \rho^2, \rho^4$. Theorem \ref{main} is true for the  $[\dots ]_{\Gd'}$ continued fractions; this follows directly from \e{gf} and  Theorem \ref{main}. In fact all the results in this paper also hold for $[\dots ]_{\Gd'}$ with $\Gd'$ replacing $\Fd'$. The only exception is Theorem \ref{bm2} since $z_p \in \Fd$ is required for Gauss reduction. Of course another form of reduction could be based on $z_p \in \Gd$.


\begin{prop}
Suppose $z=[a_0, a_1, \dots, a_n+z_0]_{\Fd'}$. Then to give its $\Gd'$ representation:
\begin{align*}
\text{if \ $\overline{z_0} \in \Fd'$ \quad then}\quad  z & =[a_0,a_1, \dots, a_{n-1}, a_n+z_0]_{\Gd'},\\
 \text{if \ $\overline{z_0} \in B_1$ \quad then}\quad  z & =[a_0,a_1, \dots, a_{n-1}, a_n+1+(-\overline{z_0})]_{\Gd'},\\
 \text{if \ $\overline{z_0} \in B_2$ \quad then}\quad   z & =\begin{cases} 
                         [a_0, a_1, \dots, a_{n-1}, a_n-1, 1+\overline{z_0}]_{\Gd'} & \quad \text{if} \quad a_n \gqs 2; \\
                         [a_0, a_1, \dots, a_{n-1}+ 1+\overline{z_0}]_{\Gd'} &  \quad \text{if} \quad a_n=1,
                       \end{cases}\\
 \text{if \ $\overline{z_0} \in B_3$ \quad then}\quad    z & =[a_0,a_1, \dots, a_{n-1}, a_n, 1+(-1/\overline{z_0})]_{\Gd'}, \\
 \text{if \ $\overline{z_0} \in B_4$ \quad then}\quad  z & =\begin{cases} 
                         [a_0, a_1, \dots, a_{n-1}, a_n-1+(-\overline{z_0})]_{\Gd'} &  \quad \text{if} \quad a_n \gqs 2; \\
                         [a_0, a_1, \dots, a_{n-1} +(-1/\overline{z_0})]_{\Gd'} &  \quad \text{if} \quad a_n=1.
                       \end{cases}
\end{align*}
\end{prop}
\begin{proof}
Combine \e{tex} and \e{gf}.
\end{proof}


\SpecialCoor
\psset{griddots=5,subgriddiv=0,gridlabels=0pt}
\psset{xunit=2cm, yunit=2cm, runit=2cm}
\psset{linewidth=1pt}
\psset{dotsize=7pt 0,dotstyle=*}
\begin{figure}[ht]
\centering
\begin{pspicture}(-0.5,-0.9)(4.5,1.1) 

\psset{arrowscale=1.4,arrowinset=0.3,arrowlength=1.1}
\newrgbcolor{light}{0.8 0.8 1.0}
\newrgbcolor{pale}{1 0.7 1}
\newrgbcolor{pale}{1 0.7 0.4}
\newrgbcolor{pale}{0.9179 0.7539 0.5781}
\newrgbcolor{pale}{0.996094, 0.71875, 0.617188}
\newrgbcolor{greendark}{0.1 0.1 1.0}
\newrgbcolor{greenlight}{0.8 0.8 1.0} 

\newrgbcolor{cx}{0.8 0.8 1.0}


\rput(0.5,0.){
\begin{pspicture}(-0.5,-1)(1.5,1) 

\psset{linecolor=greendark}

\pswedge[linecolor=greenlight,fillstyle=solid,fillcolor=greenlight](0,0){1}{-60}{60}
\pswedge[linecolor=greenlight,fillstyle=solid,fillcolor=greenlight](1,0){1}{120}{240}

\psarc(0,0){1}{0}{60}
\psarc(1,0){1}{180}{240}

\psarc[linestyle=dashed](0,0){1}{-60}{0}
\psarc[linestyle=dashed](1,0){1}{120}{180}

\pscircle[linecolor=black,fillstyle=solid,fillcolor=white](0.5,0.866025){0.04}
\pscircle[linecolor=black,fillstyle=solid,fillcolor=white](0.5,-0.866025){0.04}
\pscircle[linecolor=black,fillstyle=solid,fillcolor=white](0.,0){0.04}
\pscircle[linecolor=black,fillstyle=solid,fillcolor=white](1,0){0.04}

\psset{linecolor=black}


\psline[linecolor=gray](-0.5,0)(1.5,0)


\pscircle[linecolor=black,fillstyle=solid,fillcolor=white](0.5,0.866025){0.04}
\pscircle[linecolor=black,fillstyle=solid,fillcolor=white](0.5,-0.866025){0.04}
\pscircle[linecolor=black,fillstyle=solid,fillcolor=white](0.,0){0.04}
\pscircle[linecolor=black,fillstyle=solid,fillcolor=white](1,0){0.04}

\rput(-0.25,0.4){$\mathcal C$}

\rput(-0.1,-0.12){$_0$}
\rput(1.1,-0.12){$_1$}

\end{pspicture}}

\rput(3.5,0.){
\begin{pspicture}(-0.5,-1)(1.5,1) 

\psset{linecolor=greendark}

\pscircle[linestyle=dashed,fillstyle=solid,fillcolor=greenlight](0.5,0.){0.51}

\psset{linecolor=black}


\psline[linecolor=gray](-0.5,0)(1.5,0)
\psline[linecolor=gray](0,0.04)(0,-0.04)
\psline[linecolor=gray](1,0.04)(1,-0.04)


\rput(-0.25,0.4){$\mathcal D$}

\rput(-0.1,-0.12){$_0$}
\rput(1.1,-0.12){$_1$}

\end{pspicture}}

\end{pspicture}
\caption{The regions $\mathcal C$ and $\mathcal D$}
\label{simple}
\end{figure}
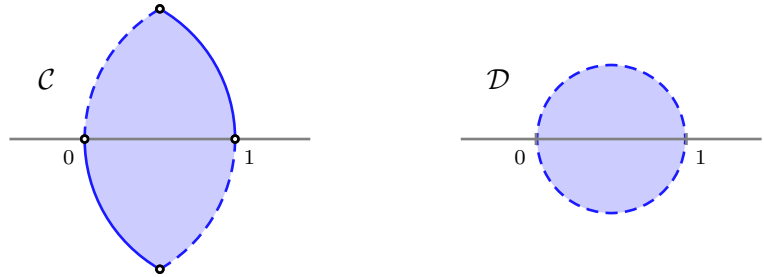

Returning to  $[\dots ]_{\Fd'}$, we finish with two equivalent versions of our continued fraction algorithm from the introduction. They simplify the halting criterion at the expense of a slightly more complicated description of the output. Let 
\begin{align*}
  \mathcal C & := \big\{ z  \in \C \, : \, |z|\lqs 1, |z-1|\lqs 1, \lnot( |z|= 1, \Im(z) \lqs 0), \lnot( |z-1|= 1, \Im(z) \gqs 0) \big\},\\
  \mathcal D & := \{ z\in \C  \, : \, |z-1/2| < 1/2\},
\end{align*}
as shown in Figure \ref{simple}.

\newpage

\begin{algo2}
Start with index $i=0$.
\begin{itemize}
  \item Let $a_i=m=\lfloor \Re(z)\rfloor$. If $z-m \in \mathcal C$  then replace $z$ by $1/(z-m)$, increment $i$ and repeat this step. 
      If $z-m \notin \mathcal C$, let $w=z-m$ and finish.
\end{itemize}
If the algorithm doesn't finish then the output is $(a_0,a_1, \dots )$. Otherwise let $n$ denote the last $i$, and the output is $(a_0,a_1, \dots, a_n)$ along with $z_0=w$ if $w\in \Fd'$ or  $(a_0,a_1, \dots, a_n+1)$ along with $z_0=w-1$ if $w \notin \Fd'$. 
\end{algo2}

The region $\mathcal C$ lies between $\Fd'$ and $1+\Fd'$, as seen in Figure \ref{funds}. So it is straightforward to see that Version 2 is equivalent to Version 1, giving a canonical output.

\begin{algo3}
Start with index $i=0$.
\begin{itemize}
  \item Let $a_i=m=\lfloor \Re(z)\rfloor$. If $z-m \in \mathcal D$  then replace $z$ by $1/(z-m)$, increment $i$ and repeat this step. 
      If $z-m \notin \mathcal D$, let $w=z-m$ and finish.
\end{itemize}
If the algorithm doesn't finish then the output is $(a_0,a_1, \dots )$. Otherwise let $n$ denote the last $i$, and the output has three cases:
\begin{align*}
  (a_0,a_1, \dots, a_n) & \quad \text{with \quad $z_0=w$ \quad if $w \in \Fd'$} \\
  (a_0,a_1, \dots, a_n+1) & \quad \text{with \quad $z_0=w-1$ \quad if $w \notin \Fd'$ and $w-1 \in \Fd'$} \\
  (a_0,a_1, \dots, a_n,1) & \quad \text{with \quad $z_0=1/w-1$ \quad otherwise.}
\end{align*}
\end{algo3}

The loop in Version 3 is equivalent to the algorithm studied in \cite{DFV}; see Section \ref{latt}.
Version 3 finishes earlier than Version 2 when $w \in \mathcal C - \mathcal D$. This region is contained in $R \Fd'$, as can be seen on the left of Figure \ref{funds2}, so we may write
\begin{equation*}
  w=R(1/z_0) =  1/(1+z_0) \quad \text{for} \quad z_0 \in \Fd'
\end{equation*}
in this case. Hence Version 3 is equivalent to Version 2.

{\small \bibliography{cfbib} }

{\small 
\vskip 5mm
\noindent
\textsc{Department of Mathematics, The CUNY Graduate Center, 365 Fifth Avenue, New York, NY 10016-4309, U.S.A.}

\noindent
{\em E-mail address:} \texttt{cosullivan@gc.cuny.edu}
}

\end{document}